\documentclass[a4paper,10pt]{amsart}

\usepackage{graphicx}

\usepackage[T1]{fontenc} \usepackage[utf8]{inputenc}
\usepackage{lmodern}

\usepackage{enumerate}

\usepackage{epsfig} \input{xy} \xyoption{all}

\usepackage{amsfonts,amsmath,amstext,amsbsy,amssymb}
\usepackage{amsbsy,amsopn,amsthm,amscd,amsxtra}

\usepackage{latexsym,mathrsfs}

\usepackage{mathabx} \newcommand{\carr}{\righttoleftarrow}

\usepackage{hyperref}
\RequirePackage[dvipsnames]{xcolor} % [dvipsnames]
\definecolor{halfgray}{gray}{0.55} 
\definecolor{webgreen}{rgb}{0,0.5,0}
\definecolor{webbrown}{rgb}{.6,0,0} \hypersetup{%
  colorlinks=true, linktocpage=true, pdfstartpage=3,
  pdfstartview=FitV,%
  breaklinks=true, pdfpagemode=UseNone, pageanchor=true,
  pdfpagemode=UseOutlines,%
  plainpages=false, bookmarksnumbered, bookmarksopen=true,
  bookmarksopenlevel=1,%
  hypertexnames=true,
  pdfhighlight=/O,%hyperfootnotes=true,%nesting=true,%frenchlinks,%
  urlcolor=webbrown, linkcolor=RoyalBlue,
  citecolor=webgreen, %pagecolor=RoyalBlue,%
  pdftitle={Minimal homeomorphisms},%
  pdfauthor={A. Kocsard},%
  pdfsubject={2000 MAthematical Subject Classification: Primary:},%
  pdfkeywords={},%
  pdfcreator={pdfLaTeX},%
  pdfproducer={LaTeX with hyperref}%
}

\usepackage{lipsum}

\newcommand{\newabstract}[1]{%
  \par\bigskip
  \csname otherlanguage*\endcsname{#1}%
  \csname captions#1\endcsname
  \item[\hskip\labelsep\scshape\abstractname.]
}

\usepackage[french,english]{babel} 
\usepackage[T1]{fontenc} \usepackage{ae}

\newcommand{\floor}[1]{\left\lfloor{#1}\right\rfloor}

\newcommand{\abs}[1]{\left\lvert{#1}\right\rvert}
\newcommand{\norm}[1]{\left\|{#1}\right\|}

\newcommand{\scprod}[2]{\left\langle{#1},{#2}\right\rangle}

\newcommand{\bb}{\mathbb} 
\newcommand{\mc}{\mathcal} \newcommand{\ms}{\mathscr}

\newcommand{\R}{\mathbb{R}}\newcommand{\N}{\mathbb{N}}
\newcommand{\Z}{\mathbb{Z}}\newcommand{\Q}{\mathbb{Q}}
\newcommand{\T}{\mathbb{T}}
\newcommand{\A}{\mathbb{A}}

 \newcommand{\ie}{i.e.\ }

\newtheorem{theorem}{Theorem}[section] \newtheorem*{mainthmA}{Theorem
  A} \newtheorem*{mainthmB}{Theorem B} \newtheorem*{mainthmC}{Theorem
  C}

\newtheorem{proposition}[theorem]{Proposition}
\newtheorem{conjecture}[theorem]{Conjecture}
\newtheorem{lemma}[theorem]{Lemma}
\newtheorem{corollary}[theorem]{Corollary}
\newtheorem{claim}[theorem]{Claim}

\theoremstyle{definition} \newtheorem{definition}[theorem]{Definition}

\theoremstyle{remark} \newtheorem{remark}[theorem]{Remark} 

\newcommand{\Homeo}[1]{\mathrm{Homeo}_{#1}}

\newcommand{\Leb}{\mathrm{Leb}} \newcommand{\dd}{\:\mathrm{d}}

\DeclareMathOperator{\M}{\mathfrak{M}} 
\DeclareMathOperator{\Per}{Per} \DeclareMathOperator{\Fix}{Fix}

 \DeclareMathOperator{\sign}{sign}
\DeclareMathOperator{\diam}{diam}

\newcommand{\Symp}[1]{\mathrm{Symp}_{#1}(\T^2)}
\newcommand{\Ham}{\mathrm{Ham}(\T^2)}
\newcommand{\Thomeo}{\widetilde{\mathrm{Homeo}_0}}
\newcommand{\Tsymp}{\widetilde{\mathrm{Symp}_0}(\T^2)}
\newcommand{\Tham}{\widetilde{\mathrm{Ham}}(\T^2)}
\newcommand{\wh}{\widehat}
\newcommand{\I}{\mathscr{I}}

\DeclareMathOperator{\Ad}{Ad} \DeclareMathOperator{\Flux}{Flux}
\DeclareMathOperator{\inter}{int}

\newcommand{\F}{\mathscr{F}} \newcommand{\cc}{\mathrm{cc}}
\newcommand{\pr}[1]{\mathrm{pr}_{#1}} \newcommand{\Bs}{\mathcal{S}}
\newcommand{\B}{\mathscr{B}}
\newcommand{\Ss}{\mathbb{S}}

\title[Minimal homeomorphisms of $\T^2$]{On the dynamics of minimal
  homeomorphisms of $\T^2$ which are not pseudo-rotations}

\author{Alejandro Kocsard} 

\email{akocsard@id.uff.br}

\address{IME - Universidade Federal Fluminense. Rua Prof. Marcos
  Waldemar de Freitas Reis, S/N. Bloco H, $4^\circ$
  andar. 24.210-201, Gragoatá, Niterói, RJ, Brasil}

\date{\today}

\begin{document}

\begin{abstract}
  We prove that any minimal $2$-torus homeomorphism which is isotopic
  to the identity and whose rotation set is not just a point exhibits
  uniformly bounded rotational deviations on the perpendicular
  direction to the rotation set. As a consequence of this, we show
  that any such homeomorphism is topologically mixing and we prove
  Franks-Misiurewicz conjecture under the assumption of minimality.
\end{abstract}

\maketitle

\section{Introduction}
\label{sec:intro}

The study of the dynamics of orientation preserving circle
homeomorphisms has a long and well established history that started
with the celebrated work of Poincaré~\cite{Poincare1880memoire}. If
$f\colon\T=\R/\Z\carr$ denotes such a homeomorphism and
$\tilde f\colon\R\carr$ is a lift of $f$ to the universal cover, he
showed there exists a unique $\rho\in\R$, the so called \emph{rotation
  number} of $\tilde f$, such that
\begin{displaymath}
  \frac{\tilde f^n(z)-z}{n}\to\rho,\quad\text{as } n\to\infty,\
  \forall z\in\R, 
\end{displaymath}
where the convergence is uniform in $z$. Moreover, in this case a
stronger (and very useful, indeed) condition holds: every orbit
exhibits \emph{uniformly bounded rotational deviations}, \ie
\begin{displaymath}
  \abs{\tilde f^n(z)-z-n\rho}\leq 1,\quad\forall n\in\Z,\ \forall
  z\in\R.
\end{displaymath}

In this setting, the homeomorphism $f$ has no periodic orbit if and
only if the rotation number is irrational; and any minimal circle
homeomorphism is topologically conjugate to a rigid irrational
rotation.

However, in higher dimensions the situation dramatically changes. If
$f\colon\T^d=\R^d/\Z^d\carr$ is a homeomorphism homotopic to the
identity and $\tilde f\colon\R^d\carr$ is a lift of $f$, then one can
define its \emph{rotation set} by
\begin{displaymath}
  \rho(\tilde f):=\left\{\rho\in\R^d : \exists n_k\uparrow +\infty,\ 
    z_k\in\R^d,\ \rho=\lim_{k\to+\infty} \frac{\tilde
      f^{n_k}(z_k)-z_k}{n_k}\right\}. 
\end{displaymath}
This set is always compact and connected, and as we mentioned above,
it reduces to a point when $d=1$. But for $d\geq 2$ some examples with
larger rotation sets can be easily constructed.

In the two-dimensional case, which is the main scenario of this work,
Misiurewicz and Ziemian showed in \cite{MisiurewiczZiemian} that the
rotation set is not just connected but convex. So, when $d=2$ all
torus homeomorphisms of the identity isotopy class can be classified
according to the geometry of their rotation sets: they can either have
non-empty interior, or be a non-degenerate line segment, or be just a
point. In the last case, such a homeomorphism is called a
\emph{pseudo-rotation}.

Regarding the boundedness of rotational deviations, this property has
been shown to be very desirable in the study of the dynamics of
pseudo-rotations (see for instance the works of J\"ager and
collaborators \cite{JaegerConcBoundMeanMot, JaegerLinearConsTorus,
  JaegerTalIrratRotFact}). However, it has been proved in
\cite{KocKorFoliations} and \cite{KoropeckiTalAreaPrsIrrotDiff} that,
in general, pseudo-rotations does not exhibit bounded rotational
deviations in any direction of $\R^2$, \ie it can hold
\begin{displaymath}
  \sup_{z\in\R^2,n\in\Z} \scprod{\tilde f^n(z)-z-n\rho(\tilde
    f)}{v}=+\infty, \quad\forall v\in\Ss^1.
\end{displaymath}

When $\rho(\tilde f)$ is a (non-degenerate) line segment, of course
there exist points with different rotation vectors, so we cannot
expect to have any boundedness at all for rotational deviations on the
plane. However, in such a case there exists a unit vector $v\in\Ss^1$
and a real number $\alpha$ such that $\rho(\tilde f)$ is contained in
the line $\{z\in\R^2 : \scprod{z}{v}=\alpha\}$, so one can analyze the
boundedness of \emph{rotational $v$-deviations,} \ie whether there
exist constants $M(z)\in\R$ such that
\begin{displaymath}
  \abs{\scprod{\tilde f^n(z)-z-n\rho}{v}} = \abs{\scprod{\tilde
      f^n(z)-z}{v} -n\alpha}\leq M(z),\quad\forall n\in\Z,
\end{displaymath}
and any $\rho\in\rho(\tilde f)$.

Unlike the case of pseudo-rotations, when $\rho(\tilde f)$ is a
non-degenerate line segment in general it is expected to have
uniformly bounded rotational $v$-deviations, \ie the constant $M(z)$
can be taken independently of $z$. This result has been already proved
by Dávalos~\cite{DavalosAnnularMapsTorus} in the case where
$\rho(\tilde f)$ has rational slope and intersects $\Q^2$, extending a
previous result of Guelman, Koropecki and
Tal~\cite{GuelKorTalAnnAreaPresToral}. In those works periodic orbits
of $f$ play a key role.

However, the situation is considerably subtler when dealing with
periodic point free homeomorphisms. So far there did not exist any
\emph{a priori} boundedness of rotational deviations of torus
homeomorphisms which are not pseudo-rotations and with no periodic
points. In fact, it had been conjectured that any periodic point free
homeomorphism should be a pseudo-rotation. More precisely, Franks and
Misiurewicz had proposed in \cite{FranksMisiure} the following
\begin{conjecture}[Franks-Misiurewicz Conjecture]
  \label{conj:fanks-misiur}
  Let $f\colon\T^2\carr$ be a homeomorphism homotopic to the identity
  and $\tilde f\colon\R^2\carr$ be a lift of $f$ such that
  $\rho(\tilde f)$ is a non-degenerate line segment.

  Then, the following dichotomy holds:
  \begin{enumerate}[(i)]
  \item either $\rho(\tilde f)$ has irrational slope and one of its
    extreme points belongs to $\Q^2$;
  \item or $\rho(\tilde f)$ has rational slope and contains infinitely
    many rational points.
  \end{enumerate}
\end{conjecture}

Recently Avila has announced the existence of a minimal smooth
diffeomorphism whose rotation set is an irrational slope segment
containing no rational point, providing in this way a counter-example
to the first case of Franks-Misiurewicz Conjecture. On the other hand,
Le Calvez and Tal have proved in \cite{LeCalvezTalForcingTheory} that
if $\rho(\tilde f)$ has irrational slope and contains a rational
point, then this point is an extreme one.

The second case of Conjecture~\ref{conj:fanks-misiur} remains open,
\ie whether there exists a homeomorphism $f$ such that
$\rho(\tilde f)$ has rational slope and
$\rho(\tilde f)\cap\Q^2=\emptyset$, and in fact this is one of the
main motivations of our work.

The main result of this paper is the following \emph{a priori
  boundedness } for rotational deviations of minimal homeomorphisms:

\begin{mainthmA}
  Let $f\colon\T^2\carr$ be a minimal homeomorphism homotopic to the
  identity which is not a pseudo-rotation. Then there exists a unit
  vector $v\in\R^2$ and a real number $M>0$ such that for any lift
  $\tilde f\colon\R^2\carr$, there is $\alpha\in\R$ so that
  \begin{equation}
    \label{eq:bounded-v-dev-Thm-A}
    \abs{\scprod{\tilde f^n(z)-z}{v} - n\alpha}\leq M,\quad\forall
    z\in\R^2,\ \forall n\in\Z.
  \end{equation}
\end{mainthmA}

As a consequence of Theorem A and a recent result due to Koropecki,
Passeggi and Sambarino~\cite{KoropeckiPasseggiSambarino}, we get a
proof of the second case of Franks-Misiurewicz Conjecture
(Conjecture~\ref{conj:fanks-misiur}) under minimality assumption. More
precisely we get the following:

\begin{mainthmB}
  There is no minimal homeomorphism of $\T^2$ in the identity isotopy
  class such that its rotation set is a non-degenerate rational slope
  segment.
\end{mainthmB}

As a consequence of Theorem B, some results of
\cite{KocRotDevPerPFree} and a recent generalization of a theorem of
Kwapisz \cite{KwapiszAPriori} due to Beguin, Crovisier and Le Roux
\cite{BegCrovLeRProjectiveTrans}, we have the following

\begin{mainthmC}
  If $f\colon\T^2\carr$ is a minimal homeomorphism homotopic to the
  identity and is not a pseudo-rotation, then $f$ is topologically
  mixing.

  Moreover, in such a case the rotation set of $f$ is a non-degenerate
  irrational slope line segment and its supporting line does not
  contain any point of $\Q^2$.
\end{mainthmC}

% In fact, as a straightforward consequence of a result due to
% Gottschalk and Hedlund~\cite{GottschalkHedlundTopDyn} (see
% Theorem~\ref{thm:gott-hedlund} for precise statement), one can show
% that any minimal torus homeomorphism exhibiting uniformly bounded
% rotational deviations with respect to a rational slope direction is
% indeed a topological extension of an irrational circle rotation (see
% \cite[Proposition 2.1]{JaegerLinearConsTorus} for details). On the
% other hand, Koropecki, Passeggi and Sambarino have recently proved
% in \cite{KoropeckiPasseggiSambarino} that any torus homeomorphism
% which is homotopic to the identity and is a topological extension of
% an irrational circle rotation is necessarily a pseudo-rotation. So,
% combining their result with our Theorem A, we get a proof of the
% second case of Franks-Misiurewicz Conjecture under the hypothesis of
% minimality.

\subsection{Strategy of the proof Theorem A}
\label{sec:strat-proof-main-thms}

Theorem A is certainly the most important result of the paper and its
proof is rather long and technical. So, for the sake of readability,
here we summarize the main steps of the proof in a rather informal
way.

We proceed by contradiction. First of all one can observe that there
is no loss of generality assuming the rotation set $\rho(\tilde f)$ is
transversal to the horizontal axes, \ie it intersects the upper and
lower horizontal semi-planes (see
Propositions~\ref{pro:rot-set-elementary-properties} and
\ref{pro:powers-min-homeos-min-too} for details). This means there
exist points with positive asymptotic vertical mean speed and others
with negative one.

Then we define the \emph{stable sets at infinity}
$\Lambda_h^+,\Lambda_h^-\subset\R^2$ as the unbounded connected
components of the maximal $\tilde f$-invariant sets of the upper and
lower semi-plane, respectively. More precisely,
\begin{displaymath}
  \Lambda_h^+:=\cc\left(\left\{z\in\R^2 : \pr{2}\big(\tilde f^n(z)\big)\geq
      0,\ \forall n\in\Z\right\}, \infty\right)
\end{displaymath}
and
\begin{displaymath}
  \Lambda_h^-:=\cc\left(\left\{z\in\R^2 : \pr{2}\big(\tilde f^n(z)\big)\leq
    0,\ \forall n\in\Z\right\},\infty\right),
\end{displaymath}
where $\pr{2}\colon\R^2\to\R$ denotes the projection on the second
coordinate and $\cc(\cdot,\infty)$ the union of the unbounded
connected components of the corresponding set.  In
\S\ref{sec:stable-sets-infinity-trans-dir}, we study the geometry of
these stable sets at infinity, showing in particular that they are non
empty (Theorem~\ref{thm:Lambda-vert-sets-main-properties}), and they
are in fact the union of ``infinitely long hairs''. Then, assuming
estimate \eqref{eq:bounded-v-dev-Thm-A} is false, we show in
Theorem~\ref{thm:Lambda-vert-oscillations} that these ``hairs''
exhibit arbitrarily large oscillations in the horizontal direction.

Then, in \S\ref{sec:stab-set-infty-parall-dir} we define the
\emph{stable sets at infinity} but this time with respect to the
direction determined by the rotation set. At this point some new
important technical problems appear. In fact, the number $\alpha$ in
\eqref{eq:bounded-v-dev-Thm-A} represents the mean asymptotic speed of
every point with respect to the perpendicular direction to the
rotation set, and we know a posteriori, by Theorem C, that it is
always irrational, and in particular, non-zero. That means if we just
define theses sets analogously to what we did above for $\Lambda_h^+$
and $\Lambda_h^-$, we shall just get empty sets. On the other hand, if
we modify the definition writing
\begin{displaymath}
  \Lambda_v^+:=\cc\left(\left\{z\in\R^2 : \scprod{\tilde f^n(z)}{v}
      -n\alpha \geq 0,\ \forall n\in\Z\right\}, \infty\right)
\end{displaymath}
and
\begin{displaymath}
  \Lambda_v^-:=\cc\left(\left\{z\in\R^2 : \scprod{\tilde f^n(z)}{v}
      -n\alpha \leq 0,\ \forall n\in\Z\right\},\infty\right),
\end{displaymath}
we get non-empty sets, but they are not dynamically defined. So, this
is the reason why we have to introduce the \emph{fiber-wise
  Hamiltonian skew-products} in \S\ref{sec:fiberw-hamilt-skew-prod}
in order to get these sets $\Lambda_v^+$ and $\Lambda_v^-$ as
dynamical ones (see \S\ref{sec:fibered-stable-sets} for details).

Then, always assuming that \eqref{eq:bounded-v-dev-Thm-A} does not
hold, we use these sets to show that the induced fiber-wise
Hamiltonian skew-product exhibits certain form of topologically mixing
behavior along the fibers
(Theorem~\ref{thm:unbounded-dev-implies-spreading}). Then we use this
dynamical information to show the sets
$\Lambda_v^+,\Lambda_v^-,\Lambda_h^+,\Lambda_h^-$ are pairwise
disjoint
(Proposition~\ref{pro:Lambda-vert-Lambda-hor-disjoint}). Finally, we
finish the proof showing that this disjointedness is incompatible with
the large horizontal oscillation of the connected components of the
sets $\Lambda_h^+$ and $\Lambda_h^-$ we proved in
Theorem~\ref{thm:Lambda-vert-oscillations}.

\subsection*{Acknowledgments}
\label{sec:acknowledgments}

I would like to thank Andrés Koropecki, Patrice Le Calvez, Fábio Tal
and Artur Avila for many insightful discussions about this work. I
also thank the anonymous referees for many helpful comments and
suggestions.

This research was partially supported by CNPq-Brasil and
FAPERJ-Brasil.

\section{Preliminaries and notations}
\label{sec:prel-notat}

\subsection{Maps, topological spaces and groups}
\label{sec:maps-top-groups}

Given any map $f\colon X\carr$, we write $\Fix(f)$ for its set of
fixed points and $\Per(f):=\bigcup_{n\geq 1}\Fix(f^n)$ for the set of
periodic ones. If $A\subset X$ denotes an arbitrary subset, we define
its positively maximal $f$-invariant subset by
\begin{displaymath}
  \I_f^+(A):=\bigcap_{n\geq 0} f^{-n}(A).
\end{displaymath}
When $f$ is bijective, we can also define its maximal $f$-invariant
subset by
\begin{equation}
  \label{eq:max-inv-set-def}
  \I_f(A):=\I_f^+(A)\cap\I_{f^{-1}}^+(A) = \bigcap_{n\in\Z} f^n(A). 
\end{equation}

When $X$ is a topological space and $A\subset X$ is any subset, we
write $\inter A$ for the interior of $A$ and $\bar A$ for its
closure. When $A$ is connected, we write $\cc(X,A)$ for the connected
component of $X$ containing $A$. As usual, $\pi_0(X)$ denotes the set
of connected components of $X$. When $X$ is connected and
$A\subset X$, we say that $A$ \emph{disconnects} $X$ when
$X\setminus A$ is not connected. Given two connected sets
$U,V\subset X$, we say that $A$ \emph{separates} $U$ and $V$ when
$\cc(X\setminus A,U)\neq\cc(X\setminus A,V)$.

The space $X$ is said to be a \emph{continuum} when is compact,
connected and non-trivial, \ie it is neither empty nor a singleton.

A homeomorphism $f\colon X\carr$ is said to be \emph{non-wandering}
when given any non-empty open set $U\subset X$, there exists a
positive integer $n$ such that $f^n(U)\cap U\neq\emptyset$.

We say that $f$ is \emph{topologically mixing} when for every pair of
non-empty open sets $U,V\subset X$, there exists $N\in\N$ such that
$f^n(U)\cap V\neq\emptyset$, for every $n\geq N$.

The homeomorphism $f$ is said to be \emph{minimal} when it does not
exhibit any proper $f$-invariant closed set, \ie $X$ and $\emptyset$
are the only closed $f$-invariant sets.

If $(X,d)$ is a metric space, the open ball of radius $r>0$ and center
at $x\in X$ will be denoted by $B_r(x)$. Given an arbitrary set
$A\subset X$ and a point $x_0\in X$, we write
\begin{displaymath}
  d(x_0,A):=\inf_{y\in A}d(x_0,y).
\end{displaymath}
For any $\varepsilon>0$, the \emph{$\varepsilon$-neighborhood of} $A$
is given by
\begin{equation}
  \label{eq:epsilon-nbh-def}
  A_\varepsilon:=\{x\in X : d(x,A)<\varepsilon\}=\bigcup_{x\in A}
  B_\varepsilon(x).  
\end{equation}

The diameter of $A\subset X$ is defined by
$\diam A:=\sup_{x,y\in A}d(x,y)$ and we say $A$ is \emph{unbounded}
whenever $\diam A=+\infty$. Making a slight abuse of notation, we
shall write $\cc(A,\infty)$ to denote the union of the unbounded
connected components of $A$.

The space of (non-empty) compact subsets of $X$ will be denoted by
\begin{displaymath}
  \mc{K}(X):=\{K\subset X : K \text{ is compact, }
  K\neq\emptyset\}. 
\end{displaymath}
and we endow this space with its Hausdorff distance $d_H$ defined by
\begin{displaymath}
  d_H(K_1,K_2):= \max\left\{\max_{x\in K_1} d(x,K_2), \max_{y\in K_2}
    d(y,K_1)\right\},
\end{displaymath}
for every $K_1,K_2\in\mc{K}(X)$.

Whenever $M_1,M_2,\ldots M_n$ are $n$ arbitrary sets, we shall use the
generic notation
$\pr{i}\colon M_1\times M_2\times\ldots\times M_n\to M_i$ to denote
the $i^\mathrm{th}$-coordinate projection map.

Finally, when $X$ is a compact topological space, we shall always
consider the vector space of continuous functions $C^0(X,\R^d)$
endowed with the uniform norm given by
\begin{displaymath}
  \norm{\phi}_{C^0}:=\max_{1\leq i\leq d}\max_{x\in
    X}\abs{\pr{i}\big(\phi(x)\big)},\quad\forall \phi\in
  C^0(X,\R^d). 
\end{displaymath}

\subsection{Topological factors and extensions}
\label{sec:top-fact-extens}

Let $(X,d_X)$ and $(Y,d_Y)$ be two compact metric spaces. We say that
a homeomorphism $f\colon X\carr$ is a \emph{topological extension} of
a homeomorphism $g\colon Y\carr$ when there exists a continuous
surjective map $h\colon X\to Y$ such that $h\circ f=g\circ h$; and we
say $g$ is a \emph{topological factor} of $f$. In such a case, $h$ is
called a \emph{semi-conjugacy.}

As usual, when $h$ is a homeomorphism, $f$ and $g$ are said to be
\emph{topologically conjugate,} and $h$ is said to be a
\emph{conjugacy.}

\subsection{Euclidean spaces, tori and the annulus}
\label{sec:euclidean-spaces-tori}

We consider $\R^d$ endowed with its usual Euclidean structure, which
is denoted by $\scprod{\cdot}{\cdot}$. We write
$\norm{v}:=\scprod{v}{v}^{1/2}$ for its induced norm and
$d(v,w):=\norm{v-w}$ for its induced distance function.

The unit $(d-1)$-sphere is denoted by
$\Ss^{d-1}:=\{v\in\R^d : \norm{v}=1\}$. For any
$v\in\R^d\setminus\{0\}$ and any $r\in\R$ we define the (open)
\emph{half-space}
\begin{equation}
  \label{eq:Hvr-half-spaces-def}
  \bb{H}_{r}^v:=\left\{z\in\R^d: \scprod{z}{v}>r\right\}.
\end{equation}

Given any $\alpha\in\R^d$, we write $T_\alpha$ for the translation
$T_\alpha\colon z\mapsto z+\alpha$ on $\R^d$.

The $d$-dimensional torus $\R^d/\Z^d$ will be denoted by $\T^d$ and we
write $\pi\colon\R^d\to\T^d$ for the canonical quotient projection. We
will always consider $\T^d$ endowed with the distance
\begin{displaymath}
  d_{\T^d}(x,y):=\min\left\{d(\tilde x,\tilde y) : \tilde
    x\in\pi^{-1}(x),\ \tilde y\in\pi^{-1}(y)\right\}, \quad\forall
  x,y\in\T^d 
\end{displaymath}

Given any $\alpha\in\T^d$, we write $T_\alpha$ for the torus
translation $T_\alpha\colon\T^d\ni z\mapsto z+\alpha$. A point
$\alpha\in\R^d$ is said to be \emph{totally irrational} when
$T_{\pi(\alpha)}$ is minimal on $\T^d$.

In several places along this paper the symbol ``$\pm$'' shall have the
following meaning: given $v\in\R^d$, we write $\pm v$ to denote either
the vector $v$ or $-v$.

\subsubsection{The plane $\R^2$}
\label{sec:the-plane}

In the particular case of $d=2$, given any $v=(a,b)\in\R^2$, we define
$v^\perp:=(-b,a)$. For any $\alpha\in\R$ and any $v\in\Ss^1$, we shall
use the following notation for the straight line through the point
$\alpha v$ and perpendicular to $v$:
\begin{equation}
  \label{eq:line-ell-alpha-v-definition}
  \ell_\alpha^v:=\alpha v+\R v^\perp = \{ \alpha v + t v^\perp :
  t\in\R\}. 
\end{equation}

We say a vector $v\in\R^2\setminus\{0\}$ has \emph{rational slope}
when there exists $\alpha>0$ such that $\alpha v\in\Z^2$; and it is
said to have \emph{irrational slope} otherwise. 

We will also need the following notation for strips on $\R^2$: given
any $v\in\Ss^1$ and $r<s$, we define the (closed) strip
\begin{equation}
  \label{eq:Asv-strip-def}
  \A^v_{r,s}:=\overline{\bb{H}^v_{r}}\setminus \bb{H}^{v}_{s} =
  \{z\in\R^2 : r\leq \scprod{z}{v}\leq s\}.
\end{equation}

A \emph{Jordan curve} is any subset of $\R^2$ which is homeomorphic to
$\Ss^1$. A \emph{Jordan domain} is any bounded open subset of $\R^2$
whose boundary is a Jordan curve.

We shall need the following theorem due to Janiszewski (see for
instance, \cite[Chapter X, Theorem 2]{KuratowskiTopologyVol2}):

\begin{theorem}
  \label{thm:janiszewski}
  If $X_1,X_2\subset\Ss^2$ are two continua such that $X_1\cap X_2$
  is not connected, then $X_1\cup X_2$ disconnects $\Ss^2$, \ie
  $\Ss^2\setminus(X_1\cup X_2)$ is not connected.
\end{theorem}

\subsubsection{The annulus}
\label{sec:annulus}

The \emph{open annulus} is given by $\A:=\T\times\R$. Its universal
covering map will be denoted by $P\colon\R^2\to\A$ and is defined by 
\begin{equation}
  \label{eq:P-covering-map}
  P(x,y):=\big(\pi(x),y\big)=(x+\Z,y), \quad\forall (x,y)\in\R^2.   
\end{equation}
We will always considered the annulus endowed the distance
\begin{displaymath}
  d_\A(x,y):=\min\left\{\norm{\tilde x-\tilde y} : \tilde x\in
    P^{-1}(x),\ \tilde y\in P^{-1}(y)\right\}, \quad\forall x,y\in\A. 
\end{displaymath}

We write $\hat\A:=\A\sqcup\{-\infty,+\infty\}$ for the two-end
compactification of the annulus $\A$. Observe $\hat\A$ is homeomorphic
to the $2$-sphere $\bb{S}^2$.

We will need the following elementary result about unbounded connected
subsets of $\A$:

\begin{lemma}
  \label{lem:lifting-unbounded-continuum}
  Let $C\subset\A$ be a closed connected unbounded set. If $P$ is the
  covering map given by \eqref{eq:P-covering-map}, then every
  connected component of $P^{-1}(C)$ is unbounded in $\R^2$,
\end{lemma}

\begin{proof}
  Reasoning by contradiction, let us suppose there is a bounded
  connected component $K$ of $P^{-1}(C)\subset\R^2$.

  If we write $\overline{C}$ for the closure of $C$ in $\hat\A$, we
  get that $\overline{C}$ is compact and connected as well, and
  contains at least one of the two ends. Without loss of
  generalization, we can assume the upper end $+\infty$ belongs to
  $\overline{C}$.

  Then we consider a sequence of open connected subsets
  $(U_n)_{n\geq 1}$ of $\hat\A$ satisfying the following properties:
  $\overline{C}\subset U_n\subset\hat\A$ and
  $\overline{U_{n+1}}\subset U_n$, for every $n\geq 1$; and
  \begin{displaymath}
    \overline{C} = \bigcap_{n\geq 1} U_n = \bigcap_{n\geq 1}
    \overline{U_n}. 
  \end{displaymath}

  Such a sequence of nested open sets can be constructed as follows:
  one consider a distance function $d_{\hat\A}$ on $\hat\A$ which is
  compatible with its topology and then defines
  $U_n:={\overline{C}}_{1/n}$, for every $n\geq 1$, where
  ${\overline{C}}_{1/n}$ denotes the $1/n$-neighborhood of $\overline
  C$ with respect to the distance $d_{\hat\A}$ as defined by
  \eqref{eq:epsilon-nbh-def}. 

  Now let $z_0$ be an arbitrary point of $K\subset\R^2$. So we have
  $P(z_0)\in C\subset U_n$, for every $n\geq 1$. Since $U_n$ is open
  and connected, there is a continuous curve
  $\gamma_n\colon [0,1]\to U_n$ such that $\gamma_n(0)=P(z_0)$,
  $\gamma_n(1)=+\infty$ and $\gamma_n(t)\in\A$, for each $n\in\N$ and
  every $t\in [0,1)$.

  Then observe that, since $P$ is a covering map, there exists a
  unique continuous curve $\tilde\gamma_n\colon [0,1)\to\R^2$ such
  that $\tilde\gamma_n(0)=z_0$ and $P\circ\tilde\gamma_n=\gamma_n$.

  If $\hat\R^2:=\R^2\sqcup\{\infty\}$ denotes the one-point
  compactification of $\R^2$, one sees that each $\tilde\gamma_n$ has
  a unique continuous extension from $[0,1]$ to $\hat\R^2$ just
  defining $\tilde\gamma_n(1):=\infty$. In this way, each
  $\tilde\gamma_n\big([0,1]\big)$ is a compact subset of $\hat\R^2$,
  and by compactness of the Hausdorff space $\mathcal{K}(\hat\R^2)$,
  there exists a sub-sequence $n_j\to\infty$ and a non-empty compact
  subset $L\subset\hat\R^2$ such that
  $\tilde\gamma_{n_j}\big([0,1]\big)\to L$, as $n_j\to\infty$, where
  the convergence is respect to the the Hausdorff distance. One can
  easily verify that $L$ is connected, both points $z_0$ and $\infty$
  belong to $L$ and $P\big(L\setminus\{\infty\}\big)\subset C$. In
  particular, $L\setminus\{\infty\}$ is a closed, connected, unbounded
  subset of $P^{-1}(C)\subset\R^2$ and
  $K\cap (L\setminus\{\infty\})\neq\emptyset$, contradicting the fact
  that $K$ were a bounded connected component of $P^{-1}(C)$.
\end{proof}

\subsection{Ergodic theory and cocycles}
\label{sec:ergodic-theory-cocycles}

Given a topological space $X$, we write $\B_X$ to denote its Borel
$\sigma$-algebra.

The Haar (probability) measure on $(\T^d,\B_{\T^d})$, also called
Lebesgue measure, will be denoted by $\Leb_d$. By a slight abuse of
notation, we will also write $\Leb_d$ for the Lebesgue measure on
$\R^d$; and for the sake of simplicity of notation, we shall just
write $\Leb$ instead of $\Leb_1$.

Given an arbitrary $\sigma$-finite measure space $(X,\ms{B},\mu)$, a
map $f\colon (X,\ms{B})\carr$ is said to be \emph{non-singular}
(respect to $\mu$) when it is measurable and, for every $B\in\ms{B}$,
it holds $\mu(f^{-1}(B))=0$ if and only if $\mu(B)=0$. A non-singular
map $f\colon (X,\ms{B})\carr $ is said to be \emph{conservative} (with
respect to $\mu$) when for every $B\in\ms{B}$ such that $\mu(B)>0$,
there exists $n\geq 1$ satisfying $\mu\big(B\cap f^{-n}(B)\big)>0$.

As usual, we say that a measurable map $f\colon (X,\B)\carr$ preserves
$\mu$ when $f_\star\mu=\mu$, where $f_\star\mu(B):=\mu(f^{-1}(B))$,
for every $B\in\ms{B}$; and $f$ is said to be an \emph{automorphism}
of $(X,\B,\mu)$ when it is bijective and its inverse is measurable and
preserves $\mu$, too.

Given an invertible map $f\colon X\carr$, a function
$\phi\colon X\to\R$ and any $n\in\Z$, one defines the \emph{Birkhoff
  sum}
\begin{equation}
  \label{eq:Birk-sum-def}
  \mc{S}_f^n(\phi) :=
  \begin{cases}
    \sum_{j=0}^{n-1}\phi\circ f^j,& \text{if } n\geq 1;\\
    0, & \text{if } n=0;\\
    -\sum_{j=1}^{-n}\phi\circ f^{-j},& \text{if } n<0.
  \end{cases}
\end{equation}

Putting together two classical results of
Atkinson~\cite[Theorem]{AtkinsonRecurrCocycles} and Schmidt
\cite[Proposition 6]{SchmidtRecCocStatRandWalk}, we get the following
\begin{theorem}
  \label{thm:atkinson}
  Let $(X,\ms{B},\mu)$ be a probability space,
  $f\colon (X,\mc{B},\mu)\carr$ be an ergodic automorphism and
  $\phi\in L^1(X,\ms{B},\mu)$ be a real function such that
  $\int_X\phi\dd\mu=0$. Then, the skew-product automorphism
  $F\colon X\times\R\carr$ given by
  \begin{equation}
    \label{eq:skew-product-Birk-definition}
    F(x,t):=\big(f(x),t+\phi(x)\big), \quad\forall (x,t)\in X\times\R,
  \end{equation}
  is conservative with respect to the $F$-invariant $\sigma$-finite
  measure $\mu\otimes\Leb$.
\end{theorem}

\subsection{Groups of homeomorphisms}
\label{sec:groups-homeos}

From now on and until the end of this section, $M$ will denote an
arbitrary topological manifold. We write $\Homeo{}(M)$ for the group
of homeomorphisms from $M$ onto itself. The subgroup formed by those
homeomorphisms which are homotopic to the identity $id_M$ will be
denoted by $\Homeo0(M)$.

\subsubsection{Torus homeomorphisms and their lifts}
\label{sec:torus-homeomorphisms}

The group of lifts of torus homeomorphisms which are homotopic to the
identity will be denoted by
\begin{displaymath}
  \Thomeo(\T^d):=\left\{\tilde f\in\Homeo0(\R^d) : \tilde f-id_{\R^d}\in
    C^0(\T^d,\R^d)\right\}.
\end{displaymath}
Notice that in this definition, as it is usually done, we are
identifying the elements of $C^0(\T^d,\R^d)$ with those
$\Z^d$-periodic continuous functions from $\R^d$ to itself.

Making some abuse of notation, we also write
$\pi\colon\Thomeo(\T^d)\to\Homeo0(\T^d)$ for the map that associates
to each $\tilde f$ the only torus homeomorphism $\pi\tilde f$ such
that $\tilde f$ is a lift of $\pi\tilde f$. Notice that with our
notations, it holds $\pi T_\alpha =T_{\pi(\alpha)}\in\Homeo0(\T^d)$,
for every $\alpha\in\R^d$.

Given any $\tilde f\in\Thomeo(\T^d)$, we define its \emph{displacement
  function} by
\begin{equation}
  \label{eq:Delta-cocycle}
  \Delta_{\tilde f}:=\tilde f-id_{\R^d}\in C^0(\T^d,\R^d).
\end{equation}
Observe that this function can be naturally considered as a cocycle
over $f:=\pi\tilde f$ because
\begin{equation}
  \label{eq:Delta-Birk-sum}
  \Delta_{\tilde f^n} = \sum_{j=0}^{n-1}\Delta_{\tilde f}\circ
  f^j, 
  \quad\forall n\geq 1.
\end{equation}
For the sake of readability, we shall use the usual notation for
cocycles defining
\begin{displaymath}
  \Delta_{\tilde f}^{(n)}:=\Delta_{\tilde f^n},\quad\forall n\in\Z.  
\end{displaymath}

The map $\R^d\ni\alpha\mapsto T_\alpha\in\Thomeo(\T^d)$ defines an
injective group homomorphism, and hence, $\R^d$ naturally acts on
$\Thomeo(\T^d)$ by conjugacy. However, since every element of
$\Thomeo(\T^d)$ commutes with $T_{\boldsymbol{p}}$, for all
$\boldsymbol{p}\in\Z^d$, we conclude $\T^d$ itself acts on
$\Thomeo(\T^d)$ by conjugacy, \ie the map
$\Ad\colon\T^d\times\Thomeo(\T^d)\to\Thomeo(\T^d)$ given by
\begin{equation}
  \label{eq:Ad-Td-definition}
  \Ad_t(\tilde f):=T_{\tilde t}^{-1}\circ \tilde f\circ T_{\tilde t},
  \quad\forall(t,\tilde f)\in\T^d\times\Thomeo(\T^d),\ \forall
  \tilde t\in\pi^{-1}(t),
\end{equation}
is well-defined.

\subsubsection{Invariant measures}
\label{sec:invariant-measures}

We write $\M(M)$ for the space of Borel probability measures on $M$. A
measure $\mu\in\M(M)$ is said to have \emph{total support} when
$\mu(A)>0$ for every non-empty open set $A\subset M$. We say $\mu$ is
a \emph{topological measure} if it has total support and no atoms.

For every $\mu\in\M(M)$, we consider the group of homeomorphisms
\begin{displaymath}
  \Homeo\mu(M):=\left\{f\in\Homeo{}(M) : f_\star\mu=\mu\right\}.
\end{displaymath}
Given $f\in\Homeo{}(M)$, we write
$\M(f):=\{\nu\in\M(M) : f_\star\nu=\nu\}$.

The following classical result is due to Oxtoby and
Ulam~\cite{OxtobyUlam}:
\begin{theorem}
  \label{thm:Oxtoby-Ulam}
  Let $M$ be a compact topological manifold and $\mu,\nu\in\M(M)$ two
  topological measures. Then, there exists $h\in\Homeo{}(M)$ such that
  $h_\star\mu=\nu$.
\end{theorem}

For the sake of simplicity of notation, on the two-dimensional torus
we define the group of \emph{symplectomorphisms} (also called
\emph{area-preserving homeomorphisms}) by
\begin{displaymath}
  \Symp{}:=\left\{f\in\Homeo{}(\T^2) : \Leb_2\in\M(f)\right\}. 
\end{displaymath}

It is well known that its connected component containing the identity,
which will be denoted by $\Symp0$, coincides with
$\Symp{}\cap\Homeo0(\T^2)$. We write
\begin{displaymath}
  \Tsymp:=\pi^{-1}\left(\Symp0\right)<\Thomeo(\T^2).
\end{displaymath}

\subsection{Rotation set and rotation vectors}
\label{sec:rotation-set}

Let $f\in\Homeo0(\T^d)$ be an arbitrary homeomorphism and
$\tilde f\in\Thomeo(\T^d)$ be a lift of $f$. The \emph{rotation set
  of} $\tilde f$ is given by
\begin{equation}
  \label{eq:rot-set-def}
  \rho(\tilde f):= \bigcap_{m\geq 0} \overline{\bigcup_{n\geq
      m}\left\{\frac{\Delta_{\tilde f}^{(n)}(z)}{n} :
      z\in\R^d\right\}}.
\end{equation}
It can be easily shown that $\rho(\tilde f)$ is non-empty, compact and
connected.

When $d=1$, by classical Poincaré theory of circle homeomorphisms
\cite{Poincare1880memoire} we know that $\rho(\tilde f)$ reduces to a
point, but in general this does not hold in higher dimensions.

We summarized some elementary facts about rotation sets which are due
to Misiurewicz and Ziemian~\cite[Proposition 2.1]{MisiurewiczZiemian}:
\begin{proposition}
  \label{pro:rot-set-elementary-properties}
  Given any $\tilde f\in\Thomeo(\T^d)$, the following properties hold:
  \begin{enumerate}[(i)]
  \item
    $\rho(\tilde f^n)=n\rho(\tilde f) :=\{n\rho\in\R^d :
    \rho\in\rho(\tilde f)\}$, for any $n\in\Z$;
  \item
    $\rho(T_{\boldsymbol{p}}\circ \tilde
    f)=T_{\boldsymbol{p}}\big(\rho(\tilde f)\big)$, for any
    $\boldsymbol{p}\in\Z^d$.
  \end{enumerate}
\end{proposition}

As a consequence of \emph{(ii)} of
Proposition~\ref{pro:rot-set-elementary-properties}, we see that given
any $f\in\Homeo0(\T^d)$ and any lift $\tilde f\colon\R^d\carr$ of $f$,
we can define
\begin{displaymath}
  \rho(f):=\pi(\rho(\tilde f))\subset\T^d.
\end{displaymath}

We say that $f\in\Homeo0(\T^d)$ is a \emph{pseudo-rotation} when
$\rho(f)$ is a singleton.

By \eqref{eq:Delta-Birk-sum} and \eqref{eq:rot-set-def}, we know the
rotation set is formed by accumulation points of Birkhoff averages of
the displacement function. So given any $\mu\in\M(f)$, one can define
its \emph{rotation vector} by
\begin{displaymath}
  \rho_\mu(\tilde f) := \int_{\T^d}\Delta_{\tilde f}\dd\mu.
\end{displaymath}
Thus, by Birkhoff ergodic theorem we get
$\rho_\mu(\tilde f)\in\rho(\tilde f)$, for every $f$-invariant ergodic
probability measure $\mu$. Moreover, the following holds:

\begin{theorem}[Theorem 2.4 in \cite{MisiurewiczZiemian}]
  \label{thm:mis-ziem-extreme-points}
  Let $f\in\Homeo0(\T^d)$ and $\tilde f\colon\R^d\carr$ be a lift of
  $f$. Then, for every extreme point $w\in\rho(\tilde f)$, there
  exists an ergodic measure $\mu\in\M(f)$ such that
  $\rho_\mu(\tilde f)=w$. Consequently, it holds
  \begin{displaymath}
    \mathrm{Conv}\big(\rho(\tilde f)\big)=\left\{\rho_\nu(\tilde f) : 
      \nu\in\M(\pi\tilde f)\right\},
  \end{displaymath}
  where $\mathrm{Conv}(\cdot)$ denotes the convex hull operator.
\end{theorem}

However, in the two-dimensional case rotation sets are always convex:

\begin{theorem}[Theorem 3.4 in \cite{MisiurewiczZiemian}]
  \label{thm:mis-ziem-convexity}
  For every $\tilde f\in\Thomeo(\T^2)$, we have
  \begin{displaymath}
    \rho(\tilde f) = \left\{\rho_{\nu}(\tilde f) : \nu\in
      \M\big(\pi\tilde f\big)\right\}.
  \end{displaymath}
\end{theorem}

\subsection{Hamiltonian homeomorphisms}
\label{sec:hamiltonian-homeos}

In the symplectic setting, that is when $\tilde f\in\Tsymp$, the
rotation vector of $\Leb_2$ is also called the \emph{flux of
  $\tilde f$} and is usually denoted by
$\Flux(\tilde f):=\rho_{\Leb_2}(\tilde f)$.  In this case, it can be
easily shown that the flux map $\Flux\colon\Tsymp\to\R^2$ is indeed a
group homomorphism. Since
\begin{displaymath}
  \Flux(T_{\boldsymbol{p}}\circ\tilde f) =
  T_{\boldsymbol{p}}(\Flux(\tilde f)), \quad\forall
  \boldsymbol{p}\in\Z^2,\ \forall f\in\Tsymp
\end{displaymath}
this homomorphism clearly induces a map $\Symp0 \to\T^2$ which, by
some abuse of notation, will be denoted by $\Flux$, too.

The kernel of this homomorphism $\Flux\colon\Symp0\to\T^2$ is denoted
by
\begin{displaymath}
  \Ham:=\{f\in\Symp0 : \Flux(f)=0\}\lhd\Symp0,
\end{displaymath}
\ie it is a normal subgroup of $\Symp0$. The elements of $\Ham$ are
called \emph{Hamiltonian homeomorphisms.}

Analogously, the kernel of $\Flux\colon\Tsymp\to\R^2$ is denoted by
\begin{displaymath}
  \Tham:=\left\{\tilde f\in\Tsymp : \Flux(\tilde f)=0 \right\}.
\end{displaymath}

\begin{remark}
  \label{rem:Ham-Tham-identification}
  Notice that $\Ham$ and $\Tham$ can be naturally identified. In fact,
  the restriction $\pi\big|_{\Tham}\colon\Tham\to\Ham$ is a continuous
  group isomorphism.
\end{remark}

Observe the following short exact sequence splits:
\begin{displaymath}
  0\xrightarrow{}\Ham\hookrightarrow\Symp0\xrightarrow{\Flux} 
  \T^2\xrightarrow{}0.
\end{displaymath}
In fact, the map $\T^2\ni\alpha\mapsto T_\alpha$ is a section of
$\Flux$, and thus, the group $\Symp0$ can be decomposed as a
semi-direct product $\Symp0=\T^2\ltimes\Ham$. In other words, given
$\alpha,\beta\in\T^2$ and $h,g\in\Ham$, we have
\begin{displaymath}
  (T_\alpha\circ h)\circ (T_\beta\circ g) = T_{\alpha+\beta}\circ
  \left(\Ad_\beta(h)\circ g\right),
\end{displaymath}
where the $\T^2$-action $\Ad$ is given by \eqref{eq:Ad-Td-definition}.

This elementary fact about the group structure of $\Symp0$ is our main
inspiration for the construction of the fiber-wise Hamiltonian
skew-product we will perform in \S\ref{sec:fiberw-hamilt-skew-prod}.

\subsection{Rotation set, periodic points and minimality}
\label{sec:rotation-set-p-p-minimality}

The following result due to Handel asserts that the rotation set of a
periodic point free homeomorphism has empty interior:
\begin{theorem}[Handel \cite{HandelPerPointFree}]
  \label{thm:Handel-ppfree-rho-empty-int}
  Let $f\in\Homeo0(\T^2)$ be such that $\Per(f)=\emptyset$ and
  $\tilde f\in\Thomeo(\T^2)$ be a lift of $f$. Then, there exist
  $v\in\Ss^1$ and $\alpha\in\R$ so that
  \begin{displaymath}
    \frac{\scprod{\tilde f^n(z)-z}{v}}{n}\to\alpha,\quad \text{as }
    n\to\infty,
  \end{displaymath}
  where the convergence in uniform in $z\in\R^2$. In other words, the
  rotation set $\rho(\tilde f)\subset\ell_\alpha^v$, where the
  straight line $\ell_\alpha^v$ is given by
  \eqref{eq:line-ell-alpha-v-definition}.
\end{theorem}

The following result due to Franks will play a fundamental role in our
work:

\begin{theorem}[Franks \cite{FranksRotSetPerPoint}]
  \label{thm:Franks-symp-per-points}
  If $\tilde f\in\Tsymp$ and $\Flux(\tilde f)=(p_1/q,p_2/q)\in\Q^2$,
  then there exists $z\in\R^2$ such that
  \begin{displaymath}
    \tilde f^q(z)=z+(p_1,p_2).
  \end{displaymath}
  In particular, $\pi(z)\in\Per(\pi\tilde f)$.
\end{theorem}

Every probability measure which is invariant under a minimal
homeomorphism is necessarily a topological measure. Hence, as a
straightforward consequence of Theorems \ref{thm:Oxtoby-Ulam} and
\ref{thm:Franks-symp-per-points}, we get the following
\begin{corollary}
  \label{cor:minimal-homeos-rho-cap-Q2-empty}
  If $f\in\Homeo0(\T^2)$ is minimal and $\tilde f\in\Thomeo(\T^2)$ is
  a lift of $f$, then
  \begin{displaymath}
    \rho(\tilde f)\cap\Q^2=\emptyset. 
  \end{displaymath}
\end{corollary}

\subsection{Classification of plane fixed points}
\label{sec:class-plane-fixed-points}

Let $V,V'\subset\R^2$ be two non-empty open sets and let
$f\colon V\to V'$ be a homeomorphism. Following the terminology of Le
Calvez~\cite{LeCalvezDynamiqueHomeosDuPlanVois}, a fixed point
$z_0\in\Fix(f)$ is said to be:
\begin{itemize}
\item \emph{isolated} when it is an isolated point of the set
  $\Fix(f)$;
\item \emph{accumulated} when every neighborhood of $z_0$ contains a
  periodic orbit of $f$ different from $z_0$;
\item \emph{dissipative} when $z_0$ admits a local basis
  $(U_n)_{n\geq 0}$ of neighborhoods such that
  $f(\partial U_n)\cap\partial U_n=\emptyset$, for every $n\geq 0$,
  \ie each neighborhood is either attractive or repulsive;
\item \emph{indifferent} when there exists a neighborhood $W$ of $z_0$
  such that $\overline{W}\subset V$ and for every Jordan domain
  $U\subset W$ which is a neighborhood of $z_0$ it holds
  \begin{displaymath}
    \cc\left(\I_f(\overline{U}), z_0\right)\cap\partial 
    U\neq\emptyset,
  \end{displaymath}
  where $\I_f(\overline{U})$ denotes the maximal $f$-invariant subset
  of $U$ given by \eqref{eq:max-inv-set-def}.
\end{itemize}

\subsection{Fixed point indexes}
\label{sec:fixed-point-indexes}

If $f\colon V\to V'$ is as in \S\ref{sec:class-plane-fixed-points},
and $\gamma\colon\Ss^1\to V$ is a Jordan curve such that
$\gamma(\Ss^1)\cap\Fix(f)=\emptyset$, then one defines the \emph{index
  of $f$ along $\gamma$} as the integer
\begin{displaymath}
  i(f,\gamma):=\deg\left(\Ss^1\ni t\mapsto
    \frac{f(\gamma(t))-\gamma(t)}{\norm{f(\gamma(t))-\gamma(t)}}
    \in\Ss^1\right),
\end{displaymath}
where $\deg(\cdot)$ denotes the topological degree.

When $z_0\in\Fix(f)$ is isolated, then one can define the \emph{index
  of $f$ at $z_0$} as
\begin{displaymath}
  i(f,z_0):=i(f,\partial U),
\end{displaymath}
where $U$ denotes any Jordan domain satisfying $\overline{U}\subset V$
and $\overline{U}\cap\Fix(f)=\{z_0\}$. Since this index does not
depend on the choice of $U$, this notion is well-defined.

We will need the following topological version of Leau-Fatou's flower
theorem due to Le Calvez~\cite{LeCalvezPropDynIndMaj1}, that has been
lately improved by Le Roux~\cite{LeRouxLeauFatou}:

\begin{theorem}
  \label{thm:Patrice-index-larger-1}
  Let us suppose $f\colon V\to V'$ is an orientation-preserving
  homeomorphism and $z_0\in\Fix(f)$ is an isolated fixed point such
  that $i(f,z_0)\geq 2$.  Then there exist two open non-empty subsets
  $W^+,W^-\subset V\setminus\{z_0\}$ such that
  \begin{enumerate}[(i)]
  \item $f^n(W^+)$ is well-defined for every $n\geq 0$,
    $f^m(W^+)\cap f^n(W^+)=\emptyset$ whenever $m$ and $n$ are
    different non-negative integers and $\omega_f(z)=\{z_0\}$, for
    every $z\in W^+$;
  \item $f^{-n}(W^-)$ is well-defined for every $n\geq 0$,
    $f^{-m}(W^-)\cap f^{-n}(W^-)=\emptyset$ whenever $m$ and $n$ are
    different non-negative integers and $\alpha_f(z)=\{z_0\}$, for
    every $z\in W^-$;
  \end{enumerate}
  where $\alpha_f$ and $\omega_f$ denote the $\alpha$- and
  $\omega$-limit sets, respectively.
\end{theorem}

The following result about indexes of iterates of non-accumulated
fixed points is due to Le Calvez and
Yoccoz~\cite{LeCalvezYoccozSuitesDesIndices} but its proof has never
been published (see for instance \cite[Proposition
3.3]{LeCalvezDynamiqueHomeosDuPlanVois}):

\begin{theorem}
  \label{thm:LeCalvez-Yoccoz-indexes}
  If $f$ is an orientation-preserving homeomorphism and
  $z_0\in\Fix(f)$ is isolated, non accumulated, non indifferent and
  non dissipative, then there exist integers $q\geq 1$ and $r\geq 1$
  such that
  \begin{displaymath}
    \begin{cases}
      i(f^k,z_0)= 1, & \text{if } k\not\in q\Z;\\
      i(f^k,z_0)= 1- rq, & \text{if } k\in q\Z.
    \end{cases}
  \end{displaymath}
\end{theorem}

\subsection{Minimal homeomorphisms}
\label{sec:minimal-homeos}

In this paragraph we recall some classical and elementary results
about minimal homeomorphisms that we shall frequently use all along
the paper.

We say that a subset $A\subset\Z$ has \emph{bounded gaps} if there
exists $N\in\N$ such that
\begin{equation}
  \label{eq:bounded-gaps-def}
  A\cap\{n,n+1,\ldots,n+N\}\neq\emptyset, \quad\forall n\in\Z.
\end{equation}
The minimum natural number $N$ such that \eqref{eq:bounded-gaps-def}
holds shall be denoted by $\ms{G}(A)$.

The following three results are very well-known, but we decided to
include them here just for the sake of reference:
\begin{proposition}
  \label{pro:quasi-periodicity-minimal-homeos}
  If $(X,d)$ is a compact metric space and $f\colon X\carr$ is a
  minimal homeomorphism, then for every non-empty open set
  $U\subset X$ and any $x\in X$, the \emph{visiting time set}
  \begin{displaymath}
    \tau(x,U,f):=\{n\in\Z : f^n(x)\in U\}
  \end{displaymath}
  has bounded gaps.
\end{proposition}

As a consequence of this result, one can easily show the following:
\begin{corollary}
  \label{cor:quas-periodicity-torus-translations}
  For every $\alpha\in\T^d$, any $x\in\T^d$ and any neighborhood
  $V\subset\T^d$ of $x$, the visiting time set $\tau(x,V,T_\alpha)$
  has bounded gaps.
\end{corollary}

\begin{proposition}
  \label{pro:powers-min-homeos-min-too}
  If $(X,d)$ is a connected compact metric space and $f\colon X\carr$
  is a minimal homeomorphism, then $f^n$ is minimal for every
  $n\in\Z\setminus\{0\}$.
\end{proposition}

The last result we recall here is due to Gottschalk and Hedlund and
deals with cohomological equations:

\begin{theorem}[Gottschalk, Hedlund~\cite{GottschalkHedlundTopDyn}]
  \label{thm:gott-hedlund}
  Let $X$ be a compact metric space and $f\colon X\carr$ be a minimal
  homeomorphism. Let $\phi\colon X\to\R$ be a continuous function and
  assume there exists $x_0\in X$ such that
  \begin{displaymath}
    \sup_{n\in\N}\abs{\sum_{j=0}^{n-1}\phi\big(f^j(x_0)\big)}<\infty.
  \end{displaymath}

  Then, there is a continuous function $u\colon X\to\R$ such that
  $u\circ f- u =\phi$. In particular,
  \begin{displaymath}
    \sup_{n\in\N}\abs{\sum_{j=0}^{n-1}\phi\big(f^j(x)\big)}\leq
    2\norm{u}_{C^0}<\infty,  \quad\forall x\in X. 
  \end{displaymath}
\end{theorem}

\section{An ergodic deviation result}
\label{sec:ergodic-deviation-thm}

This section is devoted to prove an abstract ergodic deviation theorem
that will play a key role in \S\ref{sec:coex-stable-sets}. Even though
this result might be already known to some experts, we were not able
to find any reference in the literature and thus we have decided to
include its proof here.

\begin{theorem}
  \label{thm:ergodic-deviations}
  Let $(X,\ms{B},\mu)$ be a probability space,
  $f\colon (X,\ms{B},\mu)\carr$ an ergodic automorphism and
  $\phi\in L^1(X,\ms{B},\mu)$ such that $\int_X\phi\dd\mu=0$. Let us
  suppose that
  \begin{equation}
    \label{eq:unbound-from-above-bound-from-below}
    \sup_{n\geq 0} \Bs_f^n\phi(x) = + \infty, \quad\text{and}\quad
    \inf_{n\geq 0}\Bs_f^n\phi(x) > -\infty,
  \end{equation}
  for $\mu$-a.e. $x\in X$, where $\Bs_f^n\phi$ denotes the Birkhoff
  sum given by \eqref{eq:Birk-sum-def}.

  Then, it holds
  \begin{displaymath}
    \sup_{n\leq 0} \Bs_f^n\phi(x) = + \infty, \quad\text{and}\quad
    \inf_{n\leq 0}\Bs_f^n\phi(x) > -\infty,
  \end{displaymath}
  for $\mu$-a.e. $x\in X$.
\end{theorem}

To prove Theorem~\ref{thm:ergodic-deviations}, first we need a lemma
which is a simple consequence of Theorem~\ref{thm:atkinson}:

\begin{lemma}
  \label{lem:erg-dev-lem}
  Let $(X,\ms{B},\mu)$, $f$ and $\phi$ be as in
  Theorem~\ref{thm:ergodic-deviations} and let us assume there exists
  $M>0$ such that
  \begin{equation}
    \label{eq:Sn-bounded-from-above}
    \sup_{n\in\N}\mc{S}_f^n\phi(x)<M,\quad \text{for }
    \mu\text{-a.e. } x\in X.
  \end{equation}

  Then, it holds
  \begin{equation}
    \label{eq:Sn-bounded-from-below}
    \inf_{n\in\N} \mc{S}_f^n\phi(x)\geq -M, \quad \text{for }
    \mu\text{-a.e. } x\in X.
  \end{equation}
\end{lemma}

\begin{proof}[Proof of Lemma~\ref{lem:erg-dev-lem}]
  Let us suppose \eqref{eq:Sn-bounded-from-below} is false. So, if we
  define
  \begin{displaymath}
    A_n^m:=\left\{x\in X : \Bs_f^m\phi(x) \leq -M - 1/n\right\},  
  \end{displaymath}
  for each $m,n\geq 1$, we have
  $\mu\Big(\bigcup_{m,n\geq 1} A_{m,n}\Big)>0$. Then, there exist
  $N,n\geq 1$ such that the set $A:=A_n^N$ satisfies $\mu(A)>0$ and
  \begin{equation}
    \label{eq:contradicting-Sn-bounded-from-below}
    \mc{S}_f^{N}\phi(x)\leq -M-\frac{1}{n}, \quad\forall x\in A.
  \end{equation}

  Now, let us consider the skew-product $F\colon X\times\R\carr$ given
  by \eqref{eq:skew-product-Birk-definition} and define the set
  $\tilde A:=A\times[-1/2n,1/2n]\subset X\times\R$. Since
  $\mu\otimes\Leb(\tilde A)=n^{-1}\mu(A)>0$, by
  Theorem~\ref{thm:atkinson} we know there exists $k\geq 1$ such that
  $\mu\otimes\Leb\big(\tilde A\cap F^{-k}(\tilde A)\big)>0$. By
  classical arguments in ergodic theory, this implies there exists a
  sequence $k_j\to+\infty$ such that
  \begin{displaymath}
    \mu\otimes\Leb\big(\tilde A\cap F^{-k_j}(\tilde
    A)\big)>0,\quad\forall j\in\N.
  \end{displaymath}
  This is equivalent to say that the set
  \begin{displaymath}
    B_j:=\left\{x\in A : f^{k_j}(x)\in A,\
      \abs{\Bs_f^{k_j}\phi(x)}\leq\frac{1}{n}\right\} 
  \end{displaymath}
  has positive $\mu$-measure, for every $j$.

  Now, choosing $j$ large enough in order to have $k_j>N$, and
  combining this with \eqref{eq:contradicting-Sn-bounded-from-below},
  we get
  \begin{displaymath}
    \mc{S}_f^{k_j-N}\big(f^N(x)\big) = \mc{S}_f^{k_j}(x) - \mc{S}_f^{N}(x)
    > -\frac{1}{n} + M+\frac{1}{n} = M,
  \end{displaymath}
  for $\mu$-a.e. $x\in B_j$. Since $\mu(B_j)>0$, this contradicts
  \eqref{eq:Sn-bounded-from-above}.
\end{proof}

\begin{proof}[Proof of Theorem~\ref{thm:ergodic-deviations}]
  Of course we can assume $\abs{\phi(x)}<\infty$, for every $x\in
  X$. By \eqref{eq:unbound-from-above-bound-from-below}, there exist
  $K>0$ and $A\in\ms{B}$ with $\mu(A)>0$ such that
  \begin{equation}
    \label{eq:-infty-bounded-on-A}
    \Bs_f^n\phi(x) > -K, \quad\forall n\geq 0,
  \end{equation}
  for $\mu$-a.e. $x\in A$.

  Consider the functions
  $\tau_A^{+},\tau_A^{-}\colon X\to\N_0\cup\{\infty\}$ given by
  \begin{displaymath}
    \begin{split}
      \tau_A^+(x)&:=\min\{n\geq 0 : f^n(x)\in A\}, \\
      \tau_A^-(x)&:=\min\{n\geq 0 : f^{-n}(x)\in A\}, \quad\forall
      x\in X.
    \end{split}
  \end{displaymath}
  Since $f$ is ergodic, $\tau_A^+$ and $\tau_A^-$ are finite
  $\mu$-a.e. We also define the \emph{first return time} to $A$ by
  $r_A:=\tau_A^+\circ f\big|_A +1$.

  Now, we consider the probability space $(A,\ms{B}_A,\mu_A)$ given by
  $\ms{B}_A:=\{B\in\ms{B} : B\subset A\}$ and
  $\mu_A:=\mu(A)^{-1}\mu\big|_A$ and the ergodic automorphism
  $f_A\colon(A,\ms{B}_A,\mu_A)\carr$ given by the first return map:
  \begin{displaymath}
    f_A(x) := f^{\tau_A^+(f(x))}\big(f(x)\big) = f^{r_A(x)}(x),
  \end{displaymath}
  for $\mu$-a.e. $x\in A$. We also define the function
  $\phi_A(x):=\Bs_f^{r_A(x)}\phi(x)$. Then it holds
  $\phi_A\in L^1(A,\ms{B}_A,\mu_A)$ and $\int_A\phi_A\dd\mu =0$. On
  the other hand, by \eqref{eq:-infty-bounded-on-A} we know
  \begin{equation}
    \label{eq:bound-from-below-on-A}
    \Bs_{f_A}^n\phi_A(x) > -K, \quad\forall n\geq 0,
  \end{equation}
  for $\mu_A$-a.e. $x\in A$.

  Then, applying Lemma~\ref{lem:erg-dev-lem} in this context, we
  conclude that for $\mu_A$-a.e. $x\in A$ it holds
  \begin{equation}
    \label{eq:bound-from-above-on-A}
    \Bs_{f_A}^n\phi_A(x)\leq K, \quad\forall n\geq 0.
  \end{equation}

  Now, let us consider the measurable functions $M\colon A\to\R$ and
  $N\colon A\to\N_0$ given by
  \begin{equation}
    \label{eq:M-func-on-A-def}
    M(x) := \sup_{1\leq n \leq r_A(x)} \Bs_f^{n-1}\phi(x),
  \end{equation}
  and
  \begin{equation}
    \label{eq:N-funct-on-A-def}
    N(x):=\inf\left\{n\geq 0 : n<r_A(x),\ \Bs_f^n\phi(x) =
      M(x)\right\},
  \end{equation}
  and notice they are well define $\mu$-a.e. $x\in A$.

  For each pair $(m,n)\in\N\times\N_0$, let us consider the set
  \begin{displaymath}
    A_m^n := \left\{x\in A : m\leq M(x) < \infty,\ N(x)=n\right\}.  
  \end{displaymath}

  Putting together \eqref{eq:unbound-from-above-bound-from-below} and
  \eqref{eq:bound-from-above-on-A} it follows that
  \begin{equation}
    \label{eq:M-not-in-Linfty}
    \mu\bigg(\bigcup_{n\geq 0}A_m^n\bigg)>0, \quad\forall m\in\N. 
  \end{equation}

  By \eqref{eq:M-not-in-Linfty}, for each $m\in\N$ there exists
  $n_m\in\N$ so that $\mu(A_m^{n_m})>0$. So, let us define the set
  \begin{equation}
    \label{eq:B-set-definition}
    B:=\bigcap_{m\in\N} \bigcap_{i\geq 0} \bigcup_{j\geq i}
    f^j(A_m^{n_m}).
  \end{equation}
  Since $f$ is an ergodic automorphism, $B$ has full $\mu$-measure.

  Now, let us consider an arbitrary point $x\in B$ and any
  $m\in\N$. Since $A_m^{n_m}\subset A$, it clearly holds
  $\tau_A^-(x)<\infty$. By \eqref{eq:B-set-definition}, there exists a
  natural number $j=j(x,m)>n_m$ such that
  \begin{equation}
    \label{eq:choosing-j-to-fit-A-m-n-m}
    f^{-j}\big(f^{-\tau_A^-(x)}(x)\big) \in A_m^{n_m}. 
  \end{equation}
  Since both points $f^{-\tau_A^-(x)}(x)$ and
  $f^{-j}\big(f^{-\tau_A^-(x)}(x)\big)$ belong to $A$, there exists
  $j_A\in\N$ such that
  \begin{displaymath}
    f_A^{-j_A}\big(f^{-\tau_A^-(x)}(x)\big)=
    f^{-j}\big(f^{-\tau_A^-(x)}(x)\big).
  \end{displaymath}

  Now invoking \eqref{eq:bound-from-below-on-A},
  \eqref{eq:M-func-on-A-def}, \eqref{eq:N-funct-on-A-def} and
  \eqref{eq:choosing-j-to-fit-A-m-n-m}, we get
  \begin{equation}
    \label{eq:unboundedness-neg-times}
    \begin{split}
      &\Bs_f^{-\tau_A^-(x)-j+n_m}\phi(x) = \Bs_f^{-\tau_A^-(x)}\phi(x)
      +
      \Bs_f^{-j+n_m}\phi\big(f^{-\tau_A^-(x)}(x)\big) \\
      & = \Bs_f^{-\tau_A^-(x)}\phi(x) +
      \Bs_{f_A}^{-j_A}\phi_A\big(f^{-\tau_A^-(x)}(x)\big) +
      \Bs_{f}^{n_m}\phi\Big(
      f_A^{-j_A}\big(f^{-\tau_A^-(x)}(x)\big)\Big) \\
      &\geq \Bs_f^{-\tau_A^-(x)}\phi(x) - K + m.
    \end{split}
  \end{equation}

  Since $m$ is arbitrary in \eqref{eq:unboundedness-neg-times},
  $\mu(B)=1$ and $j>n_m$, we have proved that
  \begin{displaymath}
    \sup_{n\leq 0}\Bs_f^{n}\phi(x) = + \infty, \quad\text{for }
    \mu\text{-a.e. } x\in X.  
  \end{displaymath}

  On the other hand, let us consider the set
  \begin{displaymath}
    C:=\bigcap_{i\geq 0}\bigcup_{j\geq i} f^j(A).
  \end{displaymath}
  Since, $f$ is ergodic, it holds $\mu(C)=1$.

  Now, consider any $x\in C$ and any $n\in\N$. Thus, $\tau_A^-(x)$ and
  $\tau_A^-\big(f^{-n}(x)\big)$ are both finite, and both points
  $f^{-\tau_A^-(x)}(x)$ and $f^{-n-\tau_A^-(f^{-n}(x))}(x)$ belong to
  $A$. So, there exists $l_A=l_A(x,n)\in\N$ such that
  \begin{displaymath}
    f_A^{l_A}\big(f^{-n-\tau_A^-(f^{-n}(x))}(x)\big)=f^{-\tau_A^-(x)}(x).
  \end{displaymath}

  Then, we have
  \begin{displaymath}
    \begin{split}
      \Bs_f^{-n}\phi(x) &= \Bs_f^{-n-\tau_A^-(f^{-n}(x))}\phi(x) +
      \Bs_f^{\tau_A^-(f^{-n}(x))}\phi
      \big(f^{-n-\tau_A^-(f^{-n}(x))}(x)\big) \\
      &> \Bs_f^{-n-\tau_A^-(f^{-n}(x))}\phi(x) - K \\
      &= \Bs_f^{-\tau_A^-(x)}\phi(x) -
      \Bs_f^{n+\tau_A^-(f^{-n}(x))-\tau_A^-(x)}\phi
      \big(f^{-n-\tau_A^-(f^{-n}(x))}(x)\big) -K \\
      &= \Bs_f^{-\tau_A^-(x)}\phi(x) -
      \Bs_{f_A}^{l_A}\phi_A\big(f^{-n-\tau_A^-(f^{-n}(x))}(x)\big) -K \\
      &\geq \Bs_f^{-\tau_A^-(x)}\phi(x) - 2K, \\
    \end{split}
  \end{displaymath}
  where the first inequality follows from
  \eqref{eq:-infty-bounded-on-A} and the second one from
  \eqref{eq:bound-from-above-on-A}.

  From this last estimate, and since $\mu(C)=1$, it follows that
  \begin{displaymath}
    \inf_{n\leq 0}\Bs_f^n\phi(x) > -\infty, \quad\text{for }
    \mu\text{-a.e. } x\in X.  
  \end{displaymath}
\end{proof}

\section{Rotational deviations}
\label{sec:rotat-deviations}

In this section we enter into the core of this work: the study of
rotational deviations for $2$-torus homeomorphisms in the identity
isotopy class.

Let us start recalling some definitions we introduced in
\cite{KocRotDevPerPFree}. Let $f\colon\T^2\carr$ be a homeomorphism
homotopic to the identity and $\tilde f\in\Thomeo(\T^2)$ be a lift of
$f$. Let us suppose there exist $v\in\Ss^1$ and $\alpha\in\R$ such
that
\begin{equation}
  \label{eq:rot-set-in-line}
  \rho(\tilde f)\subset \ell_\alpha^v=\{\alpha v + tv^\perp :
  t\in\R\}.
\end{equation}
Observe that, by Theorem~\ref{thm:Handel-ppfree-rho-empty-int}, this
hypothesis is always satisfied when $f$ is periodic point free.

We say that a point $z_0\in\T^2$ exhibits \emph{bounded
  $v$-deviations} when there exists a real constant $M=M(z_0,f)>0$
such that
\begin{equation}
  \label{eq:bound-v-deviations}
  \scprod{\tilde f^n(\tilde z_0)-\tilde z_0-n\rho}{v} =
  \scprod{\Delta_{\tilde f}^{(n)}(z_0)}{v} - n\alpha \leq M,
  \quad\forall n\in\Z,
\end{equation}
for any $\tilde z_0\in\pi^{-1}(z_0)$, any $\rho\in\rho(\tilde f)$ and
where $\Delta_{\tilde f}$ denotes the displacement cocycle of
$\tilde f$ given by \eqref{eq:Delta-cocycle}.
 
Moreover, we say that $f$ exhibits \emph{uniformly bounded
  $v$-deviation} when there exists $M=M(f)>0$ such that
\begin{displaymath}
  \scprod{\Delta_{\tilde f}^{(n)}(z)}{v}- n\alpha \leq M, \quad\forall
  z\in\T^2,\ \forall n\in\Z. 
\end{displaymath}

Even though the straight lines $\ell_\alpha^v$ and
$\ell_{-\alpha}^{-v}$ coincide as subsets of $\R^2$, there is no
\emph{à priori} obvious relation between boundedness of $v$-deviation
and $(-v)$-deviation, because in our definition of ``bounded
$v$-deviations'' given by \eqref{eq:bound-v-deviations} we are just
considering boundedness from above.

However, we got the following result that relates both $v$- and
$(-v)$-deviations:

\begin{theorem}[Corollary 3.2 in \cite{KocRotDevPerPFree}]
  \label{thm:v-vs--v-rot-dev}
  If $f\in\Homeo0(\T^2)$ and $\tilde f\colon\R^2\carr$ is a lift of
  $f$ such that condition \eqref{eq:rot-set-in-line} holds, then $f$
  exhibits uniformly bounded $v$-deviations if and only if exhibits
  uniformly bounded $(-v)$-deviations.
\end{theorem}

As a particular case of our definition of boundedness rotational
deviations, let us recall that a homeomorphism $f\in\Homeo0(\T^2)$ is
said to be \emph{annular} (see for instance \cite{KoroTalStricTor,
  JaegerTalIrratRotFact}) when there exist a lift
$\tilde f\in\Thomeo(\T^2)$, $M>0$ and $v\in\Ss^1$ with rational slope
such that
\begin{equation}
  \label{eq:annularity-definition}
  \abs{\scprod{\Delta_{\tilde f}^{(n)}(z)}{v}}\leq M, \quad\forall
  z\in\T^ 2,\ \forall n\in\Z.  
\end{equation}
Observe that in such a case, the rotation set $\rho(\tilde f)$ is
contained in the line $\ell_0^v$, \ie the straight line parallel to
$v$ and passing through the origin.

On the other hand, a homeomorphisms $f\in\Homeo0(\T^2)$ is said to be
\emph{eventually annular} when there exists $k\in\N$ such that $f^k$
is annular.

In \cite{KocRotDevPerPFree} we proved that boundedness of
$v$-deviations is equivalent to the existence of certain invariant
topological object called torus \emph{pseudo-foliation.}

\subsection{Pseudo-foliations}
\label{sec:pseudo-foliations}

In this paragraph we recall the notions of \emph{plane and torus
  pseudo-foliations} we introduced in \cite{KocRotDevPerPFree}.

\subsubsection{Plane pseudo-foliations}
\label{sec:plane-pseudo-foliations}

Let $\F$ be a partition of $\R^2$. We shay that $\F$ is a \emph{plane
  pseudo-foliation} when every atom of $\F$ is closed, connected, has
empty interior and disconnects $\R^2$ in exactly two connected
components.

Given any $z\in\R^2$, we write $\F_z$ for the atom of the partition
$\F$ containing the point $z$. If $h\colon\R^2\carr$ is an arbitrary
map, we say that $\F$ is $h$-invariant when
\begin{displaymath}
  h\big(\F_z\big)=\F_{h(z)}, \quad\forall z\in\R^2. 
\end{displaymath}

Let us recall the following result of \cite[Preposition
5.1]{KocRotDevPerPFree}:
\begin{proposition}
  \label{pro:plane-pseudo-fol-unbound-connec-comp}
  If $\F$ is a plane pseudo-foliation, then both connected component
  of $\R^2\setminus\F_z$ are unbounded, for every $z\in\R^2$.
\end{proposition}

\subsubsection{Torus pseudo-foliations}
\label{sec:torus-pseudo-foliations}

A partition $\F$ of $\T^2$ is said to be a \emph{toral
  pseudo-foliation} whenever there exists a plane pseudo-foliation
$\widetilde\F$, called the \emph{lift of $\F$}, satisfying
\begin{displaymath}
  \pi\left(\widetilde\F_z\right)=\F_{\pi(z)},\quad\forall z\in\R^2.
\end{displaymath}
Notice that such a plane pseudo-foliation is $\Z^2$-invariant, \ie
$\widetilde\F$ is $T_{\boldsymbol{p}}$-invariant for every
$\boldsymbol{p}\in\Z^2$.

In \cite{KocRotDevPerPFree} we have gotten the following result that
guaranties the existence of an asymptotic homological direction for
torus pseudo-foliations:
\begin{theorem}
  \label{thm:torus-pseudo-fol-exist-asymp-dir}
  If $\widetilde\F$ is the lift of torus pseudo-foliation, then there
  exists $v\in\Ss^1$ and $M>0$ such that 
  \begin{displaymath}
    \abs{\scprod{w-z}{v}}\leq M, \quad\forall z\in\R^2,\
    \forall w\in\widetilde\F_z. 
  \end{displaymath}
\end{theorem}

The vector $v$ given by
Theorem~\ref{thm:torus-pseudo-fol-exist-asymp-dir} is unique up to
multiplication by $(-1)$. So, we will call it the \emph{asymptotic
  direction} of either the torus pseudo-foliation $\F$ or its lift
$\widetilde\F$.

One of the main results of \cite{KocRotDevPerPFree} is the following:
\begin{theorem}[Theorem 5.5 in \cite{KocRotDevPerPFree}]
  \label{thm:bounded-dev-iff-pseudo-fol}
  Let $f\in\Homeo0(\T^2)$ be a periodic point free, area-preserving,
  non-wandering and non-eventually annular homeomorphism. If $f$
  exhibits uniformly bounded $v$-deviations, for some $v\in\Ss^1$,
  then there exists an $f$-invariant pseudo-foliation whose asymptotic
  direction is given by $v^\perp$.
\end{theorem}

\subsection{Rotational deviations for minimal homeomorphisms}
\label{sec:rotat-devi-minim}

In this paragraph we present some simple results about rotational
deviations of minimal homeomorphisms. So, from now on let us assume
$f\in\Homeo0(\T^2)$ is minimal and $\tilde f\colon\R^2\carr$ is a lift
of $f$. By Theorem~\ref{thm:Handel-ppfree-rho-empty-int} we know there
exist $v$ and $\alpha$ such that the rotation set of $\tilde f$ is
contained in the line $\ell_\alpha^v$, \ie inclusion
\eqref{eq:rot-set-in-line} holds.

The following result is an improvement of
Theorem~\ref{thm:v-vs--v-rot-dev} under the minimality assumption:

\begin{proposition}
  \label{pro:gott-hedl-applied-to-torus}
  If $f$ is minimal, $\tilde f$ is a lift of $f$ and $v$ and $\alpha$
  are such that condition \eqref{eq:rot-set-in-line} holds, then the
  following properties are equivalent:
  \begin{enumerate}[(i)]
  \item $f$ does not exhibit uniformly bounded $v$-deviations;
  \item $f$ does not exhibit uniformly bounded $(-v)$-deviations;
  \item \label{item:large-deviations-future-past} for every $z\in\T^2$
    it holds
    \begin{displaymath}
      \sup_{n\geq 0}\abs{\scprod{\Delta_{\tilde f}^{(n)}(z)}{v}-
        n\alpha} = \sup_{n\leq 0}\abs{\scprod{\Delta_{\tilde
            f}^{(n)}(z)}{v}- n\alpha} = \infty.
    \end{displaymath}
  \end{enumerate}
\end{proposition}

\begin{proof}
  This is a straightforward consequence of
  Theorems~\ref{thm:gott-hedlund} and \ref{thm:v-vs--v-rot-dev}.
\end{proof}

For the proof of Theorem B we shall need the following

\begin{proposition}
  \label{pro:minimal-not-annular}
  If $f\in\Homeo0(\T^2)$ is a minimal homeomorphism, then it is not
  eventually annular.
\end{proposition}

\begin{proof}
  By Proposition~\ref{pro:powers-min-homeos-min-too}, $f^k$ is minimal
  for any $k\in\N$. So, it is enough to show that $f$ is not annular.

  Reasoning by contradiction, let us suppose $f$ is annular. Then,
  there exist a lift $\tilde f\colon\R^2\carr$ and $v\in\Ss^1$ with
  rational slope such that
  \begin{displaymath}
    \sup_{n\in\Z,\ z\in\T^2}\abs{\scprod{\Delta_{\tilde
          f}^{(n)}(z)}{v}} <\infty.
  \end{displaymath}
  Since $v$ has rational slope, by
  Proposition~\ref{pro:rot-set-elementary-properties} there is no loss
  of generality assuming $v=(1,0)$. So, by
  Theorem~\ref{thm:gott-hedlund}, there exists $u\in C^0(\T^2,\R)$
  satisfying
  \begin{equation}
    \label{eq:Delta1-coboundary}
    \pr{1}\circ \Delta_{\tilde f} = u\circ f - u.
  \end{equation}

  Now, let us consider the continuous maps
  $\tilde g, \tilde h\colon\R^2\carr$ given by
  \begin{align*}
    \tilde g(x,y)&:=\big(x, y+\pr{2}\circ\Delta_{\tilde f}(x,y)\big),\\ 
    \tilde h(x,y)&:=\big(x-u(x,y),y\big),
  \end{align*}
  for every $(x,y)\in\R^2$.

  As consequence of \eqref{eq:Delta1-coboundary}, we know
  $\tilde h\circ\tilde f = \tilde g\circ\tilde h$, and hence,
  $h\circ f=g\circ h$, where $g$ and $h$ are the continuous torus maps
  whose lifts are $\tilde g$ and $\tilde h$, respectively.  However,
  since $h$ is homotopic to the identity, it is surjective and $g$ is
  clearly not minimal, contradicting the minimality of $f$.

  So, $f$ is not annular.
\end{proof}

Even though our next result is rather simple, it may be useful in
future works:
\begin{theorem}
  \label{thm:bounded-rot-deviations-for-minimal-homeos}
  Let $f\in\Homeo0(\T^2)$ be a minimal homeomorphism,
  $\tilde f\colon\R^2\carr$ a lift of $f$ and $\rho$ any point in
  $\rho(\tilde f)$. Then for every $\varepsilon>0$ there exists
  $\delta>0$ such that for any $n\in\Z$ satisfying
  $d(n\rho,\Z^2)<\delta$, there is $z\in\R^2$ such that
  \begin{displaymath}
    \norm{\tilde f^n(z)-z-n\rho}<\varepsilon.
  \end{displaymath}
\end{theorem}

\begin{proof}
  By Theorem~\ref{thm:mis-ziem-convexity}, there exists $\mu\in\M(f)$
  such that $\rho_\mu(\tilde f)=\rho$. Since $f$ is minimal, $\mu$ is
  a topological measure (\ie has total support and no atoms). So, by
  Theorem~\ref{thm:Oxtoby-Ulam}, there exists $h\in\Homeo{}(\T^2)$
  such that $h_\star\Leb_2=\mu$. Moreover, after pre-composing with a
  linear automorphism of $\T^2$ if necessary, we can assume that $h$
  is isotopic to the identity. Then, if $\tilde h\in\Thomeo(\T^2)$ is
  a lift of $h$ and we write
  $\tilde g:=\tilde h^{-1}\circ\tilde f\circ \tilde h$, we have
  $\tilde g\in\Tsymp$ and $\Flux(\tilde g)=\rho$.

  Observing the displacement function $\Delta_{\tilde h}$ is
  $\Z^2$-periodic, and hence, uniformly continuous, given any
  $\varepsilon>0$ there exists $\delta>0$ such that
  $\norm{\Delta_{\tilde h}(x)-\Delta_{\tilde h}(y)}<\varepsilon$
  whenever $d_{\T^2}(x,y)<\delta$.

  On the other hand, we have
  \begin{displaymath}
    \Flux\big(T_\rho^{-n}\circ\tilde g^n\big) =
    \Flux\big(T_\rho^{-n}\big) + \Flux(g) = -n\rho +n\rho = 0.
  \end{displaymath}
  So, $T_\rho^{-n}\circ\tilde g^n\in\Tham$ for every $n\in\Z$. By
  Theorem~\ref{thm:Franks-symp-per-points}, for each $n$ there exists
  $z_n\in\R^2$ such that $T_\rho^{-n}\circ \tilde g^n(z_n)=z_n$. Then,
  \begin{displaymath}
    \tilde f^n\big(\tilde h(z_n)\big) = \tilde h(z_n+n\rho) = z_n +
    n\rho + \Delta_{\tilde h}(z_n+n\rho),
  \end{displaymath}
  and consequently, defining $w_n:=\tilde h(z_n)$, we get
  \begin{displaymath}
    \norm{\tilde f^n(w_n)-w_n - n\rho} = \norm{\Delta_{\tilde h}(z_n+n\rho)
      -\Delta_{\tilde h}(z_n)} <\varepsilon,
  \end{displaymath}
  whenever $d(n\rho,\Z^2)<\delta$.
\end{proof}

\section{Stable sets at infinity: transverse direction}
\label{sec:stable-sets-infinity-trans-dir}

All along this section $f\in\Homeo0(\T^2)$ will continue to denote a
minimal homeomorphism and $\tilde f\in\Thomeo(\T^2)$ a lift of
$f$. Here we start the study of \emph{stable sets at infinity}
associated to (certain lifts of) of $f$. We begin considering stable
sets at infinity with respect to the horizontal direction assuming
there exists a lift $\tilde f$ such that the rotation set
$\rho(\tilde f)$ intersects the horizontal axis.

We call this ``a transverse direction'' because, under the hypotheses
of Theorem A, there is no loss of generality assuming $\rho(\tilde f)$
intersects transversely the horizontal axis, modulo finite iterates,
conjugacy and appropriate choice of the lift.

To simplify our notation, in this section we will write $u$ to denote
either the vector $(0,1)$ or $(0,-1)$.

\begin{theorem}
  \label{thm:Lambda-vert-sets-main-properties}
  Let $f$ and $\tilde f$ be as above, and assume $\rho(\tilde f)$
  intersects the horizontal axis, \ie
  \begin{equation}
    \label{eq:rot-set-inters-X-axes}
    \rho(\tilde f)\cap\ell_0^{(0,1)}\neq\emptyset.
  \end{equation}
  For each $r\in\R$ and $u\in\{(0,1),(0,-1)\}$, consider the set
  \begin{equation}
    \label{eq:Lambda-ver-definition}
    \Lambda^{u}_r:=\cc\left(\I_{\tilde
        f}\big(\overline{\bb{H}_r^{u}}\big), \infty\right),
  \end{equation}
  where $\I_{\tilde f}\big(\overline{\bb{H}_r^{u}}\big)$ denotes the
  maximal invariant set given by \eqref{eq:max-inv-set-def}.

  Then, it holds:
  \begin{enumerate}[(i)]
  \item \label{eq:Lambda-v-equivariant-prop}
    $\Lambda_r^u=T_{(1,0)}(\Lambda_r^u)$,
    $T_{(0,1)}(\Lambda_r^u) = \Lambda_{r+1}^u$, and
    $\Lambda_r^u\subset \Lambda_s^u$, for any $r\in\R$ and any $s<r$;
  \item\label{eq:Lambda-v-non-empty}
    $\Lambda_r^{u}\cap\ell_r^{u}\neq\emptyset$, for every $r\in\R$;
  \item\label{eq:Lambda-v-minus-v-no-inter}
    $\I_{\tilde f}(\overline{\bb{H}_r^u})\cap \I_{\tilde
      f}(\overline{\bb{H}_{r'}^{-u}})=\emptyset$, for every
    $r,r'\in\R$;
  \item\label{eq:Lambda-v-in-no-strip} given any $r$ and any connected
    unbounded closed subset $\Gamma\subset\Lambda_r^u$, it holds
    \begin{displaymath}
      \Gamma\cap\ell_{s}^{u}
      \neq\emptyset,\quad \forall s>\inf\{\abs{\pr{2}(z)} :
      z\in\Gamma\}; 
    \end{displaymath}
    (see Figure~\ref{fig:Gamma-unbounded});
  \item\label{eq:Lambda-v-r-union-dense}
    \begin{displaymath}
      \overline{\bigcup_{r\in\R} \Lambda_r^{u}}=\R^2;
    \end{displaymath}
  \item\label{eq:Lambda-v-not-disconnect} For each $r\in\R$, the set
    $\I_{\tilde f}(\overline{\bb{H}_r^u})$ has empty interior and does
    not disconnect $\R^2$, \ie
    $\R^2\setminus\I_{\tilde f}(\overline{\bb{H}_r^u})$ is
    connected. In particular, this implies $\Lambda_r^u$ has empty
    interior and does not disconnect $\R^2$ as well.
  \end{enumerate}
\end{theorem}

\begin{figure}
  \centering \def\svgwidth{.9\columnwidth} 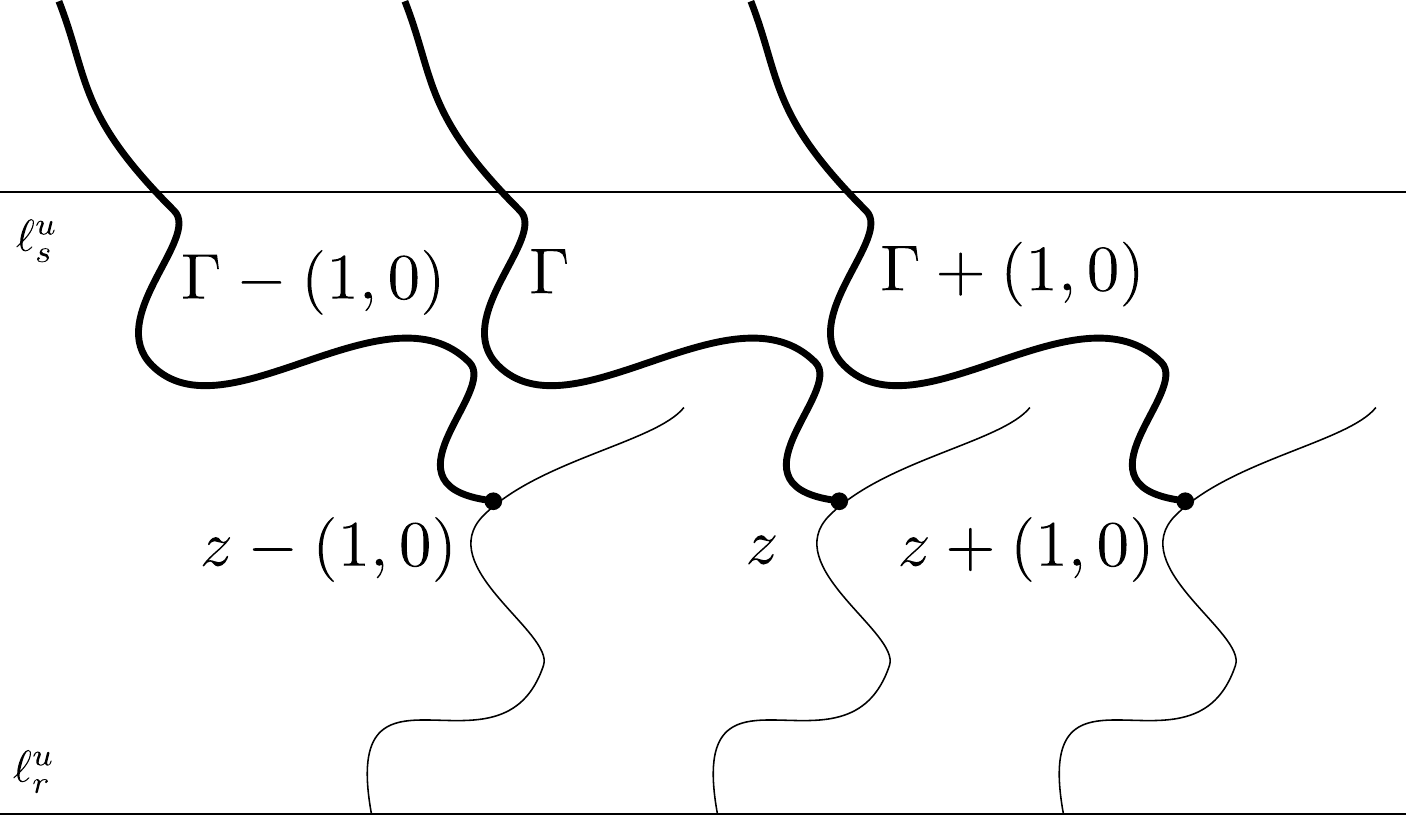
    \caption{$\Gamma\subset\Lambda_r^u$ intersects $\ell_s^u$ for
      $u=(0,1)$ and $s>\pr{2}(z)$.}
    \label{fig:Gamma-unbounded}
\end{figure}

\begin{proof}
  Statement \eqref{eq:Lambda-v-equivariant-prop} easily follows from
  the fact that $\bb{H}_r^u\subset\bb{H}_s^u $ whenever $r>s$ and
  recalling that $\tilde f$ commutes with every integer translation.

  To show \eqref{eq:Lambda-v-non-empty}, let $\A$ be the open annulus
  and $P\colon\R^2\to\A$ the covering map as defined in
  \S\ref{sec:annulus}. Since $\tilde f$ commutes with all deck
  transformations of $P$, it induces a homeomorphism on $\A$; and if
  we write $\hat\A:=\A\sqcup\{-\infty,+\infty\}$ for the two-end
  compactification of the annulus $\A$, this homeomorphism admits a
  unique extension to $\hat\A$. More precisely, one can define
  $\hat f\in\Homeo{}(\hat\A)$ by
  \begin{displaymath}
    \hat f(z) :=
    \begin{cases}
      +\infty,& \text{if } z=+\infty; \\
      -\infty,& \text{if } z=-\infty;\\
      P\big(\tilde f(\tilde z)\big),& \text{if } z\in\A,\ \tilde z\in
      P^{-1}(z).
    \end{cases}
  \end{displaymath}

  Now, we want to show both fixed points $-\infty$ and $+\infty$ are
  indifferent for $\hat f$, according to the classification of fixed
  points given in \S\ref{sec:class-plane-fixed-points}.

  Since $f$ is minimal, it holds
  $\Fix(\hat f)=\Per(\hat f)=\{-\infty,+\infty\}$. In particular, both
  fixed points are non-accumulated.

  On the other hand, since $\hat\A$ is homeomorphic to $\Ss^2$ and
  $\hat f$ is isotopic to the identity, by Lefschetz fixed point
  theorem one gets
  \begin{equation}
    \label{eq:Lefschetz-fix-pnt-formula}
    i\big(\hat f^n,-\infty\big) + i\big(\hat
    f^n,+\infty\big)=L(\hat f^n)=\chi(\hat\A)=2, \quad\forall 
    n\in\Z\setminus\{0\}. 
  \end{equation}

  Then, we make the following
  \begin{claim}
    \label{cla:index-less-than-2-at-infty}
    Fixed point indexes at $+\infty$ and $-\infty$ satisfy
    \begin{equation}
      \label{eq:index-less-than-2-at-infty}
      i\big(\hat f^n,\pm\infty\big)\leq 1, \quad\forall
      n\in\Z\setminus\{0\}. 
    \end{equation}
  \end{claim}

  Since the statement is completely symmetric, we will just prove
  Claim \ref{cla:index-less-than-2-at-infty} for the point
  $+\infty$. Let us proceed by contradiction. Suppose that
  $i(\hat f,+\infty)\geq 2$. Then, let $W^+$ and $W^-\subset\A$ denote
  the open sets given by Theorem~\ref{thm:Patrice-index-larger-1}.

  Let $\nu\in\M(f)$ be any ergodic $f$-invariant measure. We will
  consider the three possible cases: $\pr{2}(\rho_\nu(\tilde f))$ is
  either positive, negative or zero. Let us start assuming
  $\pr{2}(\rho_\nu(\tilde f))>0$. Then by Birkhoff ergodic theorem,
  for $\nu$-almost every $z\in\T^2$ and every
  $\hat z\in(id\times\pi)^{-1}(z)$, it holds
  \begin{equation}
    \label{eq:hat-z-points-converge-to-infty}
    \hat f^{-n}(\hat z)\to -\infty\in\hat\A,\quad\text{as }
    n\to+\infty, 
  \end{equation}
  where $id\times\pi\colon\A=\T\times\R\to\T^2$ denotes the natural
  covering map.  Since $f$ is minimal, $\nu$ is a topological measure
  and hence, taking into account $ W^-$ is open, there exists a point
  $\hat z\in W^-$ satisfying
  \eqref{eq:hat-z-points-converge-to-infty}. This contradicts the fact
  that $\alpha_{\hat f}(\hat z)=\{+\infty\}$, for every
  $\hat z\in W^-$.

  Analogously, one gets a contradiction assuming
  $\pr{2}\big(\rho_\nu(\tilde f)\big)<0$.

  So, it just remains to consider the case
  $\pr{2}\big(\rho_\nu(\tilde f)\big)=0$. In such a case, as
  consequence of by Theorem~\ref{thm:atkinson} we know
  $\hat f\colon\hat\A\carr$ is non-wandering, \ie
  $\Omega(\hat f)=\hat\A$. In fact, let $\hat V\subset\A$ be a
  non-empty open set and suppose $\diam\hat V<1/4$; let us define
  $V:=(id\times\pi)(\hat V)\subset\T^2$. Since $id\times\pi$ is a
  covering map, $V$ is open and thus, $\nu(V)>0$. On the other hand,
  since we are assuming $\pr{2}\big(\rho_\nu(\tilde f)\big)=0$, we
  have
  \begin{displaymath}
    \int_{\T^2} \pr{2}\circ\Delta_{\tilde f}\dd\nu =0.
  \end{displaymath}
  Then, invoking Theorem~\ref{thm:atkinson} we know there exists
  $z\in V$ and $n\geq 1$ such that $f^n(z)\in V$ and
  $\abs{\Bs_f^n\big(\pr{2}\circ\Delta_{\tilde f}\big)(z)} =
  \abs{\pr{2}\circ\Delta_{\tilde f}^{(n)}(z)}<1/4$. This implies
  $\hat V\cap \hat f^{-n}(\hat V)\neq\emptyset$. So,
  $\A\subset\Omega(\hat f)$, then clearly we have
  $\Omega(\hat f)=\hat\A$. But the both sets $W^+$ and $W^-$ given by
  Theorem~\ref{thm:Patrice-index-larger-1} are wandering for $\hat f$,
  getting a contradiction.

  So, $i(\hat f,+\infty)\leq 1$. By
  Proposition~\ref{pro:powers-min-homeos-min-too}, one can easily
  adapt the previous reasoning for the general case, \ie where
  $i(\hat f^n,+\infty)\geq 2$, with $\abs{n}\geq 2$; and
  Claim~\ref{cla:index-less-than-2-at-infty} is proven.

  Now, putting together \eqref{eq:Lefschetz-fix-pnt-formula} and
  \eqref{eq:index-less-than-2-at-infty} we conclude that
  \begin{equation}
    \label{eq:index-iterates-equal-1}
    i\big(\hat f^k,-\infty\big)=i\big(\hat
    f^k,+\infty\big)=1,\quad\forall k\in\Z\setminus\{0\}.
  \end{equation}

  By an argument similar to that one we used to prove
  Claim~\ref{cla:index-less-than-2-at-infty}, one can show both fixed
  points $-\infty$ and $+\infty$ are not dissipative (\ie they are
  neither attractive nor repulsive). In fact, arguing by contradiction
  let us suppose, for instance, that there is a trapping neighborhood
  $V$ of $+\infty$, \ie $V\subset\hat\A$ is an open set such that
  $+\infty\in V$, $-\infty\not\in V$ and
  $\hat f(\overline{V})\subset V$. Following the very same reasoning
  we used to prove Claim~\ref{cla:index-less-than-2-at-infty} one can
  conclude in this case that
  \begin{displaymath}
    \pr{2}\big(\rho_\nu(\tilde f)\big)>0, \quad\forall\nu\in\M(f).
  \end{displaymath}
  However this last estimate is incompatible with the convexity of
  $\rho(\tilde f)$ and our hypothesis
  \eqref{eq:rot-set-inters-X-axes}. So, $-\infty$ and $+\infty$ are
  non-dissipative.

  Now, combining this last assertion,
  Theorem~\ref{thm:LeCalvez-Yoccoz-indexes} and
  \eqref{eq:index-iterates-equal-1} we show that $+\infty$ and
  $-\infty$ are indifferent (according to the classification of fixed
  points given in \S\ref{sec:class-plane-fixed-points}). In other
  words, if we define
  \begin{equation}
    \label{eq:V-neighborhood}
    V_r^{\pm(0,1)}:=P\Big(\overline{\bb{H}_r^{\pm(0,1)}}\Big)
    \sqcup\{\pm\infty\}\subset\hat\A,  
  \end{equation}
  it can be easily verified that $V_r^{\pm(0,1)}$ is a neighborhood of
  $\pm\infty$ and therefore,
  \begin{displaymath}
    \cc\left(\I_{\hat f}\big(V_r^{\pm(0,1)}\big),
      \pm\infty\right)\cap\partial V_r^{\pm(0,1)}\neq\emptyset,
    \quad\forall r\in\R,
  \end{displaymath}
  where $\partial V_r^{\pm(0,1)}=\T\times\{r\}\subset\A$.
  In particular, the set
  \begin{equation}
    \label{eq:Lambda-hat-vertical-def}
    \hat\Lambda_r^{\pm(0,1)}:=
    P^{-1}\left(\cc\Big(\I_{\hat f}\big(V_r^{\pm(0,1)}\big),
      \pm\infty\Big)\setminus\{\pm\infty\}\right)\subset\R^2 
  \end{equation}
  intersects the horizontal line $\ell_r^{\pm(0,1)}$, for every
  $r\in\R$. By Lemma~\ref{lem:lifting-unbounded-continuum}, every
  connected component of $\hat\Lambda_r^{\pm(0,1)}$ is unbounded, so
  it holds
  \begin{equation}
    \label{eq:Lambda-from-annulus-contained-Lambda-plane}
    \hat\Lambda_r^{u}\subset\Lambda_r^{u},\quad\forall
    r\in\R.
  \end{equation}
  Hence, \eqref{eq:Lambda-v-non-empty} is proven (see
  Figure~\ref{fig:hat-Lambda-in-Lambda} for an illustrative
  representation of the construction we have performed).

  \begin{figure}
  \centering \def\svgwidth{.9\columnwidth} 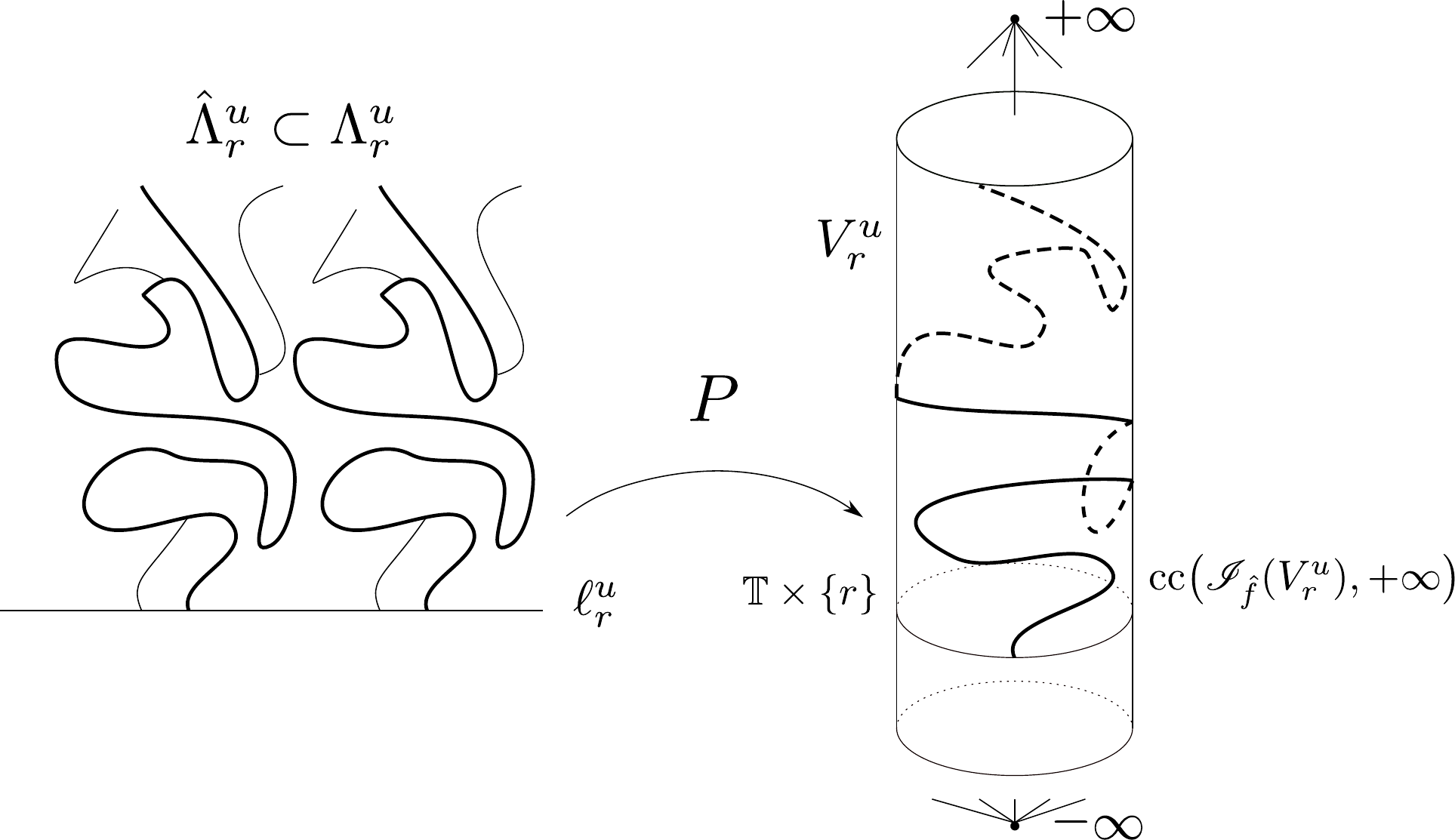
  \caption{$\hat{\Lambda}_r^u:=P^{-1}\big(\I_{\hat
      f}(V_r^u),+\infty\big)\subset\Lambda_r^u$, with $u=(0,1)$.}
    \label{fig:hat-Lambda-in-Lambda}
\end{figure}

  Assertion \eqref{eq:Lambda-v-minus-v-no-inter} easily follows from
  Proposition~\ref{pro:gott-hedl-applied-to-torus} and
  Proposition~\ref{pro:minimal-not-annular}. In fact, let us assume
  there exists
  $z\in\I_{\tilde f}\Big(\overline{\bb{H}^{(0,1)}_r}\Big)\cap
  \I_{\tilde f}\Big(\overline{\bb{H}^{(0,-1)}_{r'}}\Big)$, for some
  $r,r'\in\R$. Then, this implies
  \begin{displaymath}
    r\leq \pr{2}\Big(\Delta_{\tilde f}^{(n)}(z)\Big) =
    \sum_{j=0}^{n-1}\pr{2}\circ\Delta_{\tilde f}\big(\tilde
    f^j(z)\big) \leq -r', \quad\forall n\in\N.  
  \end{displaymath}
  So, by Proposition~\ref{pro:gott-hedl-applied-to-torus}, $f$ should
  be annular and by Proposition~\ref{pro:minimal-not-annular}, this is
  incompatible with minimality of $f$.

  Then, let us prove \eqref{eq:Lambda-v-in-no-strip} reasoning by
  contradiction. Suppose there exists a connected closed
  unbounded set $\Gamma\subset\Lambda_r^u$ such that
  $\Gamma\cap\ell_s^u=\emptyset$, for some real number
  $s>\inf\{\abs{\pr{2}(z)} : z\in\Gamma\}$. This means $\Gamma$ is
  contained in $\A_{r,s}^u$, where the strip $\A_{r,s}^u$ is given by
  \eqref{eq:Asv-strip-def}.

  By \eqref{eq:Lambda-v-equivariant-prop} we know that $\Lambda_r^u$
  is $T_{1,0}$-invariant. So,
  \begin{displaymath}
    \Gamma':=\bigcup_{n\in\Z}
    T_{1,0}^n(\Gamma)\subset\Lambda_r^u,
  \end{displaymath}
  and $\Gamma'$ is contained in $\A_{r,s}^u$ as well. Moreover, since
  $\I_{\tilde f}\big(\overline{\bb{H}}_r^u\big)$ is a closed
  $\tilde f$-invariant set, and
  $\Gamma'\subset\Lambda_r^u\subset \I_{\tilde
    f}\big(\overline{\bb{H}}_r^u\big)$, we conclude that
  \begin{equation}
    \label{eq:Gamma-prime-closure-in-I}
    \overline{\Gamma'}\subset \I_{\tilde
      f}\big(\overline{\bb{H}}_r^u\big)\cap\bb{A}_{r,s}^u. 
  \end{equation}

  On the other hand, since $\Gamma$ is unbounded, one sees that
  $\overline{\Gamma'}$ is contained in the interior of the strip
  $\A_{r-1,s+1}$ and separates both connected components of its
  boundary.

  Then let us write $\hat\Gamma':=P(\overline{\Gamma'})$ and
  $\hat\A_{r-1,s+1}^u:=P(\A_{r-1,s+1}^u)$, where $P\colon\R^2\to\A$
  denotes the covering map given by \eqref{eq:P-covering-map}. Observe
  that, since $\overline{\Gamma'}$ is $T_{1,0}$-invariant and
  $T_{1,0}$ generates the group of deck transformations of $P$,
  $\hat\Gamma'$ is a compact subset of $\A$.

  So $\hat\Gamma'\subset\hat\A_{r-1,s+1}^u$ and when $\hat\Gamma'$ is
  considered as a compact subset of
  $\hat\A=\A\sqcup\{-\infty,+\infty\}$, it separates the horizontal
  circle $P(\ell_{s+1}^u)$ and the point $-\sign(u)\infty\in\hat\A$,
  where $\sign(u)=1$, for $u=(0,1)$ and $\sign(u)=-1$, for $u=(0,-1)$.

  In the proof of \eqref{eq:Lambda-v-non-empty} we have shown that the
  set
  $\cc\left(\I_{\hat f}\big(V_{-s-1}^{-u}\big),-\sign(u)\infty\right)$
  intersects the boundary of $V_{-s-1}^{-u}$, where the set
  $V_{-s-1}^{-u}$ is given by \eqref{eq:V-neighborhood}, and so we
  have
  \begin{displaymath}
    \cc\left(\I_{\hat
        f}\big(V_{-s-1}^{-u}\big),-\sign(u)\infty\right)\cap\hat
    \Gamma'\neq\emptyset.
  \end{displaymath}
  By \eqref{eq:Lambda-from-annulus-contained-Lambda-plane}, this
  implies
  \begin{displaymath}
    \emptyset\neq\hat\Lambda_{-s-1}^{-u}\cap\overline{\Gamma'}
    \subset \I_{\tilde f}\big(\overline{\bb{H}_{-s-1}^{-u}}\big)\cap
    \I_{\tilde f}\big(\overline{\bb{H}_{r}^{u}}\big),
  \end{displaymath}
  contradicting \eqref{eq:Lambda-v-minus-v-no-inter}.

  % Notice that as byproduct of above argument one gets that equality
  % holds in inclusion
  % \eqref{eq:Lambda-from-annulus-contained-Lambda-plane}, \ie
  % $\hat\Lambda_r^u=\Lambda_r^u$.

  In order to prove \eqref{eq:Lambda-v-r-union-dense}, first notice
  that, as a consequence of \eqref{eq:Lambda-v-equivariant-prop}, the
  set $\bigcup_{r\in\R}\Lambda_r^{u}$ is $\Z^2$-invariant, \ie it is
  $T_{\boldsymbol{p}}$-invariant, for every $\boldsymbol{p}\in\Z^2$.

  On the other hand, since the set $\Lambda_r^{u}$ is defined as the
  union of unbounded connected components of an $\tilde f$-invariant
  closed set, it is $\tilde f$-invariant itself.

  So, the set $\bigcup_{r\in\R}\Lambda_r^{u}$ is invariant under the
  abelian subgroup
  $\langle\tilde f,\{T_{\boldsymbol{p}}\}_{\boldsymbol{p}\in\Z^2}
  \rangle<\Homeo+(\R^2)$ which acts minimally on $\R^2$. Then,
  $\bigcup_{r\in\R}\Lambda_r^{u}$ is dense in $\R^2$, as desired.

  Last assertion \eqref{eq:Lambda-v-not-disconnect} is a rather
  straightforward consequence of \eqref{eq:Lambda-v-minus-v-no-inter},
  \eqref{eq:Lambda-v-in-no-strip} and
  \eqref{eq:Lambda-v-r-union-dense}.

  In fact, first observe that, combining
  \eqref{eq:Lambda-v-minus-v-no-inter} and
  \eqref{eq:Lambda-v-r-union-dense} one easily shows that
  $\Lambda_r^u$ has empty interior.

  On the other hand, if $\R^2\setminus\Lambda_r^{u}$ were not
  connected, then there should exist a connected component
  $V\in\pi_0(\R^2\setminus\Lambda_r^u)$ such that
  $V\subset\bb{H}_r^{u}$.

  By \eqref{eq:Lambda-v-r-union-dense}, there exists $r'\in\R$ such
  that $\Lambda_{r'}^{-u}\cap V\neq\emptyset$. If $z_0$ is any point
  in $\Lambda_{r'}^{-u}\cap V$, then by
  \eqref{eq:Lambda-v-in-no-strip} we know that
  $\cc\big(\Lambda_{r'}^{-u},z_0\big)$ is not contained in the strip
  $\A^u_{r,-r'}$. Consequently, the connected set
  $\cc\big(\Lambda_{r'}^{-u},z_0\big)$ is not contained in $V$. So it
  intersects the boundary of $V$ which is contained in
  $\Lambda_r^u$. This contradicts \eqref{eq:Lambda-v-minus-v-no-inter}
  and \eqref{eq:Lambda-v-not-disconnect} is proved.
\end{proof}

\begin{figure}
    \centering
    \def\svgwidth{.7\columnwidth}
    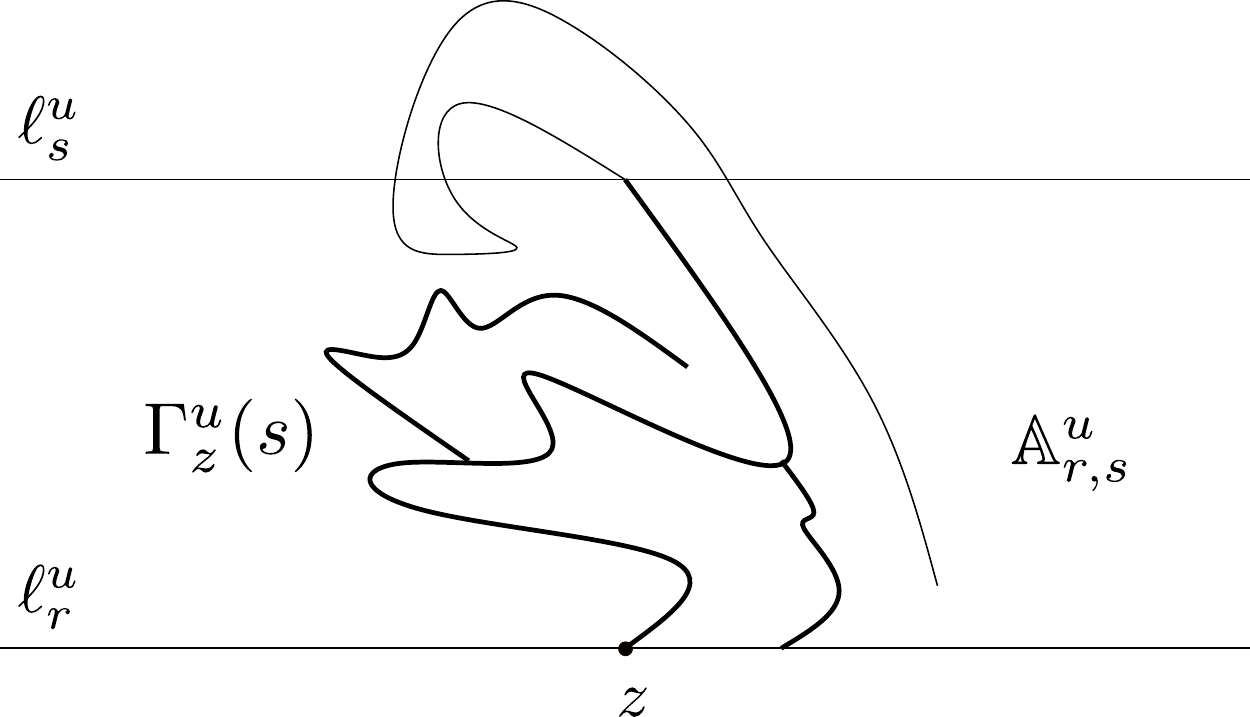
    \caption{Definition of the set $\Gamma_z^u(s)$ sets for $u=(0,1)$}
    \label{fig:Gamma-z-u-s-def}
\end{figure}

Let us continue assuming $\tilde f\in\Thomeo(\T^2)$ is a lift of a
minimal homeomorphism $f$ and condition
\eqref{eq:rot-set-inters-X-axes} holds. Fixing a real number $r$, for
each $z\in\Lambda_r^u\cap\ell_r^u$ and every $s>r$ we define the set
\begin{equation}
  \label{eq:Gamma-s-z-definition}
  \Gamma_z^u(s):=\cc\left(\Lambda_r^u\cap\A_{r,s}^{u}, 
    z\right),
\end{equation}
where the strip $\A_{r,s}^u$ is given by \eqref{eq:Asv-strip-def} (see
Figure \ref{fig:Gamma-z-u-s-def}).

As consequence of Theorem~\ref{thm:Lambda-vert-sets-main-properties},
we get the following result about the geometry of the sets
$\Gamma_z^u(s)$:
\begin{corollary}
  \label{cor:Lambda-vert-oscilations}
  For every $r\in\R$ and $u\in\{(0,1);(0,-1)\}$ the following
  conditions are satisfied:

  \begin{enumerate}[(i)]
  \item for every $z\in\Lambda_r^u\cap\ell_r^u$ and any $s>r$,
    \begin{equation}
      \label{eq:Gamma-z-s-disjoint-integer-trans}
      T_{1,0}^n\big(\Gamma_z^u(s)\big)\cap\Gamma_z^u(s)=\emptyset,
      \quad\forall n\in\Z\setminus\{0\}; 
    \end{equation}
  \item for any $s>r$, there exists a real number $D=D(f,s,r)>0$ such
    that
    \begin{equation}
      \label{eq:Gamma-z-unif-bounded-diamater}
      \diam\left(\pr{1}\big(\Gamma_z^u(s)\big)\right)
      \leq D, \quad\forall z\in\Lambda_r^u\cap\ell_r^u,
    \end{equation}
    and so, $\Gamma_z^u(s)$ is compact;
  \item for every
    $U\in\pi_0\big(\bb{H}_r^{u}\setminus\Lambda_r^{u}\big)$,
    \begin{equation}
      \label{eq:cc-Lambda-v-disj-hor-trans}
      T_{1,0}^n(U)\cap U=\emptyset, \quad\forall
      n\in\Z\setminus\{0\}.
    \end{equation}
  \end{enumerate}

  See Figure \ref{fig:Gamma-sets-horizontal-trans} for a graphical
  representation of these properties. 
\end{corollary}

\begin{figure}
    \centering
    \def\svgwidth{.7\columnwidth}
    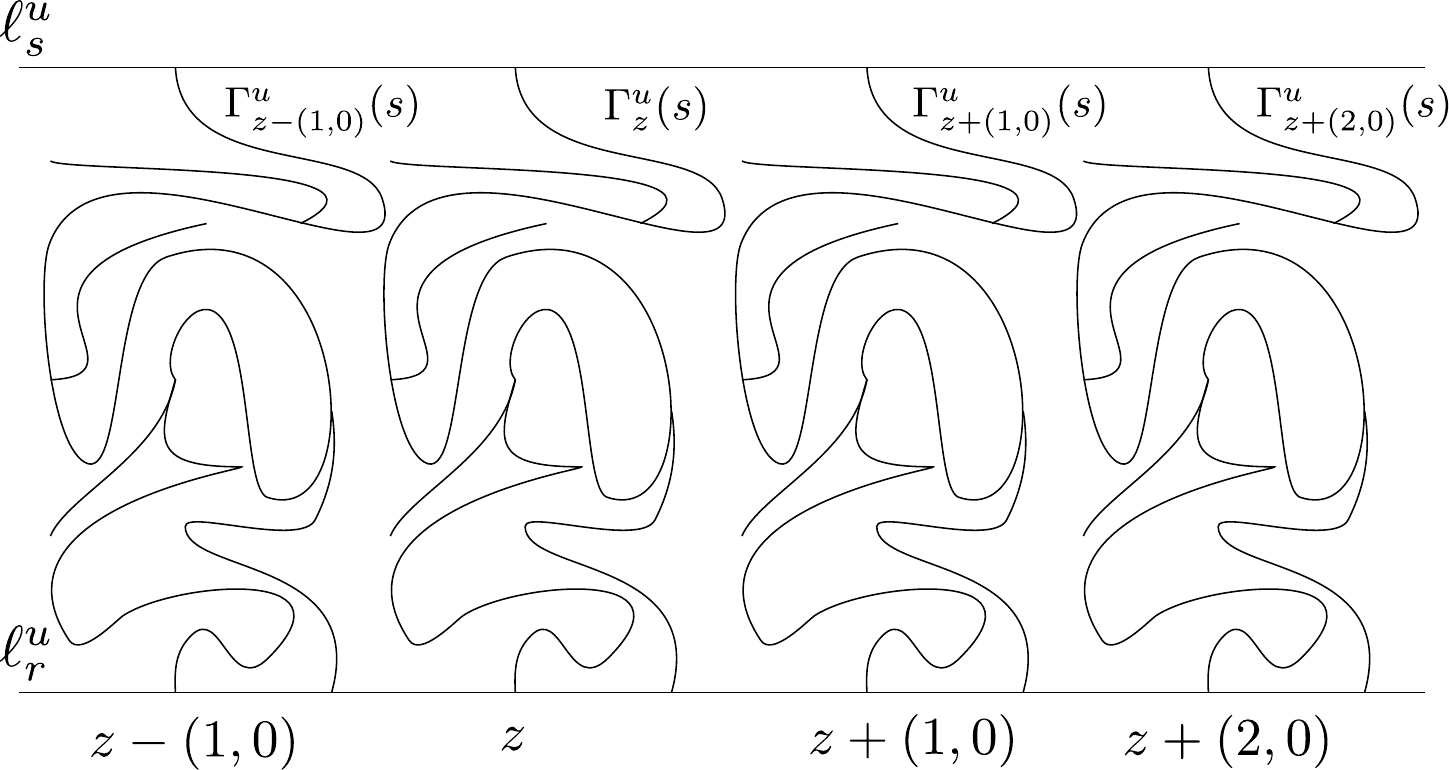
    \caption{$\Gamma_z^u(s)$ and their horizontal translations for for
      $u=(0,1)$}
    \label{fig:Gamma-sets-horizontal-trans}
\end{figure}

% \begin{figure}[h]
%   \centering
%   \includegraphics[height=5cm]{gamma}
%   \caption{Defining $\Gamma_z^u(s)$ sets}
%   \label{fig:Gamma-sets}
% \end{figure}

\begin{proof}
  Let us fix real numbers $s>r$ and let $z$ denote an arbitrary point
  in $\Lambda_r^u\cap\ell_r^u$. Reasoning by contradiction, let us
  start supposing $\diam\big(\pr{1}(\Gamma_z^u(s)\big)$ is
  infinite. Then, the set
  \begin{displaymath}
    \Gamma := \overline{\bigcup_{n\in\Z}
      T_{1,0}^n\big(\Gamma_z^u(s)\big)}
  \end{displaymath}
  disconnects $\R^2$ and since
  $\I_{\tilde f}\big(\overline{\bb{H}_r^u}\big)$ is
  $T_{1,0}$-invariant,
  $\Gamma\subset\I_{\tilde f}\big(\overline{\bb{H}_r^u}\big)$. So,
  $\I_{\tilde f}\big(\overline{\bb{H}_r^u}\big)$ disconnects $\R^2$
  contradicting \eqref{eq:Lambda-v-not-disconnect} of
  Theorem~\ref{thm:Lambda-vert-sets-main-properties}. Thus it holds
  \begin{equation}
    \label{eq:Gamma-z-bounded-diamater}
    \diam\Big(\pr{1}\big(\Gamma_z^u(s)\big)\Big)<\infty,\quad\forall
    z\in\Lambda_r^u\cap\ell_r^u,\  \forall s>r. 
  \end{equation}

  Now suppose \eqref{eq:Gamma-z-s-disjoint-integer-trans} is false,
  \ie there exist $z$, $s$ and $n$ such that
  \begin{displaymath}
    T_{n,0}\big(\Gamma_z^u(s)\big)\cap\Gamma_z^u(s)\neq\emptyset,
  \end{displaymath}
  with $n\neq 0$. 
  Then, since $\Lambda_r^u$ is $T_{1,0}$-invariant, the set
  \begin{displaymath}
    \bigcup_{m\in\Z}
    T_{n,0}^m\big(\Gamma_z^u(s)\big)\subset\Lambda_r^u 
  \end{displaymath}
  is connected, contains $\Gamma_z^u(s)$ and so, it coincides with
  $\Gamma_z^u(s)$. Since $n\neq 0$, we conclude
  $\diam(\pr{1}(\Gamma_z^u(s)))=\infty$, contradicting
  \eqref{eq:Gamma-z-bounded-diamater}.

  Property \eqref{eq:Gamma-z-unif-bounded-diamater} is a
  straightforwards consequence of
  \eqref{eq:Gamma-z-s-disjoint-integer-trans} and
  \eqref{eq:Gamma-z-bounded-diamater}. In fact, for fixed real numbers
  $s>r$ and any point $w\in\Lambda_r^u\cap\ell_r^u$, by
  \eqref{eq:Gamma-z-bounded-diamater} we know that
  $\diam(\pr{1}(\Gamma_w(s)))$ is finite. Then, given any
  $z\in\Lambda_r^u\cap\ell_r^u$, there exists a unique $n\in\Z$ such
  that
  $\pr{1}\circ T^n_{1,0}(w)\leq\pr{1}(z)<\pr{1}\circ
  T^{n+1}_{1,0}(w)$. By \eqref{eq:Gamma-z-s-disjoint-integer-trans},
  $T_{1,0}^n\big(\Gamma_w^u(s)\big)$ and
  $T_{1,0}^{n+1}\big(\Gamma_w^u(s)\big)$ are disjoint. Thus, it holds
  \begin{displaymath}
    \diam\Big(\pr{1}\big(\Gamma_z^u(s)\big)\Big)\leq
    \diam\Big(\pr{1}\big(\Gamma_w^u(s)\big)\Big) + 2,\quad\forall
    z\in\Lambda_r^u\cap\ell_r^u,
  \end{displaymath}
  and \eqref{eq:Gamma-z-unif-bounded-diamater} is proved.

  To prove \eqref{eq:cc-Lambda-v-disj-hor-trans}, we first need to
  introduce some definitions: for each connected component
  $U\in\pi_0\big(\bb{H}_r^{u}\setminus\Lambda_r^{u}\big)$ we define
  its \emph{boundary at level $r$} by
  \begin{equation}
    \label{eq:bound-at-level-r}
    \partial_r^u U:= (\partial U\cap\ell^u_r)\setminus\Lambda_r^u =
    (\overline{U}\cap\ell_r^u)\setminus\Lambda_r^u,
  \end{equation}
  where $\partial U$ denotes the boundary of $U$ in $\R^2$ (see
  Figure~\ref{fig:boundary-level-r-def}).

  \begin{figure}
  \centering \def\svgwidth{.7\columnwidth} 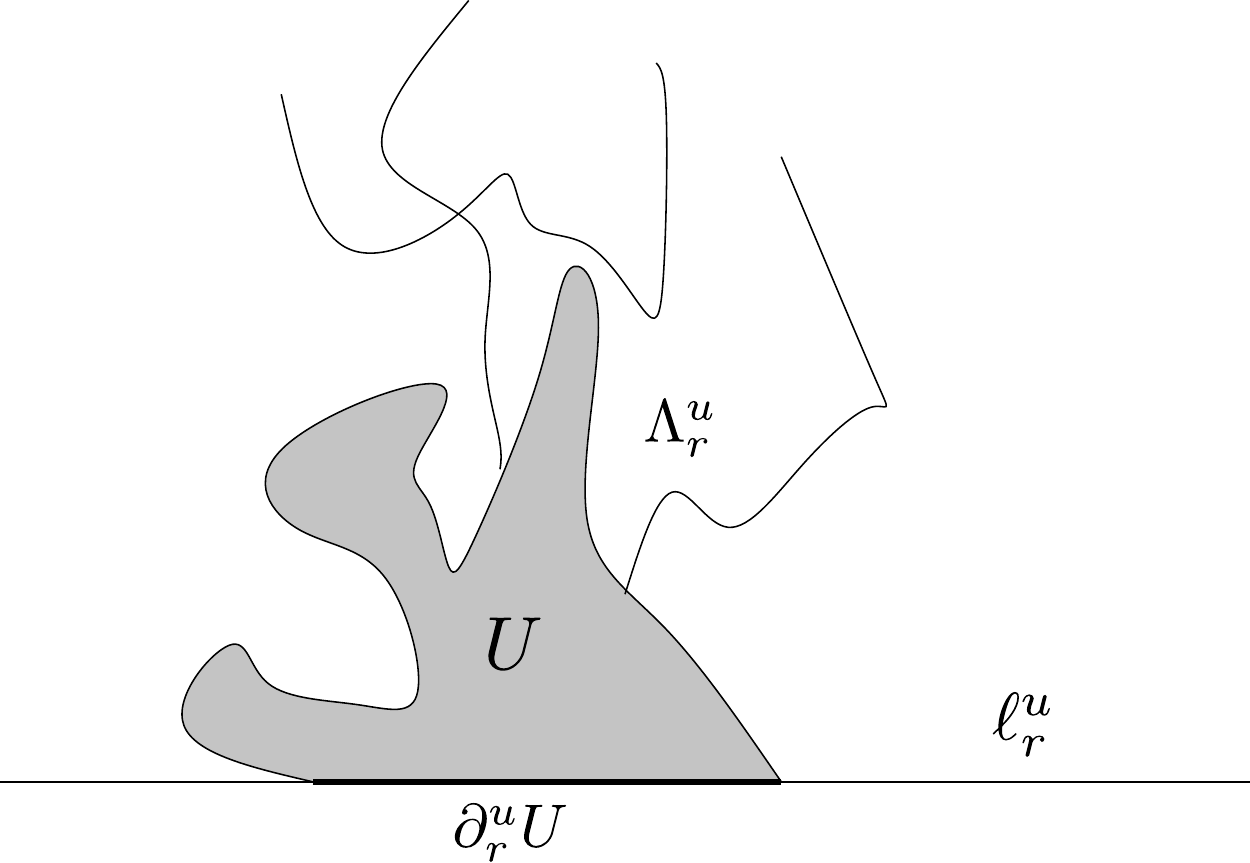
    \caption{$\partial_r^u U$ definition for a $u=(0,1)$.}
    \label{fig:boundary-level-r-def}
  \end{figure}

  So, the boundary at level $r$ operator satisfies the following
  property:
  \begin{claim}
    \label{cla:bound-at-level-r-properties}
    Every boundary at level $r$ is connected and non-empty. In other
    words, the operator
    \begin{displaymath}
      \partial_r^u\colon \pi_0\big(\bb{H}_r^u\setminus\Lambda_r^u\big)
      \to\pi_0\big(\ell_r^u\setminus\Lambda_r^u\big)
    \end{displaymath}
    given by \eqref{eq:bound-at-level-r} is a well-defined bijection.
  \end{claim}

  To prove our claim, let us consider an arbitrary point
  $z\in\Lambda_r^u\cap\ell_r^u$ whose existence is guaranteed by
  \eqref{eq:Lambda-v-non-empty} of Theorem~\ref{thm:Lambda-non-empty}
  and define $\Gamma_z^u:=\cc\big(\Lambda_r^u,z\big)$.  Then, notice
  that $\Gamma_z^u$ disconnects the half-plane $\bb{H}_r^u$. In fact,
  let us consider the one-point compactification of the plane
  $\wh{\R^2}:=\R^2\sqcup\{\infty\}$. Then, the closures
  $\wh{\ell_r^u}$ and $\wh{\Gamma_z^u}$ in $\wh{\R^2}$ are compact and
  connected, and the points $z$ and $\infty$ belong to both
  continua. On the other hand, since $\Lambda_r^u$ does not disconnect
  $\R^2$ and is $T_{1,0}$-invariant, then the intersection
  $\wh{\ell_r^u}\cap\wh{\Gamma_z^u}$ cannot be connected. Thus, by
  Theorem~\ref{thm:janiszewski}, $\wh{\ell_r^u}\cup\wh{\Gamma_z^u}$
  disconnects $\wh{\R^2}$ and consequently,
  $\bb{H}_r^u\setminus\Gamma_z^u$ is not connected. This implies
  $\partial_r^u U$ is connected for every
  $U\in\pi_0\big(\bb{H}_r^u\setminus\Lambda_r^u\big)$, and
  Claim~\ref{cla:bound-at-level-r-properties} is proved.

  On the other hand, since $\bb{H}_r^u$, $\ell_r^u$ and $\Lambda_r^u$
  are $T_{1,0}$-invariant, we observe the translation $T_{1,0}$
  naturally acts on $\pi_0\big(\bb{H}_r^u\setminus\Lambda_r^u\big)$
  and, consequently, on $\pi_0\big(\ell_r^u\setminus\Lambda_r^u\big)$,
  too. Moreover, it can be easily seen that the following diagram
  commutes
  \begin{displaymath}
    \xymatrix{\pi_0\big(\bb{H}_r^u\setminus\Lambda_r^u\big)
      \ar[r]^{T_{1,0}} \ar[d]_{\partial_r^u} &
      \pi_0\big(\bb{H}_r^u\setminus\Lambda_r^u\big)\ar[d]^{\partial_r^u}
      \\ 
      \pi_0\big(\ell_r^u\setminus\Lambda_r^u\big)\ar[r]^{T_{1,0}} &
      \pi_0\big(\ell_r^u\setminus\Lambda_r^u\big)},
  \end{displaymath}
  being the boundary at level $r$ operator $\partial_r^u$
  bijective. Hence, the actions of $T_{1,0}$ on both sets are
  conjugate and, clearly, there is no periodic orbit for
  $T_{1,0}\colon\pi_0\big(\ell_r^u\setminus\Lambda_r^u\big)\carr$. So,
  there is no periodic orbit for
  $T_{1,0}\colon\pi_0\big(\bb{H}_r^u\setminus\Lambda_r^u\big)\carr$
  either, and \eqref{eq:cc-Lambda-v-disj-hor-trans} is proved.
\end{proof}

The rest of this section is devoted to study the geometry of sets
$\Gamma_z^u(s)$ given by \eqref{eq:Gamma-s-z-definition}, assuming $f$
is not a pseudo-rotation and does not exhibit uniformly bounded
$v$-deviations. We will show that the connected sets $\Gamma_z^u(s)$
exhibit unbounded oscillations in the $v$ direction, as $s\to+\infty$:

\begin{theorem}
  \label{thm:Lambda-vert-oscillations}
  Let us assume $f$ is minimal and $\tilde f$ is a lift of $f$ such
  that its rotation set $\rho(\tilde f)$ intersects the upper and lower
  open semi-planes, \ie with our notation it holds
  \begin{equation}
    \label{eq:tilde-f-rot-set-transverse-X-axis}
    \rho(\tilde f)\cap\bb{H}_0^{(0,1)}\neq\emptyset
    \quad\text{and}\quad \rho(\tilde
    f)\cap\bb{H}_0^{(0,-1)}\neq\emptyset.
  \end{equation}
  On the other hand, we know there exist $v\in\Ss^1$ and $\alpha\in\R$
  such that inclusion \eqref{eq:rot-set-in-line} holds.

  If $f$ does not exhibits uniformly bounded $v$-deviations, then for
  every $r\in\R$, any $z\in\Lambda_r^u\cap\ell_r^u$ and $s>r$, it
  holds
  \begin{equation}
    \label{eq:Gamma-z-s-arb-larg-dev-plus-v-dir}
    \lim_{s\to+\infty} \sup_{w\in\Gamma_z^u(s)} \scprod{w}{v} =
    +\infty, 
  \end{equation}
  and
  \begin{equation}
    \label{eq:Gamma-z-s-arb-larg-dev-minus-v-dir}
    \lim_{s\to+\infty} \inf_{w\in\Gamma_z^u(s)} \scprod{w}{v} =
    -\infty.
  \end{equation}
\end{theorem}

The proof of Theorem~\ref{thm:Lambda-vert-oscillations} will follow
combining Theorem~\ref{thm:ergodic-deviations} and the following
\begin{lemma}
  \label{lem:Lambda-vert-oscilations}
  Under hypotheses of Theorem~\ref{thm:Lambda-vert-oscillations}, the
  following holds:
  \begin{equation}
    \label{eq:Gamma-z-s-arb-larg-dev-v-dir}
    \lim_{s\to+\infty} \sup_{w\in\Gamma_z^u(s)} \abs{\scprod{w}{v}} =
    +\infty.
  \end{equation}
\end{lemma}

\begin{proof}[Proof of Lemma~\ref{lem:Lambda-vert-oscilations}]
  For the sake of simplicity of notation, all along this proof we
  shall just write $\Lambda^+_r$ and $\Lambda^-_r$ instead of
  $\Lambda^{(0,1)}_r$ and $\Lambda^{(0,-1)}_r$, and do the same for
  any object that depends on the vectors $(0,1)$ or $(0,-1)$. Let us
  fix an arbitrary real number $r$. We will just prove
  \eqref{eq:Gamma-z-s-arb-larg-dev-v-dir} for $\Lambda_r^+$. The other
  case is completely analogous.

  By our hypothesis \eqref{eq:tilde-f-rot-set-transverse-X-axis} and
  Theorem~\ref{thm:mis-ziem-extreme-points}, there exist two ergodic
  measures $\mu^+,\mu^-\in\M(f)$ such that
  $\pr{2}\big(\rho_{\mu^+}(\tilde f)\big)>0$ and
  $\pr{2}\big(\rho_{\mu^-}(\tilde f)\big)<0$.

  By Birkhoff ergodic theorem, for $\mu^+$-almost every $x\in\T^2$ and
  any $\tilde x\in\pi^{-1}(x)$, it holds
  $\pr{2}(\tilde f^n(\tilde x))\to+\infty$ and
  $\pr{2}(\tilde f^{-n}(\tilde x))\to-\infty$, as
  $n\to+\infty$. Analogously, for $\mu^-$-almost every $x\in\T^2$ and
  any $\tilde x\in\pi^{-1}(x)$, it holds
  $\pr{2}(\tilde f^n(\tilde x))\to-\infty$ and
  $\pr{2}(\tilde f^{-n}(\tilde x))\to+\infty$, as $n\to+\infty$. In
  particular, this implies
  \begin{displaymath}
    \mu^+\big(\pi(\Lambda_r^+)\cup\pi(\Lambda_r^-)\big) =
    \mu^-\big(\pi(\Lambda_r^+)\cup\pi(\Lambda_r^-)\big) = 0. 
  \end{displaymath}
  
  Then, given any $\mu^+$-generic point
  $x^+\in\T^2\setminus\pi(\Lambda_r^+)$, we can find a point
  $z^+\in\pi^{-1}(x^+)$ such that
  \begin{equation}
    \label{eq:choosing-z+-fut-orb-above-2norm-Delta}
    \pr{2}\big(\tilde f^n(z^+)\big) > r+2\norm{\Delta_{\tilde
        f}}_{C^0}, \quad\forall n\geq 0.
  \end{equation}

  So we can define
  \begin{displaymath}
    U_n:=\cc\big(\bb{H}_r^+\setminus\Lambda_r^+,\tilde
    f^n(z^+)\big)\in\pi_0(\bb{H}_r^+\setminus\Lambda_r^+),
    \quad\forall n\geq 0.
  \end{displaymath}
  We claim that the sequence of boundaries at level $r$, \ie
  $(\partial_r^+(U_n))_{n\geq 0}$ where $\partial_r^+$ is given by
  \eqref{eq:bound-at-level-r}, exhibits bounded ``rotational
  deviations''. More precisely, we make the following

  \begin{claim}
    \label{cla:base-cc-bounded-rot-dev}
    There exists a constant $C>0$ such that for any sequence of real
    numbers $(a_n)_{n\geq 0}$ satisfying
    \begin{displaymath}
      (a_n,r)\in\partial_r^+(U_n), \quad\forall n\geq 0,
    \end{displaymath}
    it holds
    \begin{displaymath}
      \abs{a_n-n\frac{\alpha}{\pr{1}(v)}}\leq C,\quad\forall 
      n\geq 0.
    \end{displaymath}
  \end{claim}
  
  To prove our claim, first observe that, since we are assuming
  condition \eqref{eq:tilde-f-rot-set-transverse-X-axis}, $v$ is not
  vertical, \ie its first coordinate $\pr{1}(v)$ does not vanish,
  because the rotation set is not horizontal.

  Since the measure $\mu^-$ has total support on $\T^2$ and the set
  $U_0\subset\R^2$ is open, there is a point $w^-\in U_0$ such that
  $\pi(w^-)$ is $\mu^-$-generic and consequently, it holds
  $\pr{2}\big(\tilde f^n(w^-)\big)\to-\infty$, as $n\to+
  \infty$. Since $z^+,w^-\in U_0$ and $U_0$ is arc-wise connected,
  there is a continuous path $\gamma\colon [0,1]\to U_0$ connecting
  $w^-$ and $z^+$. Then for every $n$ sufficiently large,
  $\tilde f^n(w^-)$ belongs to semi-plane $\bb{H}_{-r}^-$ and so,
  there exists $t_n\in[0,1]$ such that
  $\tilde f^n\big(\gamma(t_n)\big)\in\partial_r^+(U_n)$. By inclusion
  \eqref{eq:rot-set-in-line} we know that
  \begin{displaymath}
    \frac{1}{n}\scprod{\tilde f^n\big(\gamma(t_n)\big) -
      \gamma(t_n)}{v}\to\alpha, \quad\text{as } n\to+ \infty.
  \end{displaymath}
  By Claim~\ref{cla:bound-at-level-r-properties} we know
  $\diam(\partial_r^+(U_n))\leq 1$ and since both points
  $\tilde f^n\big(\gamma(t_n)\big)$ and $(a_n,r)$ belong to
  $\partial_r^+(U_n)$ for $n$ sufficiently big, we conclude that
  \begin{equation}
    \label{eq:lim-an-div-n}
    \lim_{n\to+\infty}\frac{a_n}{n} = \frac{\alpha}{\pr{1}(v)}.
  \end{equation}
  To finish the proof of our claim, we use a classical sub-additive
  argument: let us show there exists $C>0$ such that
  \begin{equation}
    \label{eq:an-quasi-morphism-prop}
    \abs{a_{m+n}-a_m-a_n}\leq C, \quad\forall m,n\geq 0.
  \end{equation}

  To prove this, let us define the following total order on
  $\pi_0(\bb{H}_r^+\setminus\Lambda_r^+)$: given any pair of connected
  components $V,V'\in\pi_0(\bb{H}_r^+\setminus\Lambda_r^+)$, we write
  \begin{displaymath}
    V\prec V' \iff \pr{1}(w)<\pr{1}(w'),\quad\forall w\in\partial_r^+V,\
    \forall w'\in\partial_r^+V'.
  \end{displaymath}

  For each $n\geq 0$, let us write $p_n:=\floor{a_n-a_0}\in\Z$, where
  $\floor{\cdot}$ denotes the integer part operator. Then, observe
  that
  \begin{displaymath}
    T_{1,0}^{p_n-1}(U_0)\prec U_n\prec
    T_{1,0}^{p_n+1}(U_0),\quad\forall n\geq 0.
  \end{displaymath}

  Since $\tilde f$ commutes with every integer translation, preserves
  orientation and the point $z^+$ has been chosen such that
  \eqref{eq:choosing-z+-fut-orb-above-2norm-Delta} holds, we have
  \begin{displaymath}
    T_{1,0}^{p_n-1}(U_m)\prec U_{m+n} \prec T_{1,0}^{p_n+1}(U_m),
    \quad\forall m,n\geq 0. 
  \end{displaymath}

  In particular, this implies that $a_m+p_n-1<a_{m+n}<a_m+p_n-1$ and
  then,
  \begin{displaymath}
    \abs{a_{m+n}-a_m-a_n}\leq \abs{a_0}+1, \quad\forall m,n\geq 0.
  \end{displaymath}

  Then, Claim~\ref{cla:base-cc-bounded-rot-dev} easily follows from
  \eqref{eq:lim-an-div-n}, \eqref{eq:an-quasi-morphism-prop} and an
  elementary fact about sub-additive sequences (see for instance
  \cite[Lemma 2.2.1]{NavasGroupsCircleDiff}).

  Continuing with the notation we introduced in the proof of
  Claim~\ref{cla:base-cc-bounded-rot-dev} and since we are assuming
  $f$ exhibits unbounded $v$-deviations, by
  \eqref{item:large-deviations-future-past} of
  Proposition~\ref{pro:gott-hedl-applied-to-torus} we know that for
  every $M>0$ there exists $n=n(M)\geq 0$ such that
  \begin{displaymath}
    \abs{\scprod{\tilde f^n(z^+)}{v}-n\alpha}> M.
  \end{displaymath}
  Hence,
  \begin{equation}
    \label{eq:v-dir-dev-f-n-z+-from-an-r}
    \begin{split}
      \abs{\scprod{\tilde f^n(z^+)-(a_n,r)}{v}} &\geq
      \abs{\scprod{\tilde f^n(z^+)}{v}-a_n\pr{1}(v)}-\abs{r\pr{2}(v)} \\
      &\geq \abs{\scprod{\tilde f^n(z^+)}{v}-n\alpha}
      -\abs{\pr{1}(v)C} - \abs{r\pr{2}(v)}\\
      &\geq M-\abs{\pr{1}(v)C} - \abs{r\pr{2}(v)},
    \end{split}
  \end{equation}
  where $C$ is the constant given by
  Claim~\ref{cla:base-cc-bounded-rot-dev}.

  Finally, estimate \eqref{eq:Gamma-z-s-arb-larg-dev-v-dir} easily
  follows from Corollary~\ref{cor:Lambda-vert-oscilations},
  \eqref{eq:v-dir-dev-f-n-z+-from-an-r} and the fact that $M$ is
  arbitrary.
\end{proof}

Then, Theorem~\ref{thm:Lambda-vert-oscillations} will follow combining
Theorem~\ref{thm:ergodic-deviations} and
Lemma~\ref{lem:Lambda-vert-oscilations}.

\begin{proof}[Proof of Theorem~\ref{thm:Lambda-vert-oscillations}]
  By Lemma~\ref{lem:Lambda-vert-oscilations} we know that, for each
  $u\in\{(0,1),(0,-1)\}$, either
  \eqref{eq:Gamma-z-s-arb-larg-dev-plus-v-dir} or
  \eqref{eq:Gamma-z-s-arb-larg-dev-minus-v-dir} holds.

  Reasoning by contradiction and for the sake of concreteness, let us
  suppose that for every $r\in\R$ and every
  $z\in\Lambda_r^+\cap\ell_r^+$ condition
  \eqref{eq:Gamma-z-s-arb-larg-dev-minus-v-dir} does not hold. Notice
  here we continue using the notation we introduced in the proof of
  Lemma~\ref{lem:Lambda-vert-oscilations}. Analyzing the argument we
  used in the proof of Lemma~\ref{lem:Lambda-vert-oscilations}, this
  implies that
  \begin{equation}
    \label{eq:mu+-ae-point-dev+not--pos-iterates}
    \sup_{n\geq 0}\scprod{\Delta_{\tilde f}^{(n)}(x)}{v}-n\alpha =
    +\infty\quad\text{and}\quad \inf_{n\geq 0}\scprod{\Delta_{\tilde 
        f}^{(n)}(x)}{v}-n\alpha > -\infty,
  \end{equation} 
  for $\mu^+$-almost every $x\in\T^2$.

  On the other hand, by \eqref{eq:Lambda-v-minus-v-no-inter} of
  Theorem~\ref{thm:Lambda-vert-sets-main-properties} and invoking
  Lemma~\ref{lem:Lambda-vert-oscilations} for $\Lambda_r^-$, we
  conclude that \eqref{eq:Gamma-z-s-arb-larg-dev-minus-v-dir} holds
  and \eqref{eq:Gamma-z-s-arb-larg-dev-plus-v-dir} does not, for any
  $z\in\Lambda_r^-\cap\ell_r^-$ and every $r\in\R$.

  However, applying Theorem~\ref{thm:ergodic-deviations} to the
  ergodic system $(f,\mu^+)$ and the real function
  $\phi:=\scprod{\Delta_{\tilde f}}{v}-\alpha$, and taking into
  account \eqref{eq:mu+-ae-point-dev+not--pos-iterates}, we conclude
  that
  \begin{displaymath}
    \sup_{n\geq 0}\scprod{\Delta_{\tilde f}^{(-n)}(x)}{v}-n\alpha =
    +\infty\quad\text{and}\quad \inf_{n\geq 0}\scprod{\Delta_{\tilde 
        f}^{(-n)}(x)}{v}-n\alpha > -\infty,
  \end{displaymath}
  for $\mu^+$-a.e. $x\in\T^2$.

  Now, taking into account that
  $\pr{2}(\Delta_{\tilde f}^{(-n)}(x)) \to -\infty$, as $n\to+\infty$
  and $\mu^+$-a.e. $x\in\T^2$, we can repeat the argument we used to
  prove Lemma~\ref{lem:Lambda-vert-oscilations} for negative times and
  a $\mu^+$-generic point to show that
  \eqref{eq:Gamma-z-s-arb-larg-dev-plus-v-dir} holds for the set
  $\Lambda_r^-$ as well. Then, we have gotten a contradiction and
  Theorem~\ref{thm:Lambda-vert-oscillations} is proved.
\end{proof}

Combining Corollary~\ref{cor:Lambda-vert-oscilations} and
Theorem~\ref{thm:Lambda-vert-oscillations} one can easily strengthen
this last result getting the following

\begin{corollary}
  \label{cor:Lambda-vert-unbound-cc-oscilations}
  If $f$, $\tilde f$, $v$ and $r$ are as in
  Theorem~\ref{thm:Lambda-vert-oscillations} and $\Gamma$ is a closed
  connected unbounded subset of $\Lambda_r^u$, then it holds
  \begin{displaymath}
    - \inf_{z\in\Gamma} \scprod{z}{v} = \sup_{z\in\Gamma}
    \scprod{z}{v} = +\infty.
  \end{displaymath}
\end{corollary}

\section{Stable sets at infinity: parallel direction}
\label{sec:stab-set-infty-parall-dir}

Our next purpose consists in defining stable sets at infinity with
respect to the same direction of a supporting line of the rotation
set. More precisely, if $f\in\Homeo0(\T^2)$, $\tilde f\colon\R^2\carr$
is a lift of $f$ and we suppose there are $v\in\Ss^1$ and $\alpha$
such that the rotation set $\rho(\tilde f)$ is contained in the line
$\ell_\alpha^v$, we want to define stable sets at infinity with
respect to $v$, \ie associated to the families of semi-planes
$\bb{H}_r^v$ and $\bb{H}_r^{-v}$, with $r\in\R$. 

In such a case it might happen that there is no lift $\tilde f$ of $f$
such that the supporting line of $\rho(\tilde f)$ pass through the
origin, and therefore, if we naively defined
$\Lambda^v_r(\tilde f)=\I_{\tilde f}(\overline{\bb{H}_r^v})$, we would
get $\Lambda_r^v(\tilde f)=\emptyset$ for every $\tilde f$ and every
$r\in\R$.

To overpass this difficulty, we shall use the \emph{fiber-wise
  Hamiltonian skew-product} to define such stable sets at infinity.

\subsection{The fiber-wise Hamiltonian skew-product}
\label{sec:fiberw-hamilt-skew-prod}

Since we are assuming $f$ is minimal, by Theorem~\ref{thm:Oxtoby-Ulam}
we do not lose any generality assuming $f$ is area-preserving, \ie
$f\in\Symp0$ and let $\tilde f\in\Tsymp$ denote an arbitrary lift of
$f$.

Then, we define the \emph{fiber-wise Hamiltonian skew-product}
associated to $\tilde f$, which can be considered as a particular case
of the construction performed in \cite{KocRotDevPerPFree}. In very
rough words the main idea of this construction consists in splitting
our homeomorphism $f$ into a ``rotational'' part and a ``Hamiltonian''
or ``rotationless'' one. Doing that, the ``Hamiltonian'' part is
responsible by rotational deviations. The main technical advantage of
dealing with such skew-products is that an arbitrary point exhibits
bounded rotational deviations if and only if its orbit is bounded.

This novel object is certainly the main character of this work and
will play a fundamental role in the rest of the paper.

For the sake of simplicity, we fix some notations we shall use until
the end of the paper: we write $\tilde\rho:=\Flux(\tilde f)\in\R^2$
and $\rho:=\pi(\tilde\rho)=\Flux(f)\in\T^2$.

Then we define the map $H\colon\T^2\to\Tham$ by
\begin{displaymath}
  H_t:=\Ad_t\left(T_{\tilde\rho}^{-1}\circ\tilde f\right) = T_{\tilde
    t}^{-1}\circ T_{\tilde\rho}^{-1}\circ\tilde f\circ T_{\tilde t},
  \quad\forall t\in\T^2,\ \forall\tilde t\in\pi^{-1}(t),
\end{displaymath}
where $\Ad$ denotes $\T^2$-action given
by~\eqref{eq:Ad-Td-definition}.

Considering $H$ as a cocycle over the torus translation
$T_\rho\colon\T^2\carr$, one defines the \emph{fiber-wise Hamiltonian
  skew-product} associated to $f$ as the map
$F\colon\T^2\times\R^2\carr$ given by
\begin{displaymath}
  F(t,z):=\big(T_\rho(t),H_t(z)\big),\quad\forall
  (t,z)\in\T^2\times\R^2.
\end{displaymath}
Notice that $F$ depends just on $f$ and not on the chosen lift
$\tilde f$.

One can easily verify that
\begin{displaymath}
  F(t,z)=\left(t+\rho, z + \Delta_{\tilde f}\big(t+\pi(z)\big) -
    \tilde\rho\right), \quad\forall (t,z)\in\T^2\times\R^2,
\end{displaymath}
where $\Delta_{\tilde f}\in C^0(\T^2,\R^2)$ is the displacement
function given by \eqref{eq:Delta-cocycle}. We will use the following
usual notation for cocycles: given $n\in\Z$ and $t\in\T^2$, we write
\begin{displaymath}
  H_t^{(n)}:=
  \begin{cases}
    id_{\T^2},& \text{if } n=0;\\
    H_{t+(n-1)\rho}\circ H_{t+(n-2)\rho}\circ\cdots \circ H_t,&
    \text{if } n>0;\\
    H_{t+n\rho}^{-1}\circ\cdots\circ H_{t-2\rho}^{-1}\circ
    H_{t-\rho}^{-1},& \text{if } n<0.
  \end{cases}
\end{displaymath}
Then we have 
\begin{equation}
  \label{eq:iter-F}
  \begin{split}
    F^n(t,z)= \left(T_\rho^n(t),H_t^{(n)}(z)\right) =
    \left(t+n\rho,\Ad_t\big(T_{\tilde\rho}^{-n}\circ\tilde
      f^n\big)(z)\right),
  \end{split}
\end{equation}
for every $(t,z)\in\T^2\times\R^2$ and every $n\in\Z$.

Hence, if $\rho(\tilde f)$ has empty interior, there exist
$\alpha\in\R$ and $v\in\Ss^1$ such that inclusion
\eqref{eq:rot-set-in-line} holds, and from \eqref{eq:iter-F} it easily
follows that a point $z\in\R^2$ exhibits bounded $v$-deviations if and
only if
\begin{displaymath}
  \scprod{H_0^{(n)}(z) - z}{v}\leq M, \quad\forall n\in\Z, 
\end{displaymath}
where $M=M(z,f)$ denotes the positive constant given by
\eqref{eq:bound-v-deviations}.

\subsection{Fibered stable sets at infinity}
\label{sec:fibered-stable-sets}

Continuing with previous notation, let $\alpha\in\R$ and $v\in\Ss^1$
be such that property \eqref{eq:rot-set-in-line} holds. Notice that in
such a case, $\langle \tilde\rho,v\rangle =\alpha$.

Then, for each $r\in\R$ and each $t\in\T^2$ we define the
\emph{fibered $(r,v)$-stable set at infinity} by

\begin{equation}
  \label{eq:Lambda-r-t-def}
  \Lambda_{r}^v\big(\tilde f,
  t\big):=\pr{2}\bigg(\cc\Big(\{t\}\times\R^2\cap\I_F\big(\T^2\times
  \overline{\bb{H}_r^v}\big), \infty\Big)\bigg)\subset\R^2, 
\end{equation}
where $\bb{H}_r^v$ is the semi-plane given by
\eqref{eq:Hvr-half-spaces-def}, $\{t\}\times\R^2$ is naturally endowed
with the euclidean distance of $\R^2$, $\cc(\cdot,\infty)$ denotes the
union of unbounded components as defined in \ref{sec:maps-top-groups},
and $\pr{2}\colon\T^2\times\R^2\to\R^2$ is the projection on the
second coordinate.

Let us also define the \emph{$(r,v)$-stable set at infinity} as
\begin{displaymath}
  \Lambda_r^v\big(\tilde f\big):=\bigcup_{t\in\T^2}
  \{t\}\times\Lambda_r^v\big(\tilde f,t\big)\subset\T^2\times\R^2. 
\end{displaymath}

For the sake of simplicity, if there is no risk of confusion we shall
just write $\Lambda_r^v(t)$ and $\Lambda_r^v$ instead of
$\Lambda_r^v\big(\tilde f,t\big)$ and $\Lambda_r^v\big(\tilde f\big)$,
respectively.

Now we recall some results of \cite{KocRotDevPerPFree}:

\begin{theorem}[Theorem 3.4 in \cite{KocRotDevPerPFree}]
  \label{thm:Lambda-non-empty}
  Assuming inclusion \eqref{eq:rot-set-in-line} holds, for every
  $r\in\R$ the set $\Lambda_r^v$ is non-empty, closed and
  $F$-invariant.  Moreover, $\Lambda_r^v(t)\neq\emptyset$, for every
  $t\in\T^2$.

  Analogously, the same assertions hold for the $(r,-v)$-stable set at
  infinity.
\end{theorem}

The following result describes some elementary properties of
$(r,v)$-stable sets at infinity:
\begin{proposition}[Proposition 3.6 in \cite{KocRotDevPerPFree}]
  \label{pro:Lambda-equivariant-properties}
  For each $t\in\T^2$ and any $r\in\R$, the following properties hold:
  \begin{enumerate}[(i)]
  \item \label{eq:Lambda-r-monoton-inclusion}
    $\Lambda_r^v(t)\subset\Lambda_{r'}^v(t)$, for every $r'<r$;
  \item \label{eq:Lambda-r-lateral-continuity}
    \begin{displaymath}
      \Lambda_r^v(t)=\bigcap_{s<r} \Lambda_s^v(t);
    \end{displaymath}
  \item \label{eq:Lambda-t-conjug}
    $\Lambda_{r+\langle\tilde t,v\rangle}^v\big(t'-\pi\big(\tilde
    t\big)\big) = T_{\tilde t}\big(\Lambda_r^v(t')\big)$, for all
    $\tilde t\in\R^2$ and every $t'\in\T^2$;
  \item \label{eq:Lambda-Z2-translations}
    $T_{\boldsymbol{p}}\big(\Lambda_r^v(t)\big) =
    \Lambda^v_{r+\langle\boldsymbol{p},v\rangle}(t)$, for every
    $\boldsymbol{p}\in\Z^2$.
  \end{enumerate}
\end{proposition}

We shall need the following additional regularity result:
\begin{proposition}
  \label{pro:Lambda-r-semicont}
  Continuing with the same notation of
  Proposition~\ref{pro:Lambda-equivariant-properties}, the map
  $t\to\Lambda_r^v(t)$ is \emph{compactly upper semi-continuous,} \ie
  if $t_0\in\T$ and $U\subset\R^2$ is an open set such that
  $\Lambda_r^v(t_0)\subset U$ and $\R^2\setminus U$ is compact, then
  then there is a neighborhood $W(t_0)$ of $t_0$ in $\T^2$ such that
  \begin{displaymath}
    \Lambda_r^v(t)\subset U,\quad\forall t\in W(t_0).
  \end{displaymath}
\end{proposition}

\begin{proof}
  This is a straightforward consequence of the very definition of
  $(r,v)$-stable sets at infinity given by \eqref{eq:Lambda-r-t-def}.

  In fact, arguing by contradiction, let us suppose there exists a
  sequence $\{t_n\}_{n\geq 1}$ of points of $\T$ such $t_n\to t_0$ as
  $n\to+\infty$ and
  \begin{displaymath}
    \Lambda_r^v(t_n)\cap (\R^2\setminus U)\neq\emptyset, \quad\forall
    n\geq 1. 
  \end{displaymath}

  For each $n\geq 1$, let us consider a point
  $z_n\in \Lambda_r^v(t_n)\cap (\R^2\setminus U)$. Since the
  complement of $U$ is compact, there exists a sub-sequence
  $\{z_{n_j}\}_{j\geq 1}$ converging to a point
  $z_\infty\in\R^2\setminus U$. However, the whole set $\Lambda_r^v$
  is closed in $\T^2\times\R^2$ and thus,
  $z_\infty\in\Lambda_r^v(t_0)$ as well, contradicting the fact that
  $\Lambda_r^v(t_0)\subset U$.
\end{proof}

We also need the following
\begin{theorem}[Theorem 4.1 in \cite{KocRotDevPerPFree}]
  \label{thm:translates-of-Lambdas-dense}
  If $f\in\Symp0$ is periodic point free, \ie $\Per(f)=\emptyset$,
  then for every $t\in\T^2$ the set
  \begin{displaymath}
    \bigcup_{r\geq 0}\Lambda_{-r}^v(t)
  \end{displaymath}
  is dense in $\R^2$.
\end{theorem}

% Then, combining Theorems~\ref{thm:v-vs--v-rot-dev},
% Proposition~\ref{pro:gott-hedl-applied-to-torus} and Proposition 3.3
% of \cite{KocRotDevPerPFree}, we get the following result relating
% rotational deviations and the topology of stable sets at infinity
% for minimal homeomorphisms:

As a rather straightforward consequence of
Theorems~\ref{thm:gott-hedlund} and
\ref{thm:translates-of-Lambdas-dense} we get the following

\begin{corollary}
  \label{cor:rotational-dev-vs-topology-Lambda-sets}
  If $f\in\Symp0$ is minimal and does not exhibit uniformly bounded
  $v$-deviations, then, for every $r\in\R$ and any $t\in\T^2$, the
  following assertions hold:
  \begin{enumerate}[(i)]
  \item\label{cond:Lambda-v-minus-v-disjoint}
    $\Lambda_r^v(t)\cap\Lambda^{-v}_{r'}(t)=\emptyset$, for any
    $r'\in\R$;
  \item\label{cond:Lambda-v-empty-interior} $\Lambda_r^{v}(t)$ has
    empty interior;
  \item\label{cond:Lambda-v-not-disconnect} $\Lambda_r^v(t)$ does not
    disconnect $\R^2$, \ie $\R^2\setminus\Lambda_r^v(t)$ is connected.
  \end{enumerate}
\end{corollary}

\begin{proof}
  To prove \eqref{cond:Lambda-v-minus-v-disjoint} let us start
  assuming there exists
  $z\in\Lambda_r^v(t)\cap\Lambda_{r'}^{-v}(t)$. Thus, putting together
  \eqref{eq:iter-F} and \eqref{eq:Lambda-r-t-def} we get
  \begin{displaymath}
    r \leq \scprod{T_{\tilde t}^{-1}\circ T_{\tilde\rho}^{-n}\circ
      \tilde f^n\circ T_{\tilde t}(z)}{v} \leq -r', \quad\forall
    n\in\Z,\ \forall\tilde t\in\pi^{-1}(t),
  \end{displaymath}
  and consequently, if we define $\phi\in C^0(\T^2,\R)$ by
  $\phi(x):=\scprod{\Delta_{\tilde f}(x)}{v}$, then it holds
  \begin{displaymath}
    r\leq \Bs_f^n\phi\big(t+\pi(z)\big)\leq -r', \quad\forall n\in\Z.
  \end{displaymath}
  Then, since $f$ is minimal, by Theorem~\ref{thm:gott-hedlund} we
  conclude that $\phi$ is a continuous coboundary for $f$, and hence,
  $f$ exhibits uniformly bounded $v$-deviations, contradicting our
  assumption.

  Property \eqref{cond:Lambda-v-empty-interior} is a straightforward
  consequence of Theorem~\ref{thm:translates-of-Lambdas-dense} and
  property \eqref{cond:Lambda-v-minus-v-disjoint}.

  Finally, in order to show \eqref{cond:Lambda-v-not-disconnect} let
  us suppose there exists $r\in\R$ and $t\in\T$ such that
  $\Lambda_r^v(t)$ disconnects $\R^2$. So, there exists
  $U\in\pi_0\big(\R^2\setminus\Lambda_r^v(t)\big)$ such that
  $U\cap\overline{\bb{H}^{-v}_{-r}}=\emptyset$. Then one can easily
  check that the boundary of $U$ is completely contained in
  $\Lambda_r^v(t)$ and thus, $\overline{U}$ is contained in
  $\Lambda_r^v(t)$ as well, contradicting property
  \eqref{cond:Lambda-v-empty-interior}.

\end{proof}

\subsection{Rotational deviations and the spreading property}
\label{sec:rotat-devi-spreading}

From now on and until the end of this section, we shall assume
$f\in\Symp0$ is a minimal homeomorphism such that there is a lift
$\tilde f\in\Tsymp$ satisfying
\begin{gather}
  \label{eq:flux-assumptions}
  \tilde\rho=\Flux(\tilde f)=(\tilde\rho_1,0)\in\R^2 \\
  \label{eq:rot-set-assumptions}
  \rho(\tilde f)\cap\bb{H}_0^{(0,1)}\neq\emptyset \quad \text{and}
  \quad \rho(\tilde f)\cap\bb{H}_0^{(0,-1)}\neq\emptyset.
\end{gather}
Notice that, since $f$ is minimal, by
Corollary~\ref{cor:minimal-homeos-rho-cap-Q2-empty} we know
$\tilde\rho_1\in\R\setminus\Q$.

Then, if $F\colon\T^2\times\R^2\carr$ denotes the fiber-wise
Hamiltonian skew-product induced by $\tilde f$, the closed set
$\T\times\{0\}\times\R^2\subset\T^2\times\R^2$ is
$F$-invariant. Making some abuse of notation and for the sake of
simplicity, we shall write $F$ to denote the restriction of the
fiber-wise Hamiltonian skew-product to this set. More precisely, from
now on we have $F\colon\T\times\R^2\carr$ where
\begin{displaymath}
  F(t,z)=\left(t+\rho_1, z+\Delta_{\tilde f}\big((t,0) + \pi(z)\big) -
    (\tilde\rho_1,0)\right), \quad\forall (t,z)\in\T\times\R^2,
\end{displaymath}
and $\rho_1:=\pi(\tilde\rho_1)$.

In a joint work with Koropecki~\cite{KocKorFoliations}, we introduced
the notion of topological \emph{spreading}, which is stronger than
topological mixing:

\begin{definition}
  \label{def:spreading}
  A homeomorphism $h\in\Homeo{}(\T^2)$ is said to be \emph{spreading}
  when for any lift $\tilde h\in\widetilde{\Homeo{}}(\T^2)$, any
  $R,\varepsilon>0$ and any non-empty open set $U\subset\R^2$, there
  exists $N\in\N$ such that for every $n\geq N$, there exists a point
  $z_n\in\R^2$ such that $\tilde h^n(U)$ is $\varepsilon$-dense in the
  ball $B_{R}(z_n)$.
\end{definition}

Motivated by this notion, we will prove the following theorem which
the main result of this section:

\begin{theorem}
  \label{thm:unbounded-dev-implies-spreading}
  Let us suppose $f$ does not exhibit uniformly bounded
  $v$-deviations. Then, for every pair of non-empty open sets
  $U,V\subset\R^2$, there exists $N\in\N$ such that for every $t\in\T$
  it holds
  \begin{displaymath}
    F^n\big(\{t\}\times U\big)\cap \T\times V\neq\emptyset,
    \quad\forall n\geq N. 
  \end{displaymath}
\end{theorem}

We shall divide the proof of
Theorem~\ref{thm:unbounded-dev-implies-spreading} in several
lemmas. Notice that without loss of generality we can assume the open
set $V$ in Theorem~\ref{thm:unbounded-dev-implies-spreading} is
bounded.

\begin{lemma}
  \label{lem:choosing-r-inters-U-V}
  There exists $r\in\R$ such that
  \begin{displaymath}
    \Lambda_r^v(t)\cap V\neq\emptyset\quad \text{and}\quad
    \Lambda_r^{-v}(t)\cap V\neq\emptyset,\quad\forall t\in\T.
  \end{displaymath}
\end{lemma}

\begin{proof}
  This is a straightforward consequence of
  Theorem~\ref{thm:translates-of-Lambdas-dense}, properties
  \eqref{eq:Lambda-r-monoton-inclusion} and \eqref{eq:Lambda-t-conjug}
  of Proposition~\ref{pro:Lambda-equivariant-properties}, and
  compactness of $\T$.
\end{proof}

From now on we fix a real number $r\in\R$ such that the conclusion of
Lemma~\ref{lem:choosing-r-inters-U-V} holds.

Since we are assuming $V$ is bounded and $f$ does no exhibit uniformly
bounded $v$-deviations, by Theorem~\ref{thm:janiszewski} we know that
the set $\Lambda_r^v(t)\cup\Lambda_r^{-v}(t)\cup V$ disconnects
$\R^2$, for every $t\in\T$. For the sake of simplicity of notation,
for each $t\in\T$ let us write
\begin{displaymath}
  \Gamma_t:=\Lambda_r^v(t)\cup\Lambda_r^{-v}(t)\cup V.
\end{displaymath}

Now for each $\varepsilon >0$, we define the following set:
\begin{displaymath}
  \Z^2(v,\varepsilon):=\left\{\boldsymbol{p}\in\Z^2 :
    \abs{\scprod{\boldsymbol{p}}{v}}\leq\varepsilon \right\}.
\end{displaymath}
Notice that by \eqref{eq:rot-set-assumptions}, the rotation set
$\rho(\tilde f)$ is not a horizontal segment, so $v$ is not vertical,
\ie $\pr{1}(v)\neq 0$. By classical arguments about approximations by
rational numbers one can easily get the following
\begin{lemma}
  \label{lem:Z2v-bounded-gaps}
  For each $\varepsilon>0$, there exists $N\in\N$ such that for every
  $n\in\Z$, there exist $p\in\{n,n+1,\ldots,n+N\}$ and
  $\boldsymbol{p}\in\Z^2(v,\varepsilon)$ satisfying
  $\pr{2}(\boldsymbol{p})=p$.
\end{lemma}

\begin{proof}
  This easily follows from
  Corollary~\ref{cor:quas-periodicity-torus-translations} and the fact
  that $v$ is not horizontal.
\end{proof}

\begin{lemma}
  \label{lem:suff-large-trans-in-diff-connected-comp}
  For every $t\in\T^2$ there exist two connected components
  $W^+_t,W^-_t\in\pi_0(\R^2\setminus\Gamma_t)$ such that the following
  property holds: for every $z\in\R^2\setminus\Gamma_t$, there exist
  $\varepsilon>0$ and $M\in\N$ such that
  \begin{displaymath}
    T_{\boldsymbol{p}}(z)\in W^+_t, \quad\text{and}\quad
    T_{\boldsymbol{p}}^{-1}(z)\in W^-_t,
  \end{displaymath}
  for every $\boldsymbol{p}\in\Z^2(v,\varepsilon)$ satisfying
  $\pr{2}(\boldsymbol{p})>M$.
\end{lemma}

\begin{proof}
  Let $z$ be any point in $\R^2\setminus\Gamma_t$. By statement
  \eqref{eq:Lambda-r-lateral-continuity} of
  Proposition~\ref{pro:Lambda-equivariant-properties}, there exists
  $\varepsilon>0$ such that
  $z\not\in\Lambda_{r-2\varepsilon}^v(t) \cup
  \Lambda_{r-2\varepsilon}^{-v}(t)$, and so we can consider the
  positive number
  \begin{displaymath}
    \delta:=\frac{1}{2}d\big(z,\Lambda_{r-2\varepsilon}^v(t) \cup
    \Lambda_{r-2\varepsilon}^{-v}(t)\big).
  \end{displaymath}
  Hence, we have
  \begin{equation}
    \label{eq:delta-choosing-p-translations}
    T_{\boldsymbol{p}}\big(B_\delta(z)\big)\cap
    \big(\Lambda_r^v(t)\cup\Lambda_r^{-v}(t)\big)=\emptyset,  
    \quad\forall \boldsymbol{p}\in\Z^2(v,\varepsilon). 
  \end{equation}

  Now, since by Corollary
  \ref{cor:rotational-dev-vs-topology-Lambda-sets} the set
  $\Lambda_{r-2\varepsilon}^v(t) \cup
  \Lambda_{r-2\varepsilon}^{-v}(t)$ has empty interior and does not
  disconnect $\R^2$, for each $\boldsymbol{m}\in\Z^2$ we can find a
  continuous path $\gamma_{\boldsymbol{m}}\colon [0,1]\to\R^2$ such
  that $\gamma_{\boldsymbol{m}}(0)=z$,
  $\gamma_{\boldsymbol{m}}(1)\in
  T_{\boldsymbol{m}}\big(B_\delta(z)\big)$ and
  \begin{equation}
    \label{eq:gamma-curve-not-inters-Lambada}
    \gamma_{\boldsymbol{m}}(s)\not\in\Lambda_{r-2\varepsilon}^v(t)
    \cup\Lambda_{r-2\varepsilon}^{-v}(t), \quad\forall s\in [0,1].
  \end{equation}

  Let $N$ denote the natural number given by
  Lemma~\ref{lem:Z2v-bounded-gaps} for $v$ and $\varepsilon$ as above
  and consider the set
  \begin{displaymath}
    A:=\left\{\boldsymbol{p}\in\Z^2(v,2\varepsilon) :
      0\leq \pr{2}(\boldsymbol{p})\leq N\right\}. 
  \end{displaymath}
  Since $A$ is a non-empty finite set and $V$ is bounded, we can
  define the real number
  \begin{equation}
    \label{eq:M-def-floor-V}
    M:=1+\sup_{\boldsymbol{m}\in A}\sup_{s\in [0,1]}
    \abs{\pr{2}\big(\gamma_{\boldsymbol{m}}(s)\big)} + \sup_{w\in V}
    \abs{\pr{2}(w)} <\infty.
  \end{equation}

  Consider any two points
  $\boldsymbol{m},\boldsymbol{p}\in\Z^2(v,\varepsilon)$ satisfying
  $0\leq \pr{2}(\boldsymbol{p})-\pr{2}(\boldsymbol{m})\leq N$ and
  $\pr{2}(\boldsymbol{m})>M$. Thus, we have
  $T_{\boldsymbol{m}}\circ\gamma_{\boldsymbol{p}-\boldsymbol{m}}$ is a
  continuous path connecting $T_{\boldsymbol{m}}(z)$ and the ball
  $T_{\boldsymbol{p}}\big(B_\delta(z)\big)$; since
  $\boldsymbol{m}\in\Z^2(v,\varepsilon)$ and
  \eqref{eq:gamma-curve-not-inters-Lambada} holds, the image of
  $T_{\boldsymbol{m}}\circ\gamma_{\boldsymbol{p}-\boldsymbol{m}}$ does
  not intersect $\Lambda_r^v(t)\cup\Lambda_r^{-v}(t)$; and since
  $\boldsymbol{p}-\boldsymbol{m}\in A$ and $\pr{2}(\boldsymbol{m})>M$,
  invoking \eqref{eq:M-def-floor-V} we conclude that the image of
  $T_{\boldsymbol{m}}\circ\gamma_{\boldsymbol{p}-\boldsymbol{m}}$ does
  not intersect $V$ either. So, by
  \eqref{eq:delta-choosing-p-translations}, $T_{\boldsymbol{m}}(z)$
  and $T_{\boldsymbol{p}}(z)$ belong to the same connected component
  of $\R^2\setminus\Gamma_t$.

  Hence, choosing any $\boldsymbol{m}\in\Z^2(v,\varepsilon)$
  satisfying $\pr{2}(\boldsymbol{m})>M$, we can define
  \begin{displaymath}
    W^+_t:=\cc\big(\R^2\setminus\Gamma_t,T_{\boldsymbol{m}}(z)\big),
  \end{displaymath}
  and combining the last argument with
  Lemma~\ref{lem:Z2v-bounded-gaps}, one can shows that
  $T_{\boldsymbol{p}}(z)\in W^+_t$, for any
  $\boldsymbol{p}\in\Z^2(v,\varepsilon)$ such that
  $\pr{2}(\boldsymbol{p})>M$.

  To prove the uniqueness of $W_t^+$, let $w$ be any other point in
  $\R^2\setminus\Gamma_t$. Let $\varepsilon'\leq\varepsilon$ be any
  positive number such that
  $w\in\R^2\setminus(\Lambda_{r-2\varepsilon'}^v(t) \cup
  \Lambda_{r-2\varepsilon'}^{-v}(t))$. So, since by Corollary
  \ref{cor:rotational-dev-vs-topology-Lambda-sets} the set
  $\Lambda_{r-2\varepsilon'}^v(t)\cup\Lambda_{r-2\varepsilon'}^{-v}(t)$
  does not disconnect $\R^2$, then there exists a continuous path
  $\gamma\colon [0,1]\to\R^2$ such that $\gamma(0)=w$, $\gamma(1)=z$
  and
  \begin{displaymath}
    \gamma(s)\not\in
    \Lambda_{r-2\varepsilon'}^v(t)\cup\Lambda_{r-2\varepsilon'}^{-v}(t),
    \quad\forall s\in[0,1].
  \end{displaymath}
  So, the image of $T_{\boldsymbol{p}}\circ\gamma$ does not intersect
  $\Lambda_r^v(t)\cup\Lambda_r^{-v}(t)$, for any
  $\boldsymbol{p}\in\Z^2(v,\varepsilon')$, and does not intersect $V$
  either, provided $\pr{2}(\boldsymbol{p})$ is sufficiently
  large. Thus, $T_{\boldsymbol{p}}(w)\in W_t^+$ for such a
  $\boldsymbol{p}$ and uniqueness of $W_t^+$ is proven.

  Finally, defining
  $W_t^-:=\cc\big(\R^2\setminus\Gamma_t,T_{\boldsymbol{m}}^{-1}(z)\big)$
  for $\boldsymbol{m}$ as above, one can easily show that analogous
  properties hold.
\end{proof}

In order to finish the proof of
Theorem~\ref{thm:unbounded-dev-implies-spreading}, we fix a non-empty
open set $U\subset\R^2$. Without loss of generality we can assume that
$U$ is bounded, connected and
\begin{equation}
  \label{eq:U-disjoint-Lambda-v-Lambda-minus-v}
  \overline{U}\cap\left(\Lambda_r^v(0)\cup\Lambda_r^{-v}(0)\right)
  =\emptyset,  
\end{equation}
where $r$ is the real number we fixed after
Lemma~\ref{lem:choosing-r-inters-U-V}. Since $\overline U$ is compact
and $\Lambda_r^v\cup\Lambda_r^{-v}$ is contained in
$\R^2\setminus\overline U$, by Proposition~\ref{pro:Lambda-r-semicont}
we know the maps $t\mapsto \Lambda_r^v(t)$ and
$t\mapsto \Lambda_r^{-v}(t)$ are both compactly upper
semi-continuous. Thus, there is $\eta>0$ such that
\begin{equation}
  \label{eq:choosing-eta-disjoint-Lambda}
  B_\eta(0)\times\overline{U}\cap
  \big(\Lambda_r^v\cup\Lambda_r^{-v}\big) =\emptyset,
\end{equation}
where $B_\eta(0)$ denotes the $\eta$-ball centered at $0\in\T$ with
respect to the distance $d_\T$.

Now, by minimality of $f$ and recalling that
$\rho_1=\pi(\tilde\rho_1)$ where $\tilde\rho_1\in\R\setminus\Q$, there
exists $k\geq 1$ such that
\begin{equation}
  \label{eq:f-k-iterates-cover-T2}
  \bigcup_{i=0}^k f^i\big(\pi(U)\big)=\T^2,\quad\text{and}\quad
  \bigcup_{i=0}^k T_{\rho_1}^i\big(B_\eta(0)\big) = \T.
\end{equation}

Let us define
\begin{displaymath}
  \ms{U}:=\bigcup_{i=0}^k F^i(B_\eta(0)\times U)\subset\T\times\R^2,
\end{displaymath}
and for every $t\in\T$, let us write
\begin{displaymath}
  \ms{U}(t):=\pr{2}\left(\ms{U}\cap\{t\}\times\R^2\right)\subset\R^2. 
\end{displaymath}

Notice that, by \eqref{eq:choosing-eta-disjoint-Lambda},
$\overline{\ms{U}}\cap (\Lambda_r^v\cup\Lambda_r^{-v})=\emptyset$. So,
by \eqref{eq:Lambda-r-lateral-continuity} of
Proposition~\ref{pro:Lambda-equivariant-properties}, there exists
$\varepsilon>0$ such that
\begin{displaymath}
  \overline{\ms{U}(t)_\varepsilon}\cap
  \big(\Lambda_{r-2\varepsilon}^v(t)\cup\Lambda_{r-2\varepsilon}^{-v}(t)
  \big) = \emptyset, \quad\forall t\in\T,
\end{displaymath}
where $(\cdot)_\varepsilon$ denotes the $\varepsilon$-neighborhood
given by \eqref{eq:epsilon-nbh-def}.

On the other hand, by our hypothesis \eqref{eq:rot-set-assumptions}
there exists $\tilde\rho^+\in\rho(\tilde f)$ such that
$\pr{2}(\tilde\rho^+)>0$.  So, let $\delta$ be a positive number given
by Theorem~\ref{thm:bounded-rot-deviations-for-minimal-homeos} for
$f$, $\tilde\rho^+$ and $\varepsilon/2$. Without loss of generality we
can assume $\delta<\min\{\eta,\frac{\varepsilon}{4}\}$, where $\eta$
was chosen in \eqref{eq:choosing-eta-disjoint-Lambda}.

Now, consider the translation
$T:=T_{\rho_1,\pi\tilde\rho^+}\colon\T\times\T^2=\T^3\carr$ and the
visiting time set $\tau:=\tau(0,B_\delta(0),T)$ defined in Corollary
\ref{cor:quas-periodicity-torus-translations}.

Then, by Theorem~\ref{thm:bounded-rot-deviations-for-minimal-homeos}
and \eqref{eq:f-k-iterates-cover-T2} we get that, for each $n\in\tau$
there exist $z_n\in U$, $j_n\in\{0,1,\ldots,k\}$,
$\boldsymbol{p}_n\in\Z^2$ and $q_n\in\Z$ such that
$\norm{\boldsymbol{p}_n-n\tilde\rho^+}<\delta$,
$\abs{q_n-n\tilde\rho_1}<\delta$ and
\begin{displaymath}
  \norm{\tilde f^n\big(\tilde f^{j_n}(z_n)\big) - \tilde f^{j_n}(z_n) -
    n\tilde\rho^+}<\frac{\varepsilon}{2},
\end{displaymath}
or equivalently,
\begin{equation}
  \label{eq:z-n-j-n-good-translations}
  F^n\big(F^{j_n}(0,z_n)\big)\in \{(j_n+ n)\rho_1\}\times
  T_{q_n,0}^{-1} \circ
  T_{\boldsymbol{p}_n}(\ms{U}(j_n\rho_1)_\varepsilon).   
\end{equation}
Observe that
\begin{equation}
  \label{eq:p-n-q-n-in-Z2ve}
  \begin{split}
    &\abs{\scprod{v}{\boldsymbol{p}_n-(q_n,0)}}=
    \abs{\scprod{v}{\boldsymbol{p}_n-n\tilde\rho^++
        n(\tilde\rho_1,0)-(q_n,0)}} \\
    &\leq \abs{\scprod{v}{\boldsymbol{p}_n-n\tilde\rho^+}} +
    \abs{\scprod{v}{n(\tilde\rho_1,0)-(q_n,0)}}\leq 2\delta<
    \frac{\varepsilon}{2}.
  \end{split}
\end{equation}
So, in particular, this implies
$\boldsymbol{p}_n-(q_n,0)\in\Z^2(v,\varepsilon)$.

Then observe that since we are assuming $U$ is connected,
$\ms{U}(t)_\varepsilon$ has finitely many connected components for
every $t\in\T$, and then we can apply
Lemma~\ref{lem:suff-large-trans-in-diff-connected-comp} to conclude
there exists $M>0$ such that
\begin{equation}
  \label{eq:Fi-above-in-gaps}
  F^i\Big(\{t\}\times
  T_{\boldsymbol{p}}\big(\ms{U}(t)_\varepsilon\big)\Big)\subset 
  \{t+i\rho_1\}\times W_{t+i\rho_1}^+, 
\end{equation}
for every $\boldsymbol{p}\in\Z^2(v,\varepsilon)$ satisfying
$\pr{2}(\boldsymbol{p})>M$, any $t\in\T$ and every
$0\leq i\leq \max\{k,\ms{G}(\tau)\}$, where $\ms{G}(\tau)$ denotes the
maximum length gap of $\tau$, just defined after
\eqref{eq:bounded-gaps-def}.

Putting together \eqref{eq:z-n-j-n-good-translations},
\eqref{eq:p-n-q-n-in-Z2ve} and \eqref{eq:Fi-above-in-gaps}, and
observing $\mc{U}$ is open, we conclude there is a positive number
$\eta_0^+>0$ such that fixing any $N_0^+\in\tau$ verifying
$\pr{2}(\boldsymbol{p}_{N_0^+})>M$, where
$\boldsymbol{p}_{N_0^+}\in\Z^2$ is chosen as above, it holds
\begin{equation}
  \label{eq:Fm-iter-U-int-+}
  F^m(\{t\}\times U)\subset \{t+m\rho_1\}\times W_{m\rho_1}^+,
  \quad\forall m\geq N_0^+,\ \forall t\in B_{\eta_0}(0).
\end{equation}

Analogously, one may prove a similar statement for some
$\tilde\rho^-\in\rho(\tilde f)\cap\bb{H}_0^-$, showing that there
exist $\eta_0^->0$ and $N_0^-\in\N$ such that
\begin{equation}
  \label{eq:Fm-iter-U-int--}
  F^m(\{t\}\times U)\subset \{t+m\rho_1\}\times W_{m\rho_1}^-,\quad
  \forall m\geq N_0^-,\ \forall t\in B_{\eta_0}(0). 
\end{equation}

Putting together, \eqref{eq:U-disjoint-Lambda-v-Lambda-minus-v},
\eqref{eq:Fm-iter-U-int-+} and \eqref{eq:Fm-iter-U-int--} we can
conclude that
\begin{equation}
  \label{eq:Fm-iter-U-inter-V}
  F^m(\{t\}\times U)\cap \big(\{t+m\rho_1\}\times V\big)
  \neq\emptyset, \quad\forall m\geq N_0,\ \forall t\in B_{\eta_0}(0), 
\end{equation}
where $N_0:=\max\{N_0^+,N_0^-\}$ and
$\eta_0:=\min\{\eta_0^-,\eta_0^+\}$.

Then, invoking property \eqref{eq:iter-F} one can repeat above
argument to show that property \eqref{eq:Fm-iter-U-inter-V} in fact
holds for any $s$ in $\T$, \ie given any $s\in\T$, there exist
$\eta_s>0$ and $N_s\in\N$ such that
\begin{displaymath}
  F^m(\{t\}\times U)\cap \big(\{t+m\rho_1\}\times V\big)
  \neq\emptyset, \quad\forall m\geq N_s,\ \forall t\in B_{\eta_s}(s). 
\end{displaymath}

Finally, by compactness of $\T$ there are points
$s_1,s_2,\ldots s_r\in T$ such that
\begin{displaymath}
  \bigcup_{j=1}^rB_{\eta_{s_j}}(s_j)=\T.
\end{displaymath}
Defining $N:=\max\{N_{s_j}: 1\leq j\leq r\}$, one can easily verify
that the conclusion of
Theorem~\ref{thm:unbounded-dev-implies-spreading} holds for any
$n\geq N$.

\section{Proof of Theorem A}
\label{sec:coex-stable-sets}

In this section we finish the proof of Theorem~A. To do this let us
suppose $f$ does not exhibit uniformly bounded $v$-deviations. By
Proposition~\ref{pro:rot-set-elementary-properties},
Theorem~\ref{thm:Oxtoby-Ulam} and
Proposition~\ref{pro:powers-min-homeos-min-too} there is no loss of
generality if we assume that $f$ is a minimal symplectic homeomorphism
and admits a lift $\tilde f\in\Tsymp$ whose rotation set
$\rho(\tilde f)$ is transversal to the horizontal axis and they
intersect at the rotation vector of Lebesgue, \ie it holds
\begin{displaymath}
  \rho(\tilde f)\cap\bb{H}_0^{(0,1)}\neq\emptyset, \quad\text{and}\quad
  \rho(\tilde f)\cap\bb{H}_0^{(0,-1)}\neq\emptyset, 
\end{displaymath}
and where $\Flux(\tilde f)=(\tilde\rho_1,0)$, for some
$\tilde\rho_1\in\R$. Notice that by
Corollary~\ref{cor:minimal-homeos-rho-cap-Q2-empty}, $\tilde\rho_1$ is
irrational.

So, we can define the fiber-wise Hamiltonian skew-product
$F\colon\T\times\R^2\carr$ as in \S\ref{sec:rotat-devi-spreading}.

By analogy with \eqref{eq:Lambda-ver-definition}, for each $r\in\R$
and $u\in\{(0,1),(0,-1)\}$ we define the stable set at infinity with
respect to horizontal direction (we called it the transversal direction
in \S\ref{sec:stable-sets-infinity-trans-dir}) by
\begin{equation}
  \label{eq:Lambda-r-u-definition}
  \Lambda_{r}^u\big(t\big):=\pr{2}\bigg(\cc\Big(\{t\}\times\R^2\cap
  \I_F\big(\T^2\times \overline{\bb{H}_r^u}\big),
  \infty\Big)\bigg)\subset\R^2.  
\end{equation}

One can easily see that stable sets at infinity defined by
\eqref{eq:Lambda-ver-definition} and \eqref{eq:Lambda-r-u-definition}
are very close related and, in fact,
\begin{displaymath}
  \Lambda_r^u(t)= T_{\tilde t,0}^{-1}(\Lambda_r^u),
  \quad\forall \tilde t\in\pi^{-1}(t),\ \forall r\in\R. 
\end{displaymath}
In particular, this implies that all topological and geometric results
we proved in Theorems~\ref{thm:Lambda-vert-sets-main-properties} and
\ref{thm:Lambda-vert-oscillations}, and
Corollary~\ref{cor:Lambda-vert-oscilations} for the sets $\Lambda_r^u$
continue to hold \emph{mutatis mutandis} for the new ones
$\Lambda_r^u(t)$.

Then we have the following
\begin{proposition}
  \label{pro:Lambda-vert-Lambda-hor-disjoint}
  If $f$ does not exhibit uniformly bounded $v$-deviations, then
  \begin{displaymath}
    \Lambda_r^{(0,1)}(t)\cap\Lambda_s^v(t)=\emptyset, 
  \end{displaymath}
  for every $r,s\in\R$ and every $t\in\T$.
\end{proposition}

\begin{proof}
  Arguing by contradiction, let us suppose there exist $r,s\in\R$ and
  $t\in\T$ such that
  $C:=\Lambda_r^{(0,1)}(t)\cap\Lambda_s^v(t)\neq\emptyset$.  We claim
  that, in such a case, every connected component of $C$ is bounded in
  $\R^2$. In order to prove our claim, let us suppose there exists an
  unbounded closed connected component $\Gamma\in\pi_0(C)$.

  Since $\Gamma$ is contained in $\Lambda_r^{(0,1)}(t)$, invoking
  \eqref{eq:Lambda-v-in-no-strip} of
  Theorem~\ref{thm:Lambda-vert-sets-main-properties} we know that
  $\Gamma$ is ``vertically unbounded'', \ie it is not contained in any
  horizontal strip. On the other hand, since
  $\Gamma\subset\Lambda_s^v(t)\subset\overline{\bb{H}_s^v}$, we get
  that
  \begin{displaymath}
    \scprod{z}{v} \geq s, \quad\forall z\in\Gamma,
  \end{displaymath}
  which contradicts
  Corollary~\ref{cor:Lambda-vert-unbound-cc-oscilations}.

  So, every connected component of $C$ is bounded in $\R^2$. Invoking
  Theorem~\ref{thm:janiszewski} and taking into account that
  $\Lambda_r^{(0,1)}(t)\cup\Lambda_s^v(t)$ is unbounded, we conclude
  that $\Lambda_r^{(0,1)}(t)\cup\Lambda_s^v(t)$ should disconnect
  $\R^2$. Now let us consider two different connected components $V_1$
  and $V_2$ of
  $\R^2\setminus\Big(\Lambda_r^{(0,1)}(t)\cup\Lambda_s^v(t)\Big)$, and
  let $U\subset\R^2$ be a non-empty connected open set and
  $\varepsilon>0$ such that
  \begin{displaymath}
    U\cap\Big(\Lambda_r^u(t')\cup\Lambda_s^v(t')\Big)=\emptyset,
  \end{displaymath}
  for $t'\in\T$ satisfying $d_{\T}(0,t')<\varepsilon$.

  Then, invoking Theorem~\ref{thm:unbounded-dev-implies-spreading} we
  know that there is a natural number $N$ such that
  \begin{displaymath}
    F^n\big(B_\varepsilon(0)\times U\big)\cap \T\times V_i\neq\emptyset, 
  \end{displaymath}
  for every $i=1,2$ and every $n\geq N$. In particular, there is some
  $n_0\geq N$ and $t'\in B_\varepsilon(0)$ such that
  $R_{\rho_1}^{n_0}(t')=t$ and $F^{n_0}(\{t'\}\times U)$ intersects
  $\{t\}\times V_1$ and $\{t\}\times V_2$, and therefore, intersects
  $\{t\}\times\big(\Lambda_r^{(0,1)}(t)\cup\Lambda_s^{v}(t)\big)$ as
  well, getting a contradiction.
\end{proof}

Now, by Proposition \ref{pro:Lambda-vert-Lambda-hor-disjoint}, given
any $r\in\R$ and any $z\in\Lambda_r^v(0)$, we can define the set
\begin{displaymath}
  U_s:=\cc\Big(\bb{H}_s^{(0,1)}\setminus\Lambda_s^{(0,1)}(t),z\Big),
  \quad\forall s<\pr{2}(z). 
\end{displaymath}
Since $\Lambda_r^v(t)\subset\bb{H}_r^v$ and is connected, combining
Theorem~\ref{thm:Lambda-vert-oscillations} and
Proposition~\ref{pro:Lambda-vert-Lambda-hor-disjoint} we conclude that
\begin{equation}
  \label{eq:Lambda-v-inter-bound-s}
  \Lambda_r^v(t)\cap\partial_s^{(0,1)}(U_s)\neq\emptyset, \quad\forall 
  s<\pr{2}(z), 
\end{equation}
where $\partial_s^{(0,1)}$ denotes the boundary operator as level $s$
given by \eqref{eq:bound-at-level-r}.

However, Theorem~\ref{thm:Lambda-vert-oscillations} also implies that
there is some $s_0<\pr{2}(z)$ such that
\begin{equation}
  \label{eq:bound-at-level-s-smaller-than-r}
  \partial_{s_0}^{(0,1)}(U_{s_0})\subset\bb{H}_{-r}^{-v}.
\end{equation}
Since $\Lambda_r^v(t)\subset\bb{H}_r^v$, we see that
\eqref{eq:Lambda-v-inter-bound-s} and
\eqref{eq:bound-at-level-s-smaller-than-r} cannot simultaneously hold,
and Theorem A is proved.

\section{Proof of Theorem B}
\label{sec:proof-thm-B}

Let us suppose there exists a minimal homeomorphism
$f\in\Homeo0(\T^2)$ such that its rotation set is a non-degenerate
rational slope segment. So, if $\tilde f\colon\R^2\carr$ is a lift of
$f$, then there are $v\in\Ss^1$ and $\alpha\in\R$ such that inclusion
\eqref{eq:rot-set-in-line} holds. We know that $v$ has rational slope
and, by Corollary~\ref{cor:minimal-homeos-rho-cap-Q2-empty}, $\alpha$
is an irrational number. 

Then, by Theorem~A $f$ exhibits uniformly bounded $v$-deviations, \ie
estimate \eqref{eq:bounded-v-dev-Thm-A} holds. As a straightforward
consequence of Theorem~\ref{thm:gott-hedlund} one can show that $f$ is
a topological extension of an irrational circle rotation (see
\cite[Proposition 2.1]{JaegerLinearConsTorus} for details). But this
contradicts the following result due to Koropecki, Passaggi and
Sambarino~\cite[Theorem I]{KoropeckiPasseggiSambarino}:

\begin{theorem}
  \label{thm:Koro-Rata-Samba}
  If $f\in\Homeo0(\T^2)$ is a topological extension of an irrational
  circle rotation, then $f$ is a pseudo-rotation.
\end{theorem}

\section{Proof of Theorem C}
\label{sec:proof-thm-C}

Let $v\in\Ss^1$ and $\alpha\in\R$ such that
$\rho(\tilde f)\subset\ell_\alpha^v$. By Theorem A we know that $f$
exhibits uniformly bounded $v$-deviations. On the other hand, invoking
Theorem B we conclude $v$ has irrational slope. 

% First, let us suppose $v$ has rational slope. Then, by
% Theorem~\ref{thm:Oxtoby-Ulam} and since $f$ is minimal, there is no
% loss of generality assuming $f$ is area-preserving. On the other hand,
% by Proposition~\ref{pro:minimal-not-annular} we know that $f$ is not
% eventually annular. So, we can apply Theorem 1.1 of
% \cite{JaegerTalIrratRotFact} to conclude that $f$ is a topological
% extension of an irrational circle rotation. Hence, by Theorem I of
% \cite{KoropeckiPasseggiSambarino}, any such homeomorphism is a
% pseudo-rotation and we get a contradiction. Thus, we can assume $v$
% has irrational slope.

Then, the first step of proof consists in showing the existence of an
$f$-invariant torus pseudo-foliation (see
\S\ref{sec:pseudo-foliations} for definitions). Since $f$ is minimal,
by Theorem~\ref{thm:Oxtoby-Ulam} there is no loss of generality
assuming it is area-preserving, and by
Proposition~\ref{pro:minimal-not-annular}, $f$ is not eventually
annular. So we can invoke Theorem~\ref{thm:bounded-dev-iff-pseudo-fol}
to conclude $f$ leaves invariant a torus pseudo-foliation $\F$. Let
$\widetilde\F$ denote its lift to $\R^2$.

In order to study some topological and geometric properties of
$\widetilde\F$, let us recall some simple steps of its construction
from \cite{KocRotDevPerPFree}. Since $f$ exhibits uniformly bounded
$v$-deviations, by \cite[Corollary 3.2]{KocRotDevPerPFree}, there
exists a constant $C>0$ such that every $(r,v)$-stable set at infinity
given by \eqref{eq:Lambda-r-t-def} satisfies
\begin{displaymath}
  \bb{H}_{r+C}^v\subset\Lambda_r^v(0), \quad\forall r\in\R. 
\end{displaymath}

So, for each $r\in\R$, we define the open set
$U_r:=\cc\Big(\mathrm{int}\big(\Lambda_r^v(0)\big),
\bb{H}_{r+C}^v\Big)$; and then, we consider the function
$H\colon\R^2\to\R$ given by
\begin{equation}
  \label{eq:H-def-func-level-sets-pseudo-leaves}
  H(z):=\sup\{r\in\R : z\in U_r\}, \quad\forall z\in\R^2.
\end{equation}
In the proof of \cite[Theorem 5.5]{KocRotDevPerPFree}, we showed that
\begin{equation}
  \label{eq:H-almost-cocycle-property}
  H\big(\tilde f(z)\big) = H(z) + \alpha, \quad\forall z\in\R^2, 
\end{equation}
and then we defined the \emph{pseudo-leaves} (\ie the atoms of the
partition $\widetilde\F$) by
\begin{displaymath}
  \widetilde\F_z:=H^{-1}\big(H(z)\big), \quad\forall z\in\R^2.
\end{displaymath}

In general the function $H$ is just semi-continuous, but under our
minimality assumption, we will show it is indeed continuous.  In fact,
let $\phi\colon\T^2\to\R$ be given by
\begin{equation}
  \label{eq:phi-function-def}
  \phi(z):=\scprod{\Delta_{\tilde f}(z)}{v}-\alpha,\quad\forall
  z\in\T^2.
\end{equation}
Since $f$ exhibits uniformly bounded $v$-deviations, invoking
Theorem~\ref{thm:gott-hedlund} we know there is $u\in C^0(\T^2,\R)$
satisfying
\begin{equation}
  \label{eq:phi-cohomological-eq}
  \phi=u\circ f - u.
\end{equation}

However, putting together \eqref{eq:Lambda-r-t-def},
\eqref{eq:H-def-func-level-sets-pseudo-leaves} and
\eqref{eq:phi-function-def} one can show that the function
$u'\colon\R^2\to\R$ given by
\begin{displaymath}
  u'(z):=\scprod{z}{v}-H(z),\quad\forall z\in\R^2,
\end{displaymath}
is semi-continuous, $\Z^2$-periodic and satisfies $\phi=u'\circ f-u'$,
as well. By minimality and classical arguments on semi-continuous
functions, we have that $u-u'$ is a constant, and thus, $u'$ is
continuous. Consequently, function $H$ given by
\eqref{eq:H-def-func-level-sets-pseudo-leaves} is continuous as well.

So, for simplicity from now on we can assume that functions $H$ and
$u$ given by \eqref{eq:H-def-func-level-sets-pseudo-leaves} and
\eqref{eq:phi-cohomological-eq}, respectively, satisfy
\begin{equation}
  \label{eq:H-u-relation}
  H(z)=\scprod{z}{v}-u(z), \quad\forall z\in\R^2,
\end{equation}
where $u\in C^0(\T^2,\R)$.

In order to complete the proof of Theorem C, let $U,V\subset\T^2$ be
two non-empty open subsets. We want to show there exists $N\in\N$ such
that $f^n(U)\cap V\neq\emptyset$, for all $n\geq N$. Without loss of
generality we can assume both of them are connected and
inessential. Let $\tilde U,\tilde V\subset\R^2$ be two connected
components of $\pi^{-1}(U)$ and $\pi^{-1}(V)$, respectively.

Since pseudo-leaves of $\widetilde{\F}$ have empty interior, there
exist two points $z_0,z_1\in\tilde V$ such that $H(z_0)<H(z_1)$; let
us write $\delta:= 1/2\big(H(z_1)-H(z_0)\big)$. Notice that for every
$z\in\R^2$ satisfying $H(z_0)<H(z)<H(z_1)$, the corresponding
pseudo-leaves $\widetilde\F_z$ separates the points $z_1$ and $z_2$ in
$\R^2$, \ie the connected components
$\cc\big(\R^2\setminus\widetilde\F_z, z_0\big)$ and
$\cc\big(\R^2\setminus\widetilde\F_z, z_1\big)$ are different. In
particular, the pseudo-leaves $\widetilde\F_z$ must intersect the
connected set $\tilde V$.

On the other hand, by compactness, there exists a finite set
$\{\boldsymbol{p}_1,\boldsymbol{p}_2,\ldots,
\boldsymbol{p}_k\}\subset\Z^2$ satisfying the following property: for
every $z\in [0,1]^2$, there is $1\leq n_z\leq k$ such that the
pseudo-leaf $H^{-1}\big(r\big)$ intersects
$T_{\boldsymbol{p}_{n_z}}(\tilde V)$, for every
$r\in (H(z)-\delta,H(z)+\delta)$, and moreover, it separates the
points $T_{\boldsymbol{p}_{n_r}}(z_0)$ and
$T_{\boldsymbol{p}_{n_r}}(z_1)$ in $\R^2$.

Since every set of the form $H^{-1}(r-\delta,r+\delta)$ is arc-wise
connected, by compactness there exists a real number $M>0$ such that
for every $z\in [0,1]^2$ and any continuous arc
$\gamma\colon [0,1]\to\R^2$ such that
\begin{equation}
  \label{eq:curve-in-H-tube-condition}
  \gamma(t)\in H^{-1}\big(H(z)-\delta,H(z)+\delta\big),
  \quad\forall t\in [0,1],
\end{equation}
and
\begin{equation}
  \label{eq:tall-curve-condition}
  \min_{t\in [0,1]}\pr{2}\circ\gamma(t) < -M < M < \max_{t\in
    [0,1]}\pr{2}\circ\gamma(t),
\end{equation}
there exists $t_z\in [0,1]$ such that $\gamma(t_z)\in
T_{\boldsymbol{p}_{n_z}}(\tilde V)$. 

Now, let $z$ be an arbitrary point of $\tilde U$, and write $r:=H(z)$
and
\begin{equation}
  \label{eq:tilde-U-z-def}
  \tilde U_z:=\cc\big(\tilde U\cap H^{-1}(r-\delta,r+\delta),z\big).  
\end{equation}

Since we are assuming the rotation set $\rho(\tilde f)$ is an
irrational slope segment, there exists two $f$-invariant Borel
probability measures $\mu$ and $\nu$ such that
$\pr{2}\big(\rho_\mu(\tilde f)\big)\neq \pr{2}\big(\rho_\nu(\tilde
f)\big)$. By minimality of $f$, $\mu$ and $\nu$ have total support,
and this means there exist two points $z_\mu,z_\nu\in U_z$ such that
\begin{equation}
  \label{eq:zmu-znu-diff-rot-vect}
  \frac{\tilde f^n(z_\mu)-z_\mu}{n}\to\rho_{\mu}(\tilde f),
  \quad\text{and}\quad \frac{\tilde
    f^n(z_\nu)-z_\nu}{n}\to\rho_{\nu}(\tilde f), 
\end{equation}
as $n\to\infty$.

By \eqref{eq:tilde-U-z-def}, $\tilde U_z$ is arc-wise connected. So,
there is a continuous curve $\eta\in [0,1]\to\tilde U_z$ such that
$\eta(0)=z_\mu$ and $\eta(1)=z_\nu$. Then, by
\eqref{eq:zmu-znu-diff-rot-vect} we conclude there exists a natural
number $N_M$ such that
\begin{equation}
  \label{eq:gamma-f-iterate-event-large}
  \abs{\pr{2}\circ\tilde f^n\big(\gamma(0)\big)-\pr{2}\circ\tilde 
    f^n\big(\gamma(1)\big)} > 2M, \quad\forall n\geq N_M,
\end{equation}
where $M$ is the positive real constant invoked in
\eqref{eq:tall-curve-condition}.

Observing that, by \eqref{eq:H-almost-cocycle-property} and
\eqref{eq:tilde-U-z-def}, one has
\begin{displaymath}
  \tilde f^n\big(\tilde U_z\big)=\cc\Big(\tilde f^n(\tilde U)\cap 
  H^{-1}(r+n\alpha -\delta,r+n\alpha+\delta), \tilde
  f^n(z)\Big),\quad\forall n\in\Z. 
\end{displaymath}
This implies that, putting together
\eqref{eq:curve-in-H-tube-condition}, \eqref{eq:tall-curve-condition}
and \eqref{eq:gamma-f-iterate-event-large} one can conclude that, for
each $n\geq N_M$ there exists $\boldsymbol{q}_n\in\Z^2$ and
$i_n\in\{1,\ldots,k\}$ such that
\begin{displaymath}
  T_{\boldsymbol{q}_n}\circ \tilde f\circ\gamma [0,1]\cap
  T_{\boldsymbol{p}_{i_n}}\big(\tilde V\big)\neq\emptyset,
  \quad\forall n\geq N_M,
\end{displaymath}
and this proves $f$ is topologically mixing, because the image is of
$\gamma$ is completely contained in $\tilde U$.

In order to show that $\ell_\alpha^v\cap\Q^2=\emptyset$, we invoke a
recent result of Beguin, Crovisier and Le Roux
\cite{BegCrovLeRProjectiveTrans} which extends a previous one due to
Kwapisz \cite{KwapiszAPriori}. In fact, if there is any rational point
on $\ell_\alpha^v$, one can show that $f$ is flow equivalent to a
homeomorphism $g\in\Homeo0(\T^2)$ such that $\rho(\tilde g)$ is a
vertical line segment, for any lift $\tilde g\colon\R^2\carr$ of $g$
(see \cite[\S\S 2,3]{KwapiszAPriori} and
\cite{BegCrovLeRProjectiveTrans} for details). Since minimality is
preserved under flow equivalence, we conclude that $g$ is a minimal
homeomorphism and its rotation set is a non-degenerate rational slope
segment, contradicting Theorem B.

\bibliographystyle{amsalpha} \bibliography{references}

\providecommand{\bysame}{\leavevmode\hbox to3em{\hrulefill}\thinspace}
\providecommand{\MR}{\relax\ifhmode\unskip\space\fi MR }
\providecommand{\MRhref}[2]{%
  \href{http://www.ams.org/mathscinet-getitem?mr=#1}{#2}
}
\providecommand{\href}[2]{#2}
\begin{thebibliography}{BCLR17}

\bibitem[Atk76]{AtkinsonRecurrCocycles}
G.~Atkinson, \emph{Recurrence of co-cycles and random walks}, Journal of the
  London Mathematical Society \textbf{2} (1976), no.~3, 486--488.

\bibitem[BCLR17]{BegCrovLeRProjectiveTrans}
F.~B{\'e}guin, S.~Crovisier, and F.~Le~Roux, \emph{Projective transformations
  of rotation sets}, arXiv preprint 1711.04020, 2017.

\bibitem[D{\'{a}}v16]{DavalosAnnularMapsTorus}
P.~D{\'{a}}valos, \emph{On annular maps of the torus and sublinear diffusion},
  J. Inst. Math. Jussieu (2016), 1--66.

\bibitem[FM90]{FranksMisiure}
J.~Franks and M.~Misiurewicz, \emph{Rotation sets of toral flows}, Proc. Amer.
  Math. Soc. \textbf{109} (1990), no.~1, 243--249.

\bibitem[Fra95]{FranksRotSetPerPoint}
J.~Franks, \emph{The rotation set and periodic points for torus homeomorphims},
  Dynamical systems and chaos, Vol. 1 (Hachioji, 1994), World Sci. Publ., River
  Edge, NJ, 1995.

\bibitem[GH55]{GottschalkHedlundTopDyn}
W.~Gottschalk and G.~Hedlund, \emph{Topological dynamics}, vol.~36, American
  Mathematical Soc., 1955.

\bibitem[GKT14]{GuelKorTalAnnAreaPresToral}
N.~Guelman, A.~Koropecki, and F.~Tal, \emph{A characterization of annularity
  for area-preserving toral homeomorphisms}, Mathematische Zeitschrift
  \textbf{276} (2014), no.~3-4, 673--689.

\bibitem[Han89]{HandelPerPointFree}
M.~Handel, \emph{Periodic point free homeomorphism of $\mathbb{T}^2$}, Proc.
  Amer. Math. Soc. \textbf{107} (1989), no.~2, 511--515.

\bibitem[J\"09]{JaegerConcBoundMeanMot}
T.~J\"ager, \emph{The concept of bounded mean motion for toral homeomorphisms},
  Dyn. Syst. \textbf{24} (2009), no.~3, 277--297.

\bibitem[J{\"a}g09]{JaegerLinearConsTorus}
T.~J{\"a}ger, \emph{Linearization of conservative toral homeomorphisms},
  Inventiones Mathematicae \textbf{176} (2009), no.~3, 601--616.

\bibitem[JT17]{JaegerTalIrratRotFact}
T.~J\"ager and F.~Tal, \emph{Irrational rotation factors for conservative torus
  homeomorphisms}, Ergodic Theory and Dynamical Systems \textbf{37} (2017),
  no.~5, 1537--1546.

\bibitem[KK09]{KocKorFoliations}
A.~Kocsard and A.~Koropecki, \emph{A mixing-like property and inexistence of
  invariant foliations for minimal diffeomorphisms of the 2-torus}, Proceedings
  of the American Mathematical Society \textbf{137} (2009), no.~10, 3379--3386.

\bibitem[KPR18]{KocRotDevPerPFree}
A.~Kocsard and F.~Pereira-Rodrigues, \emph{Rotational deviations and invariant
  pseudo-foliations for periodic point free torus homeomorphisms}, Math. Z.
  \textbf{290} (2018), no.~3-4, 1223--1247.

\bibitem[KPS16]{KoropeckiPasseggiSambarino}
A.~Koropecki, A.~Passeggi, and M.~Sambarino, \emph{The {Franks-Misiurewicz}
  conjecture for extensions of irrational rotations}, preprint
  arXiv:1611.05498, 2016.

\bibitem[KT14a]{KoropeckiTalAreaPrsIrrotDiff}
A.~Koropecki and F.~Tal, \emph{Area-preserving irrotational diffeomorphisms of
  the torus with sublinear diffusion}, Proceedings of the American Mathematical
  Society \textbf{142} (2014), no.~10, 3483--3490.

\bibitem[KT14b]{KoroTalStricTor}
\bysame, \emph{Strictly toral dynamics}, Inventiones Mathematicae \textbf{196}
  (2014), no.~2, 339--381.

\bibitem[Kur68]{KuratowskiTopologyVol2}
K.~Kuratowski, \emph{Topology}, vol.~2, Academic Press, 1968.

\bibitem[Kwa02]{KwapiszAPriori}
J.~Kwapisz, \emph{A priori degeneracy of one-dimensional rotation sets for
  periodic point free torus maps}, Transactions of the American Mathematical
  Society \textbf{354} (2002), no.~7, 2865--2895.

\bibitem[LC99]{LeCalvezPropDynIndMaj1}
P.~Le~Calvez, \emph{Une {propri\'et\'e} dynamique des hom\'eomorphismes du plan
  au voisinage d{'}un point fixe d{'}indice $>1$}, Topology \textbf{38} (1999),
  no.~1, 23--35.

\bibitem[LC03]{LeCalvezDynamiqueHomeosDuPlanVois}
\bysame, \emph{Dynamique des hom{\'{e}}omorphismes du plan au voisinage d{'}un
  point fixe}, Annales scientifiques de l{'}Ecole normale sup{\'{e}}rieure
  \textbf{36} (2003), no.~1, 139--171.

\bibitem[LCT18]{LeCalvezTalForcingTheory}
P.~Le~Calvez and F.~A. Tal, \emph{Forcing theory for transverse trajectories of
  surface homeomorphisms}, Invent. Math. \textbf{212} (2018), no.~2, 619--729.

\bibitem[LCY]{LeCalvezYoccozSuitesDesIndices}
P.~Le~Calvez and J.-C. Yoccoz, \emph{Suite des indices de lefschetz des
  {it\'{e}r\'{e}s} pour un domaine de {Jordan} qui est un bloc isolant},
  Unpublsihed manuscript.

\bibitem[LR04]{LeRouxLeauFatou}
F.~Le~Roux, \emph{Hom{\'{e}}omorphismes de surfaces: th{\'{e}}or{\`{e}}mes de
  la fleur de leau-fatou et de la vari{\'{e}}t{\'{e}} stable}, Ast{\'{e}}risque
  \textbf{292} (2004), --210.

\bibitem[MZ89]{MisiurewiczZiemian}
M.~Misiurewicz and K.~Ziemian, \emph{Rotation sets for maps of tori}, Journal
  of the London Mathematical Society \textbf{2} (1989), no.~3, 490--506.

\bibitem[Nav11]{NavasGroupsCircleDiff}
A.~Navas, \emph{Groups of circle diffeomorphisms}, University of Chicago Press,
  2011.

\bibitem[OU41]{OxtobyUlam}
J.~Oxtoby and S.~Ulam, \emph{Measure-preserving homeomorphims and metrical
  transitivity}, Annals of Mathematics \textbf{42} (1941), no.~4, 874--920.

\bibitem[Poi80]{Poincare1880memoire}
H.~Poincar{\'{e}}, \emph{M{\'{e}}moire sur les courbes d{\'{e}}finies par les
  {\'{e}}quations diff{\'{e}}rentielles {I{--}VI}, oeuvre {I}},
  Gauthier-Villar: Paris (1880), 375--422.

\bibitem[Sch06]{SchmidtRecCocStatRandWalk}
K.~Schmidt, \emph{Recurrence of cocycles and stationary random walks}, Lecture
  Notes-Monograph Series (2006), 78--84.

\end{thebibliography}

\end{document}